\documentclass[11pt, 
]{amsart}
\usepackage{}
\usepackage{amssymb}

\usepackage{amsmath}
\usepackage{txfonts}

\input xy
\xyoption{all}

\usepackage{amsfonts}
\usepackage{mathrsfs}

\usepackage{latexsym}
\usepackage{graphicx}
\usepackage{amscd,amssymb,amsmath,amsbsy,amsthm}
\usepackage[all]{xy}
\usepackage[colorlinks,plainpages,backref,urlcolor=blue]{hyperref}
\usepackage{verbatim}
\usepackage{mathtools}

\usepackage{tikz}
\usetikzlibrary{arrows,calc,decorations.markings,decorations.pathreplacing}
\usetikzlibrary{shapes.geometric}
\tikzset{
>=stealth',
help lines/.style={dashed, thick},
axis/.style={<->},
important line/.style={thick},
connection/.style={thick, dotted},
}

\long\def\pzjrem#1{}
\long\def\yaping#1{}
\renewcommand\ss{\scriptstyle}
\newcommand{\bra}[1]{\left\langle #1\right|}
\newcommand{\ket}[1]{\left|#1\right\rangle}

\newcommand{\nc}{\newcommand}
\nc{\rnc}{\renewcommand}
\nc{\bb}[1]{{\mathbb #1}}
\nc{\bbA}{\bb{A}}\nc{\bbB}{\bb{B}}\nc{\bbC}{\bb{C}}\nc{\bbD}{\bb{D}}
\nc{\bbE}{\bb{E}}\nc{\bbF}{\bb{F}}\nc{\bbG}{\bb{G}}\nc{\bbH}{\bb{H}}
\nc{\bbI}{\bb{I}}\nc{\bbJ}{\bb{J}}\nc{\bbK}{\bb{K}}\nc{\bbL}{\bb{L}}
\nc{\bbM}{\bb{M}}\nc{\bbN}{\bb{N}}\nc{\bbO}{\bb{O}}\nc{\bbP}{\bb{P}}
\nc{\bbQ}{\bb{Q}}\nc{\bbR}{\bb{R}}\nc{\bbS}{\bb{S}}\nc{\bbT}{\bb{T}}
\nc{\bbU}{\bb{U}}\nc{\bbV}{\bb{V}}\nc{\bbW}{\bb{W}}\nc{\bbX}{\bb{X}}
\nc{\bbY}{\bb{Y}}\nc{\bbZ}{\bb{Z}}
\nc{\mbf}[1]{{\mathbf #1}}
\nc{\bfA}{\mbf{A}}\nc{\bfB}{\mbf{B}}\nc{\bfC}{\mbf{C}}\nc{\bfD}{\mbf{D}}
\nc{\bfE}{\mbf{E}}\nc{\bfF}{\mbf{F}}\nc{\bfG}{\mbf{G}}\nc{\bfH}{\mbf{H}}
\nc{\bfI}{\mbf{I}}\nc{\bfJ}{\mbf{J}}\nc{\bfK}{\mbf{K}}\nc{\bfL}{\mbf{L}}
\nc{\bfM}{\mbf{M}}\nc{\bfN}{\mbf{N}}\nc{\bfO}{\mbf{O}}\nc{\bfP}{\mbf{P}}
\nc{\bfQ}{\mbf{Q}}\nc{\bfR}{\mbf{R}}\nc{\bfS}{\mbf{S}}\nc{\bfT}{\mbf{T}}
\nc{\bfU}{\mbf{U}}\nc{\bfV}{\mbf{V}}\nc{\bfW}{\mbf{W}}\nc{\bfX}{\mbf{X}}
\nc{\bfY}{\mbf{Y}}\nc{\bfZ}{\mbf{Z}}
\nc{\bfa}{\mbf{a}}\nc{\bfb}{\mbf{b}}\nc{\bfc}{\mbf{c}}\nc{\bfd}{\mbf{d}}
\nc{\bfe}{\mbf{e}}\nc{\bff}{\mbf{f}}\nc{\bfg}{\mbf{g}}\nc{\bfh}{\mbf{h}}
\nc{\bfi}{\mbf{i}}\nc{\bfj}{\mbf{j}}\nc{\bfk}{\mbf{k}}\nc{\bfl}{\mbf{l}}
\nc{\bfm}{\mbf{m}}\nc{\bfn}{\mbf{n}}\nc{\bfo}{\mbf{o}}\nc{\bfp}{\mbf{p}}
\nc{\bfq}{\mbf{q}}\nc{\bfr}{\mbf{r}}\nc{\bfs}{\mbf{s}}\nc{\bft}{\mbf{t}}
\nc{\bfu}{\mbf{u}}\nc{\bfv}{\mbf{v}}\nc{\bfw}{\mbf{w}}\nc{\bfx}{\mbf{x}}
\nc{\bfy}{\mbf{y}}\nc{\bfz}{\mbf{z}}

\nc{\mcal}[1]{{\mathcal #1}}
\nc{\calA}{\mcal{A}}\nc{\calB}{\mcal{B}}\nc{\calC}{\mcal{C}}\nc{\calD}{\mcal{D}}
\nc{\calE}{\mcal{E}} \nc{\calF}{\mcal{F}}\nc{\calG}{\mcal{G}}\nc{\calH}{\mcal{H}}
\nc{\calI}{\mcal{I}}\nc{\calJ}{\mcal{J}}\nc{\calK}{\mcal{K}}\nc{\calL}{\mcal{L}}
\nc{\calM}{\mcal{M}}\nc{\calN}{\mcal{N}}\nc{\calO}{\mcal{O}}\nc{\calP}{\mcal{P}}
\nc{\calQ}{\mcal{Q}}\nc{\calR}{\mcal{R}}\nc{\calS}{\mcal{S}}\nc{\calT}{\mcal{T}}
\nc{\calU}{\mcal{U}}\nc{\calV}{\mcal{V}}\nc{\calW}{\mcal{W}}\nc{\calX}{\mcal{X}}
\nc{\calY}{\mcal{Y}}\nc{\calZ}{\mcal{Z}}
\nc{\fA}{\frak{A}}\nc{\fB}{\frak{B}}\nc{\fC}{\frak{C}} \nc{\fD}{\frak{D}}
\nc{\fE}{\frak{E}}\nc{\fF}{\frak{F}}\nc{\fG}{\frak{G}}\nc{\fH}{\frak{H}}
\nc{\fI}{\frak{I}}\nc{\fJ}{\frak{J}}\nc{\fK}{\frak{K}}\nc{\fL}{\frak{L}}
\nc{\fM}{\frak{M}}\nc{\fN}{\frak{N}}\nc{\fO}{\frak{O}}\nc{\fP}{\frak{P}}
\nc{\fQ}{\frak{Q}}\nc{\fR}{\frak{R}}\nc{\fS}{\frak{S}}\nc{\fT}{\frak{T}}
\nc{\fU}{\frak{U}}\nc{\fV}{\frak{V}}\nc{\fW}{\frak{W}}\nc{\fX}{\frak{X}}
\nc{\fY}{\frak{Y}}\nc{\fZ}{\frak{Z}}
\nc{\fa}{\frak{a}}\nc{\fb}{\frak{b}}\nc{\fc}{\frak{c}} \nc{\fd}{\frak{d}}
\nc{\fe}{\frak{e}}\nc{\fFf}{\frak{f}}\nc{\fg}{\frak{g}}\nc{\fh}{\frak{h}}
\nc{\fri}{\frak{i}}\nc{\fj}{\frak{j}}\nc{\fk}{\frak{k}}\nc{\fl}{\frak{l}}
\nc{\fm}{\frak{m}}\nc{\fn}{\frak{n}}\nc{\fo}{\frak{o}}\nc{\fp}{\frak{p}}
\nc{\fq}{\frak{q}}\nc{\fr}{\frak{r}}\nc{\fs}{\frak{s}}\nc{\ft}{\frak{t}}
\nc{\fu}{\frak{u}}\nc{\fv}{\frak{v}}\nc{\fw}{\frak{w}}\nc{\fx}{\frak{x}}
\nc{\fy}{\frak{y}}\nc{\fz}{\frak{z}}

\newtheorem{theorem}{Theorem}[section]
\newtheorem{lemma}[theorem]{Lemma}
\newtheorem{corollary}[theorem]{Corollary}
\newtheorem{prop}[theorem]{Proposition}

\theoremstyle{definition}

\newtheorem{example}[theorem]{Example}
\newtheorem{remark}[theorem]{Remark}

\newtheorem{thm}{Theorem}

\DeclareMathOperator{\Id}{Id}

\DeclareMathOperator{\im}{im} 
 \DeclareMathOperator{\id}{id}
\DeclareMathOperator{\Image}{Im}

 \DeclareMathOperator{\GL}{GL}

\DeclareMathOperator{\Hom}{{Hom}}

\DeclareMathOperator{\sHom}{{\mathscr{H}\!\mathit{om}}}

\DeclareMathOperator{\fac}{{fac}}

 \DeclareMathOperator{\tr}{tr}

 \DeclareMathOperator{\End}{End}

\DeclareMathOperator{\SH}{SH}

\DeclareMathOperator{\ev}{ev}

\DeclareMathOperator{\Gr}{Gr}

\newcommand{\Omit}[1]{}

\newcommand{\CC}{\mathbb{C}}

\DeclareMathOperator{\Res}{Res}

\newcommand{\surj}{\twoheadrightarrow}
\newcommand{\inj}{\hookrightarrow}

\newcommand{\pt}{\text{pt}}

\newcommand{\Z}{\bbZ}
\newcommand{\C}{\bbC}

\newcommand{\N}{\bbN}

\DeclareMathOperator{\Sh}{Sh}

\nc{\tQ}{\tilde{Q}}
\nc{\ep}{\epsilon}
\nc{\tPi}{\tilde{\Pi}}

\nc{\gmod}{\text{-$\mathrm{gmod}$}}

\nc{\tX}{\tilde{X}}
\nc{\tfM}{\tilde{\fM}}
\nc{\tfL}{\tilde{\fL}}
\nc{\tLambda}{\tilde{\Lambda}}

\nc{\hGr}{\widehat{\Gr}{}}

\usepackage{enumerate} 

\newcount\cols
{\catcode`,=\active\catcode`|=\active
 \gdef\Young(#1){\hbox{$\vcenter
 {\mathcode`,="8000\mathcode`|="8000
  \def,{\global\advance\cols by 1 &}%
  \def|{\cr
        \multispan{\the\cols}\hrulefill\cr
        &\global\cols=2 }%
  \offinterlineskip\everycr{}\tabskip=0pt
  \dimen0=\ht\strutbox \advance\dimen0 by \dp\strutbox
  \halign
   {\vrule height \ht\strutbox depth \dp\strutbox##
    &&\hbox to \dimen0{\hss$##$\hss}\vrule\cr
    \noalign{\hrule}&\global\cols=2 #1\crcr
    \multispan{\the\cols}\hrulefill\cr%
   }
 }$}}
}

\author[Y.~Yang]{Yaping~Yang}
\address{The University of Melbourne,
	School of Mathematics and Statistics,
	813 Swanston Street, Parkville VIC 3010,
	Australia}
\email{yaping.yang1@unimelb.edu.au}

\author[P.~Zinn-Justin]{Paul~Zinn-Justin}
\address{The University of Melbourne,
	School of Mathematics and Statistics,
	813 Swanston Street, Parkville VIC 3010,
	Australia}
\email{pzinn@unimelb.edu.au}

\title[Higher spin representations of the Yangian of $\mathfrak{sl}_2$ and R-matrices]{Higher spin representations of the Yangian of $\mathfrak{sl}_2$ and R-matrices}

\date{\today}

\begin{document}

\begin{abstract}
We study higher spin (pure and mixed spin) representations of the Yangian of $\mathfrak{sl}_2$. We provide a geometric realization in terms of the critical cohomology of representations of the quiver with potential of Bykov and Zinn-Justin \cite{BZJ20}. When the framing dimension is 1, it recovers the evaluation pullback of the $\ell+1$-dimensional irreducible representation of $\mathfrak{sl}_2$. 
We introduce the lattice model and prove that its partition function coincides with the weight function constructed using the framed shuffle formula. The latter follows the approach of Rim\'anyi, Tarasov and Varchenko \cite{RTV15}. 
\end{abstract}
\maketitle
\tableofcontents
\section{Introduction}
In the paper, we consider the following triple quiver $\tilde{Q}$: with a gauge vertex $\bullet$ and with $w$ framing vertices $\square_1, \ldots, \square_w$, where $w\in \mathbb{N}$. The set of arrows of $\tilde{Q}$ is $\{\epsilon, a_1, \ldots, a_{w}, a_1^*, \ldots, a_{w}^*\}$, where $\epsilon$ is a loop on $\bullet$, $a_i$ is an arrow from $\bullet$ to $\square_i$, and $a^*_i$ is the reversed arrow from $\square_i$ to $\bullet$, for $i=1, \ldots, w$.
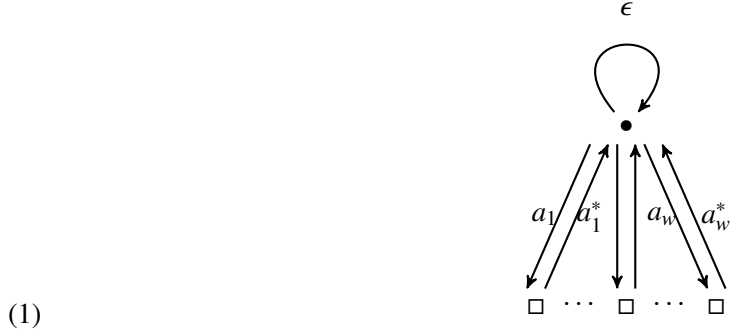
\begin{equation}\label{twQ}
\begin{tikzpicture}[scale=1.2]
         \node at (-2.2, 0) {$\bullet$};   
\node at (-2.2, -2) {$\square$}; 

\node at (-2.2-0.5, -2) {$\cdots$}; 
\node at (-2.2+0.5, -2) {$\cdots$}; 
   \draw[->, thick] (-2.2+0.1, -1.8) -- (-2.2+0.1, -0.2) ;
  \draw[<-, thick] (-2.2-0.1, -1.8) -- (-2.2-0.1, -0.2) ;
\node at (-1.2, -2) {$\square$}; 
     \draw[->, thick] (-1.2+0.1, -1.8) -- (-2.2+0.1+0.3, -0.2) ;
  \draw[<-, thick] (-1.2-0.1, -1.8) -- (-2.2-0.1+0.3, -0.2) ;
  
\node at (-3.2, -2) {$\square$};
     \draw[->, thick] (-3.2+0.1, -1.8) -- (-2.2+0.1-0.3, -0.2) ;
  \draw[<-, thick] (-3.2-0.1, -1.8) -- (-2.2-0.1-0.3, -0.2) ;
 \node at (-3.2+0.2+0.4, -1) {$a^*_1$}; 
  \node at (-3.2+0.1, -1) {$a_1$};
   \node at (-1.2, -1) {$a^*_w$}; 
  \node at (-1.2-0.6, -1) {$a_w$};
    \draw[->, thick, shorten <=7pt, shorten >=7pt] ($(-2.2,0)$)
    .. controls +(90+40:1.5) and +(90-40:1.5) .. ($(-2.2,0)$);
 \node at (-2.2, 1.3) {$\epsilon$}; 
    \end{tikzpicture}
\end{equation}

For $w$-tuple of positive integers $(\ell_1, \ell_2, \cdots, \ell_w)\in \N^w$
, let 
\begin{equation}\label{eq:potW}
\bfw=\epsilon^{\ell_1} a_1^*a_1+\epsilon^{\ell_2} a_2^*a_2+\cdots +\epsilon^{\ell_w}a_w^*a_w
\end{equation}
be the potential of the quiver $\tilde{Q}$. 
We are interested in this particular quiver with potential for a few reasons. 
\begin{enumerate}
\item The critical cohomology, also known as the compact supported cohomology with valued in a perverse sheaf, of the stable representation variety of $\tilde{Q}$ is a generalization of the cohomology of the cotangent bundle of Grassmannian. The latter corresponds to the case when $\ell_1=\cdots=\ell_{w}=1$.  
\item 
This quiver with potential gives the higher spin representations of the Yangian of $\mathfrak{sl}_2$. When $w=1$, one obtains the evaluation pullback of the $\ell+1$-dimensional irreducible representation of $\mathfrak{sl}_2$ or the Kirillov-Reshetikhin module of the Yangian of $\mathfrak{sl}_2$. 
\item For quantum loop algebras of non-simply laced type, their representations and Kirillov--Reshetikhin modules have been constructed geometrically by Varagnolo and Vasserot \cite{VV23}. Our quiver $\tilde{Q}$ and potential $\bfw$ are related to the rank $1$ reduction of the quiver with potential from \cite[Section 4.2]{VV23}. See also \cite{YZ22} for the geometric realization of the Yangian of non-simply laced type in terms of quiver with potentials.
\item The stable envelope constructed by Maulik and Okounkov in \cite{MO} requires a symplectic variety together with a torus action preserving
the symplectic form. However, the stable representation variety for $(\tilde{Q}, \bold{w})$ is not symplectic. See \cite{RSYZ1, RSYZ2} for the 3-dimensional Calabi-Yau perspective. In this paper, we introduce the lattice model and its partition function. From the diagrammatic calculus point of view, the R-matrix can be straightforwardly obtained from the partition functions. We show that the partition functions coincide with the weight functions defined using the framed shuffle formula \cite{B23, YZ17}. The approach of weight functions has been exploited in the symplectic case: see \cite{RTV15} for the cotangent bundle of partial flag varieties, and \cite{B23} for Nakajima quiver varieties. 
\item The setup in \cite{BZJ20} concerns the pure spin representations, while the present paper deals with mixed spins as well. 
For a general $w$-tuple of positive integers $(\ell_1, \ldots, \ell_w)\in \N^w$, we refer to the setup of the quiver $\tilde{Q}$ \eqref{twQ} with potential $\bold{w}$ \eqref{eq:potW} as mixed spin. In this case, we take the framing vector to be $\vec{1}_w=(1, 1, \ldots, 1)\in \N^w$. When $\ell_1=\ell_2=\cdots=\ell_w=\ell\in \N$, the potential becomes $\bold{w}=w\epsilon^{\ell}(\sum_{i=1}^w a_i^*a_i)$ and we obtain the pure spin case. Alternatively, for the pure spin, we could consider the following quiver with potential, where the framing dimension is $w\in \N$. 
\begin{equation*}
\begin{tikzpicture}[scale=0.8]
         \node at (-2.2, 0) {$\bullet$};   	 	\node at (-2.2+0.5, 0) {$v$};  
\node at (-2.2, -2) {$\square$}; 
\node at (-2.2+0.5, -2) {$w$}; 
   \draw[->, thick] (-2.2+0.1, -1.8) -- (-2.2+0.1, -0.2) ;
  \draw[<-, thick] (-2.2-0.1, -1.8) -- (-2.2-0.1, -0.2) ;
 \node at (-2.2+0.1+0.4, -1) {$a^*$}; 
  \node at (-2.2+0.1-0.4, -1) {$a$};
    \draw[->, thick, shorten <=7pt, shorten >=7pt] ($(-2.2,0)$)
    .. controls +(90+40:1.5) and +(90-40:1.5) .. ($(-2.2,0)$);
 \node at (-2.2, 1.3) {$\epsilon$}; 
 \node at (2, -1){Potential is $\epsilon^{\ell} a^*a$};
    \end{tikzpicture}
\end{equation*}
The latter is the setup in \cite{BZJ20}.  
\end{enumerate}

\subsection{}
In the paper, we use the following grading on the arrows of the quiver $\tilde{Q}$ \eqref{twQ}: 
\begin{equation}\label{grading}
\deg(\ep)=-2, \deg(a_i)=\deg(a^*_i)=\ell_i, i=1, \ldots, w.  
\end{equation}
It is clear that the potential $\bfw$ has grading zero. This grading is compatible with the grading used in \cite{VV23} after rank $1$-reduction. 

Let $V$ be a vector space with dimension $v\in \N$. Let $W_i$ be the 1-dimensional vector space at vertex $\square_i$, $i=1, \ldots, w$.  
Denote by $(v, \vec{1}_w)\in \N\times \N^w$ the corresponding dimension vector of $\tilde{Q}$. Let $\tilde{X}(v, \vec{1}_w)$ be the representations of $\tilde{Q}$ with dimension vector $(v, \vec{1}_w)$. That is,  
\[
\tilde{X}(v, \vec{1}_w)=\{(\epsilon, a_1, \ldots, a_{w}, a^*_1, \ldots, a_{w}^*)\mid \epsilon\in \Hom(V, V), a_i\in \Hom(V, W_i), a^*\in \Hom(W_i, V), i=1, \ldots, w\}. 
\]
We impose the cyclic stability condition that: $(\epsilon, a_1, \ldots, a_{w}, a^*_1, \ldots, a_{w}^*)\in \tilde{X}(v, \vec{1}_w)$ is stable if $V$ is spanned by the image $\Image(a^*_i)$ of $a^*_i$, $i\in [1, w]$, together with the operator $\epsilon$. Denote the stable locus of $\tilde{X}(v, \vec{1}_w)$ by $\tilde{X}^{st}(v, \vec{1}_w)$.

The group $G_v:=\GL(V)$ acts on $\tilde{X}(v, \vec{1}_w)$ by changing of basis of $V$, and the grading \eqref{grading} gives a commuting $\C^*$-action on $\tilde{X}(v, \vec{1}_w)$ with $\C^*$-weights given by the degree. Denote by $T_w$ the framing torus $(\C^*)^w$, which also acts on $\tilde{X}(v, \vec{1}_w)$. 

The trace of the potential $\bold{w}$ is a $G_v\times T_w\times \C^*$-equivariant function on $\tilde{X}(v, \vec{1}_w)$. Denote by $\varphi_{\bf{w}}:=\varphi_{\bf{w}}(\C)$ the sheaf of vanishing cycles of this function, which is a perverse constructible sheaf on $\tilde{X}(v, \vec{1}_w)$. 

Let $Y$ be a complex variety. Let $D^b_{c}(Y)$ be the bounded derived category of $\C$-vector spaces on $Y$ with constructible cohomology. For any morphism $f: Y\to X$, denote by $f_{!}: D^b_{c}(Y)\to D^b_{c}(X)$ the direct image with compact support. For a smooth variety $Y$ together with a function $\bold{w}: Y\to \C$  on it, denote by $\varphi_{\bold{w}}\in D^b_{c}(Y)$ the vanishing cycle complex of $\bold{w}$. 
Let $\pi: Y\to \pt$ be the structure map. 
We consider the vector space dual of the compactly supported cohomology with coefficient in $\varphi_{\bold{w}}$, that is, 
\[
H^*_{c}(Y, \varphi_{\bold{w}})^{\vee}=\pi_{*} \varphi_{\bold{w}} D_Y(\C)=D_{\pt}\pi_{!} \varphi_{\bold{w}}(\C),
\]
where $D_Y$ is the Verdier dualizing sheaf on $Y$.  See \cite{KS, D, RSYZ2} for details.


Let $Y_{\hbar}(\mathfrak{sl}_2)$ be the Yangian associated to the Lie algebra $\mathfrak{sl}_2$. 
Let $R:=H_{\C^*\times T_w}(\pt)=\C[\hbar, z_{1}, \cdots, z_{w}]$ be the base ring and let  $K=\C(\hbar, z_{1}, \cdots, z_{w})$ be its field of fractions. 

\begin{thm}\label{thm:action}
For any $w\in \N$ and any $(\ell_1, \cdots, \ell_w)\in \N^w$, the Yangian $Y_{\hbar}(\mathfrak{sl}_2)$ acts on the cohomology 
\[
\bigoplus_{v\in \N} H^*_{c, \C^*\times T_w}(\tilde{X}(v, \vec{1}_w)^{st}/G_v, \varphi_{\bf{w}})^{\vee}\otimes_R K. 
\]
\end{thm}
The proof of Theorem  \ref{thm:action} is in Section \ref{subsec:proof of thm}. The main ingredients of the proof are the dimension reduction of the critical cohomology (recalled in Proposition \ref{prop:dimred}) and the computation of the Kac-Moody generators of $Y_{\hbar}(\mathfrak{sl}_2)$ on the torus fixed points basis. This approach has been used in \cite{RSYZ2} in the setup of the generalized Yangian action on the cohomology of the moduli
spaces of certain perverse coherent systems on a toric Calabi Yau 3-fold.

Denote by $\ev_{z}: Y_{\hbar}(\mathfrak{sl}_2)\to U(\mathfrak{sl}_2)$ the evaluation homomorphism (see \cite[Proposition 12.1.15]{CP}) of $Y_{\hbar}(\mathfrak{sl}_2)$.
Recall that $\ev_z$ is the unique algebra homomophism from $Y_{\hbar}(\mathfrak{sl}_2)$ to $U(\mathfrak{sl}_2)$ such that 
\[
\ev_z(x)=x, \,\ 
\ev_z(J(x))=z x,  \text{for all $x\in \mathfrak{sl}_2$}. 
\]
Here $x, J(x)$ are the Drinfeld generators of $Y_{\hbar}(\mathfrak{sl}_2)$. See \cite[Theorem 12.1.1]{CP} for the relations among $x$ and $J(x)$, $x\in \mathfrak{sl}_2$ and see \cite[Theorem 12.1.3]{CP} for the isomorphism between the Drinfeld presentation and the Kac-Moody presentation (see \S \ref{subsec:proof of thm}) of the Yangian. 

\begin{prop}
\label{prop:w=1}
When $w=1$, we have the following isomorphism of representations of $Y_{\hbar}(\mathfrak{sl}_2)$: 
\[
\bigoplus_{v\in \N} H^*_{c, \C^*\times \C^*_{fr}}(\tilde{X}(v, 1)^{st}/G_V, \varphi_{\bf{w}})^{\vee}\cong \ev^*_z(L(\ell)), 
\]  where
$L(\ell)\cong \C^{\ell+1}$ is the $(\ell+1)$-dimensional irreducible representation of $\mathfrak{sl}_2$ and the variable $z$ comes from the framing torus $\C^*_{fr}=T_{1}$. 
\end{prop}
The proof of Proposition \ref{prop:w=1} is in Section \ref{sub:proof of prop}.

As a consequence, the representation in Theorem \ref{thm:action} has the following decomposition. We have the isomorphism as representations of $Y_{\hbar}(\mathfrak{sl}_2)$ (Corollary \ref{cor:tensor product})
\[
\bigoplus_{v\in \N} H^*_{c, \C^*\times T_w}(\tilde{X}(v, \vec{1}_w)^{st}/G_v, \varphi_{\bf{w}})^{\vee}\otimes_R K
\cong \ev^*_{z_1}(L(\ell_1))\otimes \cdots \otimes \ev^*_{z_w}(L(\ell_w)), 
\]
where the isomorphism is given by identifying the torus fixed point basis on both sides. The tensor product is the fusion tensor product which comes from the Drinfeld coproduct of $Y_{\hbar}(\mathfrak{sl}_2)$. In the case of quantum affine algebras, the action of $U_{q}(\widehat{\mathfrak{sl}}_2)$ on the fusion tensor of Kirillov--Reshetikhin modules can be found in \cite[\S3.4]{H07}.

\subsection{}
In \S\ref{sub:shuffle algebra}, we introduce the framed shuffle algebra associated to the quiver $\tilde{Q}$. 
Denote by $\SH_{fr}$ the framed shuffle algebra. It is double graded by the dimension vectors $(v, \vec{n})\in \N\times \N^w$, where $\vec{n}=(n_1, n_2, \cdots, n_w)\in \N^w$ is $w$-tuple of integers. Explicitly, as a $\mathbb{C}[\hbar]$-module, we have 
\[\SH_{fr}=\bigoplus_{{v}\in\bbN, \vec{n}\in\bbN^w}\SH_{{v}, \vec{n}}, \,\ \text{where}\,\ \SH_{{v}, \vec{n}}:=\mathbb{C}[\hbar]\otimes \mathbb{C}
[ y_1, \cdots, y_{v} ]^{\mathcal{S}_{{v}}}\otimes \mathbb{C}
[ z_j^i ]_{j=1, \ldots, w, i=1, \ldots, n_j} ,\]
here $\mathcal{S}_{{v}}$ is the symmetric group of $v$ letters and permutes the $v$ variables $y_1, \cdots,  y_{v}$. There is an algebra structure $\star$ on $\SH_{fr}$ compactible with the two gradings. Componentwise, the multiplication
\[
\star: \SH_{v_1, \vec{n}_1}\otimes_{\C[\hbar]}  \SH_{v_2, \vec{n}_2}\to \SH_{v_1+v_2, \vec{n}_1+\vec{n}_2} 
\]
 is defined by \eqref{eq:shuffle}. The shuffle algebra structure $\star$ has a geometric interpretation. 
For each component, we identify $\SH_{v, \vec{n}}$ with 
the cohomology group $H^*_{c,G_v\times T_w\times \C^*}(\tilde{X}(v, \vec{n}))^{\vee}$. Then it follows from Proposition \ref{prop:geom} that the shuffle multiplication $\star$ on $\SH_{fr}$ 
can be identified with the Hall multiplication of 
$\oplus_{(v, \vec{n})\in \N\times \N^w}H^*_{c,G_v\times T_n\times \C^*}(\tilde{X}(v, \vec{n}))^{\vee}$ defined geometrically via pullback and pushforward along correspondences. 

Note that when $\vec{n}=0$, denote the unframed shuffle algebra by \[\SH=\bigoplus_{{v}\in\bbN}\SH_{{v}}=\bigoplus_{{v}\in\bbN}\mathbb{C}[\hbar; 
 y_1, \cdots, y_{v} ]^{\mathcal{S}_{{v}}}. 
\]
Then, there is an algebra isomorphism between $\SH$ and the positive part $Y^+_{\hbar}(\mathfrak{sl}_2)$ (see for example \cite[Section 7]{YZ1}).  

Fix the $w$-tuple $(\ell_1, \ldots, \ell_w)\in \N^w$ and 
fix the framing dimension to be $\vec{1}_w=(1, \ldots, 1)\in \N^w$. By the equivariant localization theorem, the module $$\bigoplus_{v\in \N} H^*_{c, \C^*\times T_w}(\tilde{X}(v, \vec{1}_w)^{st}/G_v, \varphi_{\bf{w}})^{\vee}\otimes_R K$$ is a free module over $H_{\C^*\times T_w}(\pt)\otimes_R K$. Each graded piece $H^*_{c, \C^*\times T_w}(\tilde{X}(v, \vec{1}_w)^{st}/G_v, \varphi_{\bf{w}})^{\vee}\otimes_R K$ has fixed point basis labelled by elements in the set (Lemma \ref{lem:fixed points})
\begin{equation*}
S(\ell_1, \ldots, \ell_w)=\{(v_1, v_2, \cdots, v_w)\in \N^w \mid 0\leq v_i\leq \ell_i, \forall i\in [1, w], \text{and} \sum_{i=1}^wv_i=v\}.
\end{equation*}

For any fixed point $\lambda=(v_1, v_2, \cdots, v_w)\in S(\ell_1, \ldots, \ell_w)$, we define the following element, referred to as weight function (see \eqref{star1} and \eqref{eq:Wsigma}) 
\[
W^{\id}_{\lambda}(\hbar; y_1, \cdots, y_{v}; z_1, \cdots, z_w):=1\star_{\lambda} 1\star_{\lambda} \cdots \star_{\lambda} 1 
\]
in $\SH_{v, \vec{1}_w}=\mathbb{C}[\hbar]\otimes \mathbb{C}
[ y_1, \cdots, y_{v} ]^{\mathcal{S}_{{v}}}\otimes \mathbb{C}
[ z_1, \cdots, z_w ]$. 
Here $1\star_{\lambda} 1\star_{\lambda} \cdots \star_{\lambda} 1$ is the image of $1\otimes 1\otimes \cdots \otimes 1$ of the following multiplication
\[
\star_{\lambda}: \SH_{v_1, e_1}\otimes \SH_{v_2, e_2}\otimes \cdots \otimes \SH_{v_w, e_w}
\to \SH_{v, \vec{1}_w}
\]
depending on $\lambda=(v_1, v_2, \ldots, v_w)$, where $e_i=(0, \ldots, 0, 1, 0, \ldots, 0)$. 
The explicit formula of the weight function is given by \eqref{Wformula}.

For any element $\sigma$ in the symmetric group $\mathcal{S}_w$, we define the weight function associated to $\sigma$ to be
\[
W^{\sigma}_{\lambda}(\hbar; y_1, \cdots, y_{v}; z_1, \cdots, z_w):=W^{\id}_{\sigma^{-1}\lambda}(\hbar; y_1, \cdots, y_{v}; z_{\sigma^{-1}(1)}, \cdots, z_{\sigma^{-1}(w)}). 
\]

Let $\lambda, \mu$ be two fixed points in $S(\ell_{\sigma(1)}, \ldots, \ell_{\sigma(w)})$. Denote by $W^{\sigma}_{\lambda}|_{\mu}$ the restriction of $W^{\sigma}_{\lambda}$ to the fixed point $\mu$. By definition, $W^{\sigma}_{\lambda}|_{\mu}$ is obtained from $W^{\sigma}_{\lambda}(\hbar; y_1, \cdots, y_{v}; z_1, \cdots, z_w)$ by substituting the variables $y_1, \cdots, y_v$ according to the equivariant Chern roots of $V$ at the fixed point $\mu$ \eqref{eq:fix point wts}. In particular, $W^{\sigma}_{\lambda}|_{\mu}\in \C[\hbar, z_1, z_2, \cdots, z_w]$.  

For a general element $\sigma\in \mathcal{S}_w$, we define the total order on the set $S(\ell_{\sigma(1)}, \ldots, \ell_{\sigma(w)})$ depending on $\sigma$ by
\[
\lambda\leq _{\sigma} \mu \iff \sigma^{-1}(\lambda)\leq \sigma^{-1}(\mu), 
\]
where $\leq$ is the  lexicographical order on $\N^{w}$.

 The weight functions have the following properties, which are proved in Lemma \ref{lem:property} and Lemma \ref{lem:prop2}.

\begin{enumerate}
\item \textbf{Triangularity: } 
If $\lambda< _{\sigma}\mu$, we have
$W^{\sigma}_{\lambda}|_{\mu}=0$.
\item \textbf{The degree condition:} For any $\lambda\geq_{\sigma} \mu$, then,
\[
\deg(W^{\sigma}_{\lambda}|_{\lambda})\geq \deg (W^{\sigma}_{\lambda}|_{\mu}).
\]
\end{enumerate}
The normalization of $W^{\sigma}_{\lambda}|_{\lambda}$ can be found in  Lemma \ref{lem:property} (2).

\subsection{}
Let $\mathfrak{t}_{w}:=\text{Lie} (T_w)$ be the Lie algebra of $T_w$ and let $(\mathfrak{t}_{w})_{\mathbb{R}}=\text{Cochar}(T_w)\otimes_{\Z}\mathbb{R}$ be the real vector space. Denote by $z_1, \ldots, z_w$ the standard basis of the dual space $\mathfrak{t}_{w}^{\vee}$. 
We have the root hyperplanes $z_i-z_j=0$ of $(\mathfrak{t}_{w})_{\mathbb{R}}$, for any $1\leq i\neq j \leq w$. This gives the open chambers $\mathfrak{C}_{\sigma}=\{x\in (\mathfrak{t}_{w})_{\mathbb{R}}\mid z_{\sigma(1)}(x)> \cdots > z_{\sigma(w)}(x)\}$, such that
\[
(\mathfrak{t}_{w})_{\mathbb{R}}\setminus \bigcup_{1\leq i\neq j \leq w}\{z_i-z_j=0\}=\bigsqcup_{\sigma\in \mathcal{S}_w} \mathfrak{C}_{\sigma}. 
\]
The chambers $\{\mathfrak{C}_{\sigma}\}$ are the usual Weyl chambers, which are labeled by elements in $\mathcal{S}_w$. 
Denote by $[W^{\sigma}_{\lambda}|_{\mu}]$ the matrix whose $(\lambda, \mu)$-entry is the restriction of $W^{\sigma}_{\lambda}$ to the fixed point $\mu$. Note that, after inverting $\{z_i-z_j-a\hbar\mid 1\leq i, j\leq w, a\in \Z\}$, $[W^{\sigma}_{\lambda}|_{\mu}]$ is an invertible matrix. Let $\mathbf{R}:=H_{\C^*\times T_w}(\pt)$ be the base ring, and let $\mathbf{K}$ be the fraction field of $\mathbf{R}$. We define the $R$-matrix, as an operator in $$\bigoplus_{v\in \N} \End(H^*_{c, \C^*\times T_w}(\tilde{X}(v, \vec{1}_w)^{st}/G_v, \varphi_{\bf{w}})^{\vee}\otimes_{\bold R} {\bold K})$$ 
to be
\begin{equation}
\label{eq:defRgeneral}
R_{\sigma', \sigma}:=[W^{\sigma'}_{\lambda}|_{\mu}]^{-1} [W^{\sigma}_{\lambda}|_{\mu}]. 
\end{equation}

For $w=2$, there are only two chambers $\mathfrak{C}_{\id}$, and $\mathfrak{C}_{(21)}$, 
where $(21)\in \mathcal{S}_2$ is the non-trivial element. Set $z:=z_1-z_2$. In this case, we denote the $R$-matrix $R_{\id, (21)}$ by
\begin{equation}\label{eq:defR}
R(z):=[W^{\id}_{\lambda}|_{\mu}]^{-1}  [W^{(21)}_{\lambda}|_{\mu}] \in \End(H^*_{c,\C^*\times T_w}(\tilde{X}(v, \vec{1}_w)^{st}/G_v, \varphi_{\bf{w}})^{\vee}\otimes_{\bold R} {\bold K}). 
\end{equation}
The dimension of $H^*_{c,\C^*\times T_w}(\tilde{X}(v, \vec{1}_w)^{st}/G_v, \varphi_{\bf{w}})^{\vee}\otimes_{\bold R} {\bold K}$ over $\bold{K}$ is the cardinality $|S(\ell_1, \ldots, \ell_w; v)|$. In particular, $R(z)$ is a matrix of size $(\sum_v |S(\ell_1, \ldots, \ell_w; v)|)^2$, with entries in $\C[\hbar, z][\frac{1}{z-a\hbar }\mid a\in \Z]$. 

The explicit computations of $R$-matrices are in Section \ref{ComputeR}.

\subsection{Lattice model formulation}
It is well-known that weight functions are closely related to off-shell Bethe Ansatz \cite{AO17}, and we therefore expect a {\em lattice model}\/ formula for
our weight functions. In \S\ref{sec:lm}, we introduce such a model: it consists of lattice paths on a $v\times w$ rectangular grid where the occupancy of vertical line $j$ is bounded by $\ell_j$, and the occupancy of horizontal lines is at most $1$, with some specific boundary conditions depending of $(v_1,\ldots,v_w)$. We assign to each lattice path configuration $c$ in the set $\mathcal C_{(v_1,\ldots,v_w)}$ of allowed configurations, a weight
$F(c)=\prod_{i=1}^v \prod_{j=1}^w(\text{weight of $c(i, j)$})$,
where the weight of $c(i, j)$ is a degree $1$ polynomial in $y_i$ and $z_j$ depending on the configuration around vertex $(i,j)$.

\begin{thm}(Theorem \ref{thm:tildeW})
\label{thmB}
We have the following equality
\[
\tilde W_{(v_1,\ldots,v_w)} :=
\sum_{c\in C_{(v_1,\ldots,v_w)}} F(c) =
\prod_{j=1}^w (-1)^{v_j(v+w-j-\sum_{s=1}^j v_s)}  (2\hbar)^{v_j} \frac{\ell_j!}{(\ell_j-v_j)!}\ W^{\id}_{(v_1,\ldots,v_w)}. 
\]
\end{thm}
One important step in the proof of the Theorem \ref{thmB} is the $F$-basis in Lemma~\ref{lem:lm}, which is constructed in Appendix~\ref{app:lmlemma}. The proof of Lemma \ref{lem:lm} uses the 6-vertex R-matrix and diagrammatic calculus.

\subsection*{Acknowledgments}
We thank Gufang Zhao and Ryo Fujita for very helpful discussions. Many computations in the paper are under the help of Gufang Zhao. Y.Y. was supported by the Australian Research Council (ARC) via the award DE190101231 and DP210103081. We are very grateful to the referee for the careful reading of the paper and for the detailed comments and suggestions which helped us to improve the manuscript.

\section{Dimension reductions and torus fixed points}

\subsection{The dimension reduction}
The ``dimensional reduction'' for critical cohomology of representations of quiver is proposed in \cite[Section 4.8]{KS} and detailed discussed in \cite{D}. The dimensional reduction allows one to replace the critical cohomology associated  with a $3$-dimensional Calabi-Yau 
category  by an ordinary cohomology associated with a $2$-dimensional Calabi-Yau category.

Consider the setup of $\tilde{Q}$ with the potential given by $\bfw=\sum_{i=1}^w\epsilon^{\ell_i} a^*_ia_i$. 
Fix a vector space $V$ with dimension $v$. Let $W_i$ be the one dimensional vector space at the framing vertex $i\in [1, w]$. 
The dimension reduction along the arrows $\{a_1, \cdots, a_w\}$ amounts to replacing the critical cohomology of $\tilde{X}(v, \vec{1}_w)$ by the ordinary cohomology of the following subvariety of $X^+(v, \vec{1}_w)$: 
\begin{align*}
&\{(\epsilon, a^*_1, \ldots, a_{w}^*)\in X^+(v, \vec{1}_w)\mid \tr(\epsilon^{\ell_i} a^*_ia_i)=0, \forall a_i\in \Hom(V, W_i), i\in [1, w]\}\\
=&\{(\epsilon, a^*_1, \ldots, a_{w}^*)\in X^+(v, \vec{1}_w)\mid \epsilon^{\ell_i} a^*_i=0, i\in [1, w]\}. 
\end{align*}
Clearly, the above subvariety contains the critical locus $\frac{\partial \bfw}{\partial a_i}=0, i\in [1, w]$.

More explicitly, let $C$ be the set consisting the arrow $\{a_1, a_2, \ldots, a_w\}$ of $\tilde{Q}$. We call $C$ the cut of $\tilde{Q}$. 
We have the following derivatives of the potential ${\bfw}$ with respect to arrows in $C$. 
\begin{align*}
\frac{\partial{\bfw}}{\partial a_1}=\epsilon^{\ell_1} a_1^*,\,\  \frac{\partial{\bfw}}{\partial a_2}=\epsilon^{\ell_2} a_2^*, \,\ \ldots, \frac{\partial{\bfw}}{\partial a_w}=\epsilon^{\ell_w} a_w^*.  
\end{align*}
Denote by $Q^+$ the following quiver, which is obtained from the quiver $\tilde{Q}$ \eqref{twQ} by removing the arrows in the cut $C$. That is, 
$Q^+$ has vertices $\bullet$ and  $\square_1, \cdots, \square_w$ and the set of arrows is $\{\epsilon, a^*_1, \ldots, a_w^*\}$, where $\epsilon$ is a loop on $\bullet$, $a^*_i$ is an arrow from $\square_i$ to $\bullet$, $i=1, \cdots, w$. 
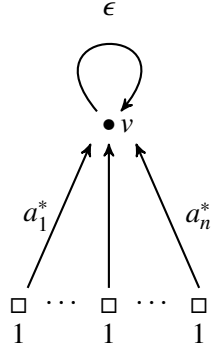
\begin{equation}
\begin{tikzpicture}[scale=1.2]
         \node at (-2.2, 0) {$\bullet$};   
          \node at (-2.2+0.2, 0) {$v$};  
\node at (-2.2, -2) {$\square$}; 
\node at (-2.2, -2-0.3) {$1$}; 

\node at (-2.2-0.5, -2) {$\cdots$}; 
\node at (-2.2+0.5, -2) {$\cdots$}; 
   \draw[->, thick] (-2.2, -1.8) -- (-2.2, -0.2) ;
\node at (-1.2, -2) {$\square$}; 
\node at (-1.2, -2-0.3) {$1$}; 
     \draw[->, thick] (-1.2, -1.8) -- (-2.2+0.3, -0.2) ;
  
\node at (-3.2, -2) {$\square$};
\node at (-3.2, -2-0.3) {$1$}; 
     \draw[->, thick] (-3.2+0.1, -1.8) -- (-2.2+0.1-0.3, -0.2) ;
 \node at (-3.2+0.2, -1) {$a^*_1$}; 
   \node at (-1.2, -1) {$a^*_n$}; 
    \draw[->, thick, shorten <=7pt, shorten >=7pt] ($(-2.2,0)$)
    .. controls +(90+40:1.5) and +(90-40:1.5) .. ($(-2.2,0)$);
 \node at (-2.2, 1.3) {$\epsilon$}; 
    \end{tikzpicture}
\end{equation}
Set $W=\prod_{i=1}^{w} W_i$. Let $\vec{1}_w=(1, 1, \cdots, 1)$ be the framing vector with $w$-components. 
Denote by $X^+(v, \vec{1}_w)$ the representations of $Q^+$:  \[
X^+(v, \vec{1}_w)=\{(\epsilon, a^*_1, \ldots, a^*_w)\mid \epsilon\in \Hom(V, V), a^*_i\in \Hom(W_i, V), i=1, \ldots, w\}. 
\]
We impose the cyclic stability condition that: $(\epsilon,  a^*_1, \ldots, a_w^*)\in X^+(v, \vec{1}_w)$ is stable if $V$ is spanned by the image $\Image(a^*_1, \ldots, a^*_w)$ of $a^*_1, \ldots, a_w^*$, together with the operator $\epsilon$. Denote the stable locus of $X^+(v, \vec{1}_w)$ by $X^{+, st}(v, \vec{1}_w)$. 
The group $G_v\times \C^*\times T_w$ acts on $X^{+}(v, \vec{1}_w)$, where the action of $\C^*$ is given by the grading 
\[
\deg(\ep)=-2, \deg(a^*_i)=\ell_i, \,\ i=1, \ldots, w.
\]
Consider the following $G_v$-equivariant map
\begin{equation*}
\mu_{v, w}: X^+(v, \vec{1}_w)=\Hom(V, V)\times \Hom(W, V)\to \Hom(W, V), (\epsilon, a^*_1, \ldots, a_{w}^*)\mapsto (\epsilon^{\ell_1} a^*_1, \epsilon^{\ell_2} a^*_2,\ldots, \epsilon^{\ell_w} a^*_w). 
\end{equation*}
Denote by $\mu^{-1}_{v, w}(0)$ the kernel of $\mu$. The action of $G_{v}$ on the stable locus $\mu^{-1}_{v, w}(0)^{st}$ is free. 
Consider the variety
\begin{align*}
\mathfrak{M}^+(v, \vec{1}_w):&=\mu^{-1}_{v, w}(0)^{st}/G_v\\
&=\{(\epsilon, a^*_1, \ldots, a_w^*)\mid \epsilon^{\ell_i} a^*_i=0, i\in [1, w]\mid 
\C[\epsilon](\Image(a^*_1, \cdots, a_w^*))=V\}/G_v \subset X^{+, st}(v, \vec{1}_w)/G_v. 
\end{align*}

The general framework from  \cite[Theorem A.1 and Corollary A.9]{D}
gives the following. 
\begin{prop}\label{prop:dimred}
Dimension reduction along the cut $C$ gives the isomorphism of abelian groups.
\[
H^*_{c,\C^*\times T_w}(\tilde{X}(v, \vec{1}_w)^{st}/G_V, \varphi_{\bf{w}})
\cong H^{*}_{c,\C^*\times T_w}(\mu_{v, w}^{-1}(0)^{st}\times_{G_v} \C^{wv}, \C)
\cong H^{*-wv}_{c,\C^*\times T_w}(\mathfrak{M}^+(v, \vec{1}_w), \C), 
\]
where the affine space $\C^{wv}$ is $\Hom(V, W)$, and the right hand side is the compactly supported cohomology of $\mathfrak{M}^+(v, \vec{1}_w)$. 
\end{prop}
The right hand side can be defined the using the right derived functor of the compactly supported global section functor $\Gamma_c$, see \cite{I86} for example. 

\begin{remark}
\yaping{I am here...}
When $\ell_1=\ell_2=\cdots=\ell_w=1$, we have $\mathfrak{M}^+(v, \vec{1}_w)\cong \Gr(v, w)$ which is smooth in this case. In general, the variety $\mathfrak{M}^+(v, \vec{1}_w)$ is singular. When $\ell_1=\cdots=\ell_w=\ell$, the geometry and topology of $\mathfrak{M}^+(v, \vec{1}_w)$ is discussed in \cite{BZJ20} (denoted by $\mathfrak{P}(k, n, \ell)$ with $k=v, n=w$). For example, the equation of $\mathfrak{P}(2, 2, 2)$ is given by $x_{21}^2=x_{11} x_{22}$ \cite[\S 4.1]{BZJ20} which is singular. 

Let $z^{\ell}=  \begin{bmatrix}
    z^{\ell_1} & & \\
    & \ddots & \\
    & & z^{\ell_w}
  \end{bmatrix}$ be the diagonal matrix of size $w\times w$. An isomorphism of $\mathfrak{M}^+(v, \vec{1}_w)$ and the following locus
\[
\bbX_{(\ell_1, \ell_2, \cdots, \ell_w)}(v):=\{\CC[z]^w=L_0 \supset L\mid  zL\subset L, \dim(L_0/L)=v, \text{and} \,\ z^\ell L_0\subset L\}, 
\]
of the affine Grassmannian of $\GL_w$ is established in \cite{CYZJ}. 
\end{remark}
 
\subsection{Torus fixed points}
 The $\C^*\times T_w$-action on $X^+(v, \vec{1}_w)$ induces a $\C^*\times T_{w}$-action on $\mathfrak{M}^+(v, \vec{1}_w)$. 
The fixed points of $\mathfrak{M}^+(v, \vec{1}_w)$ under the action of $\C^*\times T_{w}$ are given as follows. 

\begin{lemma}\label{lem:fixed points}
The set of torus fixed points $\mathfrak{M}^+(v, \vec{1}_w)^{\C^*\times T_{w}}$ is given by
\begin{equation}\label{eq:fixed points}
S(\ell_1, \ldots, \ell_w):=\{(v_1, v_2, \cdots, v_w) \mid 0\leq v_i\leq \ell_i, \forall i\in [1, w],\,\ 
\text{and} \sum_{i=1}^wv_i=v\}.  
\end{equation}
\end{lemma}
\begin{proof}
For any $(\epsilon, a^*_1, \ldots, a_w^*) \in \mathfrak{M}^+(v, \vec{1}_w)$ and $t\in \C^*$, then 
\[
t\cdot (\epsilon, a^*_1, \ldots, a_w^*)
=(t^{-2} \epsilon, t^{\ell_1}a^*_1, \ldots, t^{\ell_w}a_w^*). 
\]
If $(\epsilon, a^*_1, \ldots, a_w^*)$ is fixed by the $\C^*$-action, then there exists a group homomorphism $\rho: \C^*\to G_v, t\mapsto \rho(t)$, such that 
\[
(\rho(t)\epsilon \rho(t)^{-1}, \rho(t)a_1^*, \rho(t)a_2^*, \cdots, \rho(t)a_{w}^*)=
(t^{-2} \epsilon, t^{\ell_1}a^*_1, \ldots, t^{\ell_w}a_w^*). 
\]
We decompose $V$ into eigenspaces of $\rho(t)$: 
\[
V=\oplus_{n\in \Z} V(n), \text{where $V(n)=\{v\in V\mid \rho(t)\cdot v=t^n v\}$. }
\]
Now for the linear maps $a_i^*: W_i \to V$, we have
\[
\rho(t)(a_i^*(x))=t^{\ell_i} a_i^*(x), x\in \C. 
\]
Therefore, $\im(a_i^*)\subset V({\ell_i})$. 
For any $v\in V(n)$, we have
\[
\rho(t)\epsilon(v)=t^{-2} \epsilon \rho(t)(v)=t^{-2} \epsilon(t^n v)=t^{n-2}  \epsilon(v). 
\]
Therefore, $\epsilon(V(n))\subset V(n-2)$. 
By the stability condition, $V$ is generated by the image of $a_i^*$ and the operator $\epsilon$. Therefore, for each $i=1, \ldots, w$, we have the following chain
\[
\xymatrix{
W_i \ar@{->>}[r]^{a_{i}^*}& 
V({\ell_i}) \ar@{->>}[r]^{\epsilon} & V({\ell_i-2})
\ar@{->>}[r]^{\epsilon} & V({\ell_i-4}) \ar@{->>}[r]^{\epsilon} & \cdots  
}
\]
The above chain terminates at $V({\ell_i-(2v_i-2)})$, for some $v_i$ such that $0\leq v_i\leq \ell_i$, because of the relation $\epsilon^{\ell_i} a_i^*=0$. 

Fix a basis of $W=\oplus_i W_i$. 
The point $(\epsilon, a^*_1, \ldots, a_w^*)$ is invariant under $T_w$, then we can choose a basis of $V$, such that $\im(a_i^*), \im(\epsilon a_i^*), \im(\epsilon^2 a_i^*), \cdots $ are all coordinate quotient spaces of $W$. 

This gives the fixed point $(v_1, v_2, \cdots, v_w)$ with the property that $ \sum_{i=1}^wv_i=v$, and 
$0\leq v_i\leq \ell_i, \forall i\in [1, w]. 
$

\end{proof}
The torus fixed point $\lambda=(v_1, \cdots, v_w)$ can be illustrated by the picture as follows. 
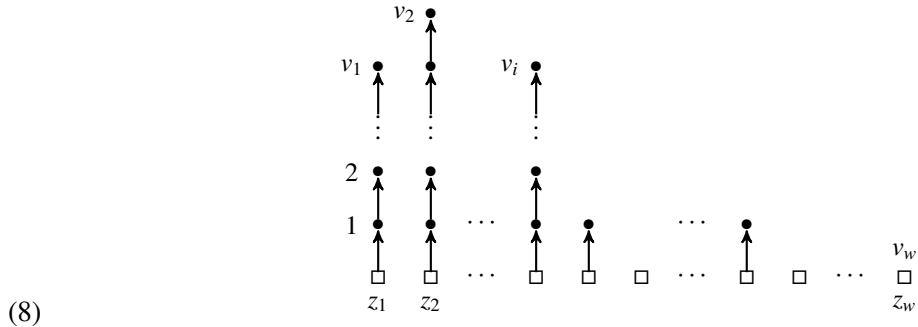
\begin{equation}\label{pic:fixed points}
\begin{tikzpicture}[scale=0.7]
\tikzstyle{every node}=[font=\small]
\node at (4, 3) {$\vdots$};
\node at (2, 4) {$\bullet$};
\node at (4, 4) {$\bullet$};
\draw[->, thick] (4, 3.1) -- (4,3.9) ;
\draw[->, thick] (2, 4.1) -- (2,4.9) ;
\node at (2, 5) {$\bullet$};
\node at (2, 3) {$\vdots$};
\draw[->, thick] (2, 3.1) -- (2,3.9) ;
\node at (1, 3) {$\vdots$};
\draw[->, thick] (1, 3.1) -- (1,3.9) ;
\node at (1, 4) {$\bullet$};
\foreach \x in {3, 7, 10}
{\node at (\x, 0) {$\cdots$};}
\foreach \x in {3, 7}
{\node at (\x, 1) {$\cdots$};}
\foreach \x in {1,2, 4, 5, 6, 8,9 , 11}
{\node at (\x, 0) {$\square$};}	
\foreach \x in {1,2, 4, 5, 8}
{\node at (\x, 1) {$\bullet$};}	
\foreach \x in {1,2, 4}
{\node at (\x, 2) {$\bullet$};}	
\foreach \x in {1,2, 4, 5, 8}
{\draw[->, thick] (\x, 0.1) -- (\x, 0.9) ;}
\foreach \x in {1,2, 4}
{\draw[->, thick] (\x, 1.1) -- (\x, 1.9) ;}
\node at (1, -0.5) {$z_1$};
\node at (2, -0.5) {$z_2$};
\node at (11, -0.5) {$z_w$};

\node at (0.5, 1) {$1$};
\node at (0.5, 2) {$2$};
\node at (0.5, 4) {$v_1$};
\node at (1.5, 5) {$v_2$};
\node at (3.5, 4) {$v_i$};
\node at (11, 0.5) {$v_w$};
\end{tikzpicture}
\end{equation}
where the height of column $i$ is at most $\ell_i$, $i\in [1, w]$. 
The square box $\square$ in column $i$ stands for the basis of the vector space $W_i$ at the framing vertex $\square_i$, and the $\bullet$'s are the eigenbasis of $V$. The arrows from the square box in column $i$ to the $\bullet$' are the maps $a^*_i: W_i \to V$, and other arrows are the maps $\epsilon: V\to V$. For each column $i$, there are $v_i$ number of $\bullet$'s on top of $\square$. Therefore, there are $w$ number of $\square$'s and we have $\sum_i v_i=v$. The height is bounded above by $\ell_i$ because of the relation $\epsilon^{\ell_i} a^*_i=0$, $i=1, \cdots, w$.

Let $z_1, \cdots, z_w$ be the $\C^*\times T_w$-equivariant Chern roots of $W$. Using the grading \eqref{grading}, the $\C^*\times T_w$-equivariant Chern roots of $V$ are given by 
\begin{equation}\label{eq:weights of fix}
z_j-\ell_j\hbar, z_j-(\ell_j-2)\hbar, z_j-(\ell_j-4)\hbar, \,\ 
\cdots, \,\ z_j-(\ell_j-2v_j+2)\hbar. 
\end{equation}
where $l_j \geq v_j\geq 1$. for $j=1, 2, \cdots, w$.  

\subsection{The case when $w=1$}
In this section, we consider the special case when $w=1$. 
\begin{lemma}
\label{lem:w=1}
When $w=1$, we have
$$
\mathfrak{M}^+(v, 1)=
\begin{cases}
\pt, & \text{for $0\leq v\leq \ell$,} \\
\emptyset, & \text{otherwise. }
\end{cases}
$$
\end{lemma}
\begin{proof}
For any pair $(\epsilon, a^*)\in X^{+, st}(v, 1)$, by stability conditions, we have $\im(a^*)$ is a non-zero vector in $V$, and $\C[\epsilon](\im (a^*))=V$. Furthermore, the relation $\epsilon^{\ell} a^*=0$ implies that the dimension of $\C[\epsilon](\im (a^*))$ is at most $\ell$. Therefore, if $v=\dim(V)>\ell$, we have $\mathfrak{M}^+(v, 1)=\emptyset$.

We now assume the dimension of $V$ is in the range $0\leq v\leq \ell$,  we first show the action of $G_v$ on the set 
\[
\{(\epsilon, a^*)\in X^{+, st}(v, 1)\mid \epsilon^{\ell} a^*=0\}
\]
is transitive.

For any element $(\epsilon, a^*)$ in the above set, by the stability condition $\C[\epsilon](\im (a^*))=V$, there exists a basis of $V$ given by 
\[
\mathcal{B}=\{\epsilon^{v-1}\im (a^*), \epsilon^{v-2}\im (a^*), \cdots, \epsilon \im (a^*), \im (a^*)\}. \]
Consider the $\ell$-th power $\epsilon^{\ell}\in \End(V)$. The condition $\epsilon^{\ell}a^*=0$ implies that the operator $\epsilon^{\ell}$ is the zero matrix under the basis $\mathcal{B}$. 
Let $\lambda$ be an eigenvalue of $\epsilon$, then $\lambda^{\ell}$ is an eigenvalue of $\epsilon^{\ell}$. Therefore, $\lambda=0$. This forces $\epsilon^v=0$. 

Choose $g\in G_v$ to be be the matrix whose column vectors are the above basis elements in $\calB$. Then $g$ is the transition matrix from $\calB$ to the standard basis $\{e_1, e_2, \cdots, e_v\}$. We have
\[
g(e_v)=\im(a^*), g(e_{v-1})=\epsilon \im(a^*), \cdots, g(e_1)=\epsilon^{v-1}\im(a^*). 
\]
This implies that \[
g^{-1}\epsilon g(e_{s})=e_{s-1}, (s=v, v-1, \cdots, 2), \text{and $g^{-1}\epsilon g(e_{1})=0$}. 
\]
Therefore, $g^{-1}\epsilon g$ is the Jordan block of the form
\[
J_v=\begin{bmatrix}
0 & 1 & 0& \cdots & 0\\
0 & 0 &1&  \cdots & 0\\
 &&&\ddots&\\
0 & 0 &0& \cdots & 1\\
0 & 0 & 0&\cdots & 0\\
\end{bmatrix}
\] and $g^{-1}a^*$ is the cyclic vector $e_v$. Therefore, the action of $G_v$ is transitive. 

We then check the  stabilizer of the pair $(J_v, e_v)\in X^{+, st}(v, 1)$ is trivial. 
For $g$ lies in the stabilizer, clearly, $gJ_v=J_v g$ and $ge_v=e_v$. Then, 
\[
ge_{v-1}=gJ_v e_v=J_vg e_v=J_v e_v=e_{v-1}. 
\]
Inductively, this shows that $g(e_k)=e_k$ for all $k=1, \cdots, v$. As $g$ fixes each basis vector, we have $g=\id$. 
Thus, $\mathfrak{M}^+(v, 1)=\pt$, if $0\leq v\leq \ell$.  
\end{proof}

\section{Proof of Theorem \ref{thm:action} and Proposition \ref{prop:w=1}} \label{sec:proof}

\subsection{Proof of  Theorem \ref{thm:action} }
\label{subsec:proof of thm}
In this subsection, we prove Theorem \ref{thm:action}.
\subsubsection{The Yangian of $\mathfrak{sl}_2$}
Recall that $Y_{\hbar}(\mathfrak{sl}_2)$ is an associative algebra over $\C[\hbar]$, generated by the variables
\[
e_{r}, f_{r}, \psi_{r}, (r\in \N),
\]
subject to relations described below. 
\begin{align}
&\text{For any $r, s\in \N$}, \,\ 
[\psi_r, \psi_s]=0;\label{Y1}\tag{Y1} \\
&\text{For any $s\in \N$}, \,\ 
 [\psi_0, e_{s}]= 2 e_{s}, \,\
 [\psi_0, f_{s}]= -2 f_{s}; \label{Y2}\tag{Y2}\\
 & \text{For any $r, s\in \N$}, \,\
 [\psi_{r+1}, e_s]-[\psi_{r}, e_{s+1}]=
2\hbar(\psi_{r}e_s+e_s\psi_{r}), \,\ 
[\psi_{r+1}, f_s]-[\psi_{r}, f_{s+1}]=
-2\hbar(\psi_{r}f_s+f_s\psi_{r});\label{Y3}\tag{Y3}\\
&\text{For $r, s\in \N$}, [e_{r+1}, e_{s}]-[e_{r}, e_{s+1}]=2\hbar(e_{r}e_{s}+e_{s}e_{r}), \,\
[f_{r+1}, f_{s}]-[f_{r}, f_{s+1}]=-2\hbar(f_{r}f_{s}+f_{s}f_{r});\label{Y4}\tag{Y4}\\
&\text{For $r, s\in \N$}, 
[e_r, f_s]=2\hbar\psi_{r+s}.\label{Y5}\tag{Y5}
\end{align}

\subsubsection{The action of the Yangian}

As in the discussion in \cite[\S 3.4]{BZJ20}, it remains unknown whether $H^*_{c,\C^*\times T_w}(\mathfrak{M}^+(v, \vec{1}_w))^{\vee}$ is free over  $H^*_{c,\C^*\times T_w}(\pt)^{\vee}$, or whether $\mathfrak{M}^+(v, \vec{1}_w)$ is equivariantly
formal (See \cite[Lemma 2]{B00} for equivalent characterizations). We plan to investigate this in future work. In the present paper, we use the following equivariant localization (see \cite[Theorem 6.2]{GKM} and also \cite[Lemma 1]{B00}).  
\begin{lemma}
Let $i: \mathfrak{M}^+(v, \vec{1}_w)^{\C^*\times T_w}\inj \mathfrak{M}^+(v, \vec{1}_w)$ be the inclusion of the fixed point set. Then
\[
i^{*}: H^*_{c,\C^*\times T_w}(\mathfrak{M}^+(v, \vec{1}_w))^{\vee}\to H^*_{c,\C^*\times T_w}(\mathfrak{M}^+(v, \vec{1}_w)^{\C^*\times T_w})^{\vee}, \,\ 
i_{*}: H^*_{c,\C^*\times T_w}(\mathfrak{M}^+(v, \vec{1}_w)^{\C^*\times T_w})^{\vee} \to  H^*_{c,\C^*\times T_w}(\mathfrak{M}^+(v, \vec{1}_w))^{\vee}
\]
become isomorphism after inverting finite many characters of $\C^*\times T_w$. 
\end{lemma}

By the equivariant localization and Lemma \ref{lem:fixed points}, a basis of 
$
H^*_{c,\C^*\times T_w}(\mathfrak{M}^+(v, \vec{1}_w))^{\vee}\otimes_R K$ is given by $\{\lambda \mid \lambda \in S(\ell_1, \ldots, \ell_w)\}$.

Define the generating series 
\begin{equation}
\psi(z):= 1+2\hbar\sum_{i\geq 0} \psi_i z^{-i-1}. 
 \end{equation}

For a $G$-equivariant vector bundle $E$ on $X$, denote by $\lambda_{t}(E)$ the equivariant Chern polynomial of $E$: $\lambda_{t} (E)\in H_{G}(X)[t]$. 
The coefficient of $t$ in $\lambda_{t}(E)$ is the equivariant first Chern class $c_{1}(E)\in H_{G}(X)$. When $E$ is a line bundle, we have 
\[
\lambda_{t}(E)=1+c_{1}(E) t. 
\] Moreover, for any $E$
and $F$,  we have $\lambda_{t}(E \oplus F) = \lambda_t(E)  \lambda_t(F)$.
Let $
t:=-\frac{1}{z}$. 
Then, for any line bundle $E$, we have $\lambda_{-\frac{1}{z}}(E)=1-c_{1}(E)\frac{1}{z}$. For any $G$-equivariant K-theory class $[E]-[q E]$, we have 
\[
\lambda_{-\frac{1}{z}}([E]-[q E])
=\frac{1-c_{1}(E)\frac{1}{z}}{1-(c_{1}(E)+\hbar)\frac{1}{z}}=\frac{z-c_{1}(E)}{z-(c_{1}(E)+\hbar)}. 
\]

We define the action of $\psi(z)$ on $\bigoplus_{v\in \N} H^*_{c,\C^*\times T_w}(\mathfrak{M}^+(v, \vec{1}_w))^{\vee}$
via the Chern polynomial $\lambda_{-\frac{1}{z}}$ of  the following class in the Grothendieck group \[(q^{-2})\calV-(q^{2})\calV+\sum_{i=1}^w(q_i^{\ell_i})\calW_i -\sum_{i=1}^w(q^{-\ell_i})\calW_i.\] Here $q$ is the coordinate line considered as a $\C^*$-representation. It is straightforward to show the following.  See  \cite[\S~6.2]{RSYZ1} for example. 
\begin{prop}\label{prop:psi(z)}
Let $\lambda$ be a $\C^*$-fixed point. The eigenvalue of $\psi(z)$ on $[\lambda]$ equals
\begin{equation}
\label{eq:psi}
\tilde{\psi}(z):=
\prod_{\square\in \lambda} \frac{z-x_{\square}+2\hbar }{z-x_{\square}-2\hbar} \prod_{j=1}^w\frac{z-z_j-\ell_j\hbar}{z-z_j+\ell_j\hbar}. 
\end{equation}
\end{prop}
\begin{proof}
The factor $\frac{z-z_j-\ell_j\hbar}{z-z_j+\ell_j\hbar}$ comes from $\lambda_{-\frac{1}{z}} \Big(q_i^{\ell_i}\calW_i -q^{-\ell_i}\calW_i\Big)$, and the factor  $\prod_{\square\in \lambda} \frac{z-x_{\square}+2\hbar }{z-x_{\square}-2\hbar}$ comes from $\lambda_{-\frac{1}{z}} \Big((q^{-2})\calV-(q^{2})\calV\Big)$. 
\end{proof}

Let $\lambda = (v_1, v_2, \dots, v_w)$ be a torus fixed point. We denote by $\lambda + \blacksquare = (v_1', v_2', \dots, v_w')$ another torus fixed point such that $v_i \leq v_i'$ for all $1 \leq i \leq w$, and $\sum_i v_i' = \sum_i v_i + 1$. We say that $\lambda + \blacksquare$ is obtained by adding a “box” $\blacksquare$ to $\lambda$.

We use the notation $\langle\lambda|e_{r}| \lambda+\blacksquare\rangle$ for the coefficient of $\lambda+\blacksquare$ in the expansion of $e_r(\lambda)$. Similar convention is used for $\langle \lambda+\blacksquare|f_{m}| \lambda\rangle$.

For a box $\square$ or $\blacksquare$, denote by $x_\square$ or $x_\blacksquare$ the $\C^*\times T_w$-weight of the corresponding box, which is an element in $H_{\C^*\times T_w}(\pt)$.
Recall that the $T_w\times \C^*$-equivariant Chern roots of $\calW$ are $z_1, z_2, \cdots, z_w$. Therefore, $x_\blacksquare$ is determined by the formula \eqref{eq:weights of fix}. 

Define the matrix coefficients of the operators $e_{r}, f_{m}$ in the  basis of fixed points are as follows:
	\begin{align}
	\langle\lambda|e_r| \lambda+\blacksquare\rangle=
	&  \Res_{z=x_{\blacksquare}} \Big(z^r 
\prod_{\square\in \lambda} \frac{z-x_{\square} }{z-x_{\square}-2\hbar} \prod_{j=1}^w\frac{z-z_j-\ell_j\hbar}{z-z_j+\ell_j\hbar}\Big), 
	\label{eaction}\\
	\langle \lambda+\blacksquare|f_{m}| \lambda\rangle=&
\Big(z^m
\prod_{ \square\in \lambda}\frac{z-x_{\square}+2\hbar}{ z-x_{\square}}\Big)|_{z=x_{\blacksquare}}. 	\label{faction}\end{align}

\begin{prop}\label{prop:ef formula}
The formulas \eqref{eq:psi} \eqref{eaction} \eqref{faction} satisfy the relations of $Y_{\hbar}(\mathfrak{sl}_2)$. 
In particular, $Y_{\hbar}(\mathfrak{sl}_2)$ acts on the cohomology $
H^*_{c,\C^*\times T_w}(\mathfrak{M}^+(v, \vec{1}_w))^{\vee}\otimes_R K$. 
\end{prop}
\subsubsection{Proof of Proposition \ref{prop:ef formula}}
In this section, we prove Proposition \ref{prop:ef formula} by verifying the relations of  $Y_{\hbar}(\mathfrak{sl}_2)$.

We say $\blacksquare$ is an \textit{addible box}, if $\lambda$ is a fixed point, and $\lambda+\blacksquare$ is still a fixed point. We say $\blacksquare$ is a \textit{removable box}, if $\lambda$ is a fixed point, and $\lambda-\blacksquare$ is still a fixed point. 
\begin{lemma}\label{lem:e}
Let $z_1, z_2, \cdots, z_w$ be generic parameters. 
All the poles of the rational function of $z$
\[
E(z):=z^k 
\prod_{\square\in \lambda} \frac{z-x_{\square} }{z-x_{\square}-2\hbar} \prod_{j=1}^w\frac{z-z_j-\ell_j\hbar}{z-z_j+\ell_j\hbar}
\] are simple poles at the addible boxes. 
\end{lemma}
\begin{proof}
As $z_1, z_2, \cdots, z_w$ are generic parameters, we have $z_i\neq z_j+a\hbar$, for any $a\in \Z$. it suffices to show the claim for $w=1$. 
In this case, the fixed point $\lambda$ contains only one column illustrated as in the following picture. 
\[
\begin{tikzpicture}[scale=0.7]
\tikzstyle{every node}=[font=\small]
\node at (1, 4) {$\bullet$};	
\node at (1, 3) {$\vdots$};	
\foreach \x in {1}
{\node at (\x, 0) {$\square$};}	
\foreach \x in {1}
{\node at (\x, 1) {$\bullet$};}	
\foreach \x in {1}
{\node at (\x, 2) {$\bullet$};}	
\foreach \x in {1}
{\draw[->, thick] (\x, 0.1) -- (\x, 0.9) ;}
\foreach \x in {1}
{\draw[->, thick] (1, 1.1) -- (1, 1.9) ;}
\draw[->, thick] (1, 3.1) -- (1, 3.9) ;
\end{tikzpicture}
\]
Let $v=\dim(V)$. Denote the equivariant variable by $z_i$. 
Assume $1\leq v\leq \ell_i$, then by \eqref{eq:weights of fix}, the weights of $\lambda$ is given by
\[
z_i-\ell_i\hbar, z_i-\ell_i\hbar+2\hbar, z_i-\ell_i\hbar+4\hbar, 
\cdots, z_i-\ell_i\hbar+2v\hbar-2\hbar. 
\]
Then, the function $E(z)$ is given as
\begin{align}
&
z^k\frac{ z-z_i+\ell_i\hbar}{z-z_i+\ell_i\hbar-2\hbar}
\frac{ z-z_i+\ell_i\hbar-2\hbar}{z-z_i+\ell_i\hbar-4\hbar} 
\frac{ z-z_i+\ell_i\hbar-4\hbar}{z-z_i+\ell_i\hbar-6\hbar} \cdots
\frac{ z-z_i+\ell_i\hbar-2v\hbar+2\hbar}{z-z_i+\ell_i\hbar-2v\hbar}  \frac{z-z_i-\ell_i\hbar}{z-z_i+\ell_i\hbar} \notag\\
&= z^k\frac{z-z_i-\ell_i\hbar}{z-z_i+\ell_i\hbar-2v\hbar} \label{fund:E(z)}
\end{align}

If $v=\ell_i$, then there is no addible boxes as the height of the fixed points is at most $\ell_i$. The function $E(z)=z^k$ is a regular function, and it has no poles. 

If $0\leq v <\ell_i$, then $E(z)$ has one simple pole at $z=z_i-\ell_i\hbar+2v\hbar$, which is the unique addible box. 

This shows the assertion. 
\end{proof}

\begin{lemma}\label{lem:psi}
The poles of the function $\tilde{\psi}(z)$ \eqref{eq:psi}
 \begin{align*}
\tilde{\psi}(z)=\prod_{\square\in \lambda} \frac{z-x_{\square}+2\hbar }{z-x_{\square}-2\hbar} \prod_{j=1}^w\frac{z-z_j-\ell_j\hbar}{z-z_j+\ell_j\hbar} 
\end{align*}
are at the addible and removable boxes of $\lambda$. Furthermore, all the poles are simple poles. 
\end{lemma}
\begin{proof}
Similar as in the proof of Lemma \ref{lem:e}, it suffices to show the claim for $w=1$. Denote the dimension of $V$ by $v$. 

Denote the equivariant variable by $z_i$. Assume $1\leq v\leq \ell_i$,  the weights of $\lambda$ are given by
\[
x_{\square}: z_i-\ell_i\hbar, z_i-\ell_i\hbar+2\hbar, z_i-\ell_i\hbar+4\hbar, 
\cdots, z_i-\ell_i\hbar+2v\hbar-2\hbar. 
\]

We have the following cases. 

\textbf{Case 1}: If $v=0$ with $w=1$, then, 
\begin{align*}
&\tilde{\psi}(z)=\frac{z-z_i-\ell_i\hbar}{z-z_i+\ell_i\hbar}
\end{align*}
Clearly, $\tilde{\psi}(z)$ in this case has only one simple pole at $z=z_i-\ell_i\hbar$, which is the unique addible box. There is no removable box in this case.

\textbf{Case 2: } If $1\leq v \leq \ell_i$ with $w=1$, we have the following detailed computation
\begin{align*}
\tilde{\psi}(z)
=&\frac{z-z_i+\ell_i\hbar+2\hbar }{z-z_i+\ell_i\hbar-2\hbar}
\frac{z-z_i+\ell_i\hbar }{z-z_i+\ell_i\hbar-4\hbar}
\frac{z-z_i+\ell_i\hbar-2\hbar }{z-z_i+\ell_i\hbar-6\hbar}
\cdots
\frac{z-z_i+\ell_i\hbar-2v\hbar+4\hbar }{z-z_i+\ell_i\hbar-2v\hbar}\frac{z-z_i-\ell_i\hbar}{z-z_i+\ell_i\hbar}\\
=& \frac{(z-z_i+\ell_i\hbar+2\hbar )(z-z_i+\ell_i\hbar) }{(z-z_i+\ell_i\hbar-2v\hbar+2\hbar)(z-z_i+\ell_i\hbar-2v\hbar)}\frac{z-z_i-\ell_i\hbar}{z-z_i+\ell_i\hbar}\\
=& \frac{(z-z_i+\ell_i\hbar+2\hbar )(z-z_i-\ell_i\hbar) }{(z-z_i+\ell_i\hbar-2v\hbar+2\hbar)(z-z_i+\ell_i\hbar-2v\hbar)}. 
\end{align*}

\textbf{Case 2 (a): }
When $1\leq v <\ell_i$, the function $\tilde{\psi}(z)$ has two simple poles. One pole is at $z= z_i-\ell_i\hbar+2v\hbar-2\hbar$, which is the unique removable box, and the other pole is at $z=z_i-\ell_i\hbar+2v\hbar$ which is the unique addible box. 

\textbf{Case 2 (b): } When $v=\ell_i$, we have
\[
\tilde{\psi}(z)=
\frac{z-z_i+\ell_i\hbar+2\hbar  }{z-z_i-\ell_i\hbar+2\hbar}. 
\]
Then $\tilde{\psi}(z)$ in this case has one simple pole at $z=z_i+\ell_i\hbar-2\hbar$, which is the unique removable box.  There is no addible box in this case. 

This completes the proof. 
\end{proof}

\begin{lemma}
The formulas \eqref{eq:psi} \eqref{eaction} \eqref{faction}  satisfy the relation \eqref{Y2}. 
\end{lemma}
\begin{proof}
For the rational function $f(z) = \frac{z - a}{z - b}$, the expansion at $z=\infty$ is given by 
\[
\begin{aligned}
f(z) &= \frac{z - a}{z - b}
= \frac{1 - \frac{a}{z}}{1 - \frac{b}{z}} = \left(1 - \frac{a}{z}\right)
\left(1 + \frac{b}{z} + \cdots \right) 
= 1 + \frac{b - a}{z}
   + \cdots 
\end{aligned}
\]
Therefore, the action of $\psi_{0}$ on $\lambda$ is given by
\begin{align*}
\psi_{0}(\lambda)=
\frac{4\hbar |\lambda|- \sum_{j=1}^{w} (2\ell_j \hbar)}{2\hbar} (\lambda)=\Big(
2 |\lambda|- \sum_{j=1}^{w} \ell_j 
\Big)(\lambda). 
\end{align*}
Therefore, the matrix coefficient of $\psi_{0}\circ e_{r}-e_r\circ \psi_0$ is given by
\begin{align*}
\langle\lambda|e_r| \lambda+\blacksquare\rangle
\langle
\lambda+\blacksquare|
\psi_0 | \lambda+\blacksquare\rangle-\langle
\lambda|
\psi_0 | \lambda\rangle\langle\lambda|e_r|\lambda+\blacksquare\rangle
=2\langle\lambda|e_r|\lambda+\blacksquare\rangle. 
\end{align*}

Similarly, we have
\begin{align*}
\langle\lambda+\blacksquare|f_r| \lambda\rangle
\langle
\lambda|
\psi_0 | \lambda\rangle
-\langle
\lambda+\blacksquare|
\psi_0 | \lambda+\blacksquare\rangle\langle\lambda+\blacksquare|f_r|\lambda\rangle
=-2\langle\lambda|f_r|\lambda+\blacksquare\rangle. 
\end{align*}
\end{proof}

\begin{lemma}
The formula \eqref{eaction}  respects the relation
\[
[e_{r+1}, e_{s}]-[e_{r}, e_{s+1}]=2\hbar(e_{r}e_{s}+e_{s}e_{r})  
\] in \eqref{Y4}. 
\end{lemma}
\begin{proof}
Rewrite the relation as
\[
e_{r+1}\circ e_{s}-e_{s}\circ e_{r+1}
-e_{r}\circ e_{s+1}
+e_{s+1}\circ e_{r}
=2\hbar(e_{r}\circ e_{s}+e_{s}\circ e_{r})
\]
Fix partitions $\lambda$ and $\lambda+\blacksquare_i+\blacksquare_j$, where $\blacksquare_i, \blacksquare_j$ are two addible boxes. 

We first assume that $\blacksquare_i, \blacksquare_j$ are on the $i$-th and $j$-th string with $i\neq j$.  
Assume that $i\neq j$, we now compute the matrix coefficient of $e_s\circ e_r$ at the entry $(\lambda, \lambda+\blacksquare_i+\blacksquare_j)$. 

The matrix coefficient of $e_s\circ e_r$ is given by
\[
\langle\lambda|e_r e_s | \lambda+\blacksquare_i +\blacksquare_j\rangle
=\langle\lambda|e_r| \lambda+\blacksquare_i\rangle
\langle
\lambda+\blacksquare_i|
e_s | \lambda+\blacksquare_i +\blacksquare_j\rangle
+\langle\lambda|e_r| \lambda+\blacksquare_j\rangle
\langle
\lambda+\blacksquare_j|
e_s | \lambda+\blacksquare_i +\blacksquare_j\rangle. 
\]

We have 
\begin{align*}
&\langle\lambda|e_r| \lambda+\blacksquare_i\rangle
\langle
\lambda+\blacksquare_i|
e_s | \lambda+\blacksquare_i +\blacksquare_j\rangle
\\
&=
\Res_{z=x_{\blacksquare_i}} \Big(z^r 
\prod_{\square\in \lambda} \frac{z-x_{\square} }{z-x_{\square}-2\hbar} \prod_{a=1}^w\frac{z-z_a-\ell_a\hbar}{z-z_a+\ell_a\hbar}\Big)\cdot
\Res_{z=x_{\blacksquare_j}} \Big(z^s 
\prod_{\square\in \lambda} \frac{z-x_{\square} }{z-x_{\square}-2\hbar} 
\frac{z-x_{\blacksquare_i} }{z-x_{\blacksquare_i}-2\hbar}
\prod_{a=1}^w\frac{z-z_a-\ell_a\hbar}{z-z_a+\ell_a\hbar}\Big)
\end{align*}
By Lemma \ref{lem:e}, the above functions have only one simple pole at the point where to take the residue, thus, we have
\begin{align*}
&\langle\lambda|e_r| \lambda+\blacksquare_i\rangle
\langle
\lambda+\blacksquare_i|
e_s | \lambda+\blacksquare_i +\blacksquare_j\rangle
+\langle\lambda|e_r| \lambda+\blacksquare_j\rangle
\langle
\lambda+\blacksquare_j|
e_s | \lambda+\blacksquare_i +\blacksquare_j\rangle\\
=&
 \Big(x_{\blacksquare_i})^r 
\prod_{\square\in \lambda} \frac{x_{\blacksquare_i}-x_{\square} }{\widehat{x_{\blacksquare_i}-x_{\square}-2\hbar}} \prod_{a=1}^w\frac{
{x_{\blacksquare_i}-z_a-\ell_a\hbar}}{\widehat{x_{\blacksquare_i}-z_a+\ell_a\hbar}}\Big)\cdot
\Big((x_{\blacksquare_j})^s 
\prod_{\square\in \lambda} \frac{x_{\blacksquare_j}-x_{\square} }{\widehat{x_{\blacksquare_j}-x_{\square}-2\hbar}} 
\frac{x_{\blacksquare_j}-x_{\blacksquare_i} }{{x_{\blacksquare_j}-x_{\blacksquare_i}-2\hbar}}
\prod_{a=1}^w\frac{x_{\blacksquare_j}-z_a-\ell_a\hbar}{\widehat{x_{\blacksquare_j}-z_a+\ell_a\hbar}}\Big)\\
&+
 \Big(x_{\blacksquare_j})^r 
\prod_{\square\in \lambda} \frac{x_{\blacksquare_j}-x_{\square} }{\widehat{x_{\blacksquare_j}-x_{\square}-2\hbar}} \prod_{a=1}^w\frac{
{x_{\blacksquare_j}-z_a-\ell_a\hbar}}{\widehat{x_{\blacksquare_j}-z_a+\ell_a\hbar}}\Big)\cdot
\Big((x_{\blacksquare_i})^s 
\prod_{\square\in \lambda} \frac{x_{\blacksquare_i}-x_{\square} }{\widehat{x_{\blacksquare_i}-x_{\square}-2\hbar}} 
\frac{x_{\blacksquare_i}-x_{\blacksquare_j} }{{x_{\blacksquare_i}-x_{\blacksquare_j}-2\hbar}}
\prod_{a=1}^w\frac{x_{\blacksquare_i}-z_a-\ell_a\hbar}{\widehat{x_{\blacksquare_i}-z_a+\ell_a\hbar}}\Big)\\
=& (\text{common factor})
\Big(
(x_{\blacksquare_i})^r(x_{\blacksquare_j})^s\frac{x_{\blacksquare_j}-x_{\blacksquare_i} }{{x_{\blacksquare_j}-x_{\blacksquare_i}-2\hbar}}
+ 
(x_{\blacksquare_j})^r(x_{\blacksquare_i})^s\frac{x_{\blacksquare_i}-x_{\blacksquare_j} }{{x_{\blacksquare_i}-x_{\blacksquare_j}-2\hbar}}
\Big)
\end{align*}
where the notation $\widehat{}$ in the denominator means removing the only factor that is zero, and the common factor is given by
\[
\text{common factor}=
 \Big( 
\prod_{\square\in \lambda} \frac{x_{\blacksquare_i}-x_{\square} }{\widehat{x_{\blacksquare_i}-x_{\square}-2\hbar}} \prod_{a=1}^w\frac{
{x_{\blacksquare_i}-z_a-\ell_a\hbar}}{\widehat{x_{\blacksquare_i}-z_a+\ell_a\hbar}}\Big)\cdot
\Big( 
\prod_{\square\in \lambda} \frac{x_{\blacksquare_j}-x_{\square} }{\widehat{x_{\blacksquare_j}-x_{\square}-2\hbar}} 
\prod_{a=1}^w\frac{x_{\blacksquare_j}-z_a-\ell_a\hbar}{\widehat{x_{\blacksquare_j}-z_a+\ell_a\hbar}}\Big)
\]

Factoring out the above common factor, for $e_{s}\circ e_{r+1}-e_{r+1}\circ e_{s}
-e_{s+1}\circ e_{r}
+e_{r}\circ e_{s+1}$, we are left with: 
\begin{align*}
&\Big((x_{\blacksquare_i})^{r+1}(x_{\blacksquare_j})^s\frac{x_{\blacksquare_j}-x_{\blacksquare_i} }{{x_{\blacksquare_j}-x_{\blacksquare_i}-2\hbar}}
+ 
(x_{\blacksquare_j})^{r+1}(x_{\blacksquare_i})^s\frac{x_{\blacksquare_i}-x_{\blacksquare_j} }{{x_{\blacksquare_i}-x_{\blacksquare_j}-2\hbar}}
\Big)\\
&-\Big(
(x_{\blacksquare_i})^s(x_{\blacksquare_j})^{r+1}\frac{x_{\blacksquare_j}-x_{\blacksquare_i} }{{x_{\blacksquare_j}-x_{\blacksquare_i}-2\hbar}}
+ 
(x_{\blacksquare_j})^s(x_{\blacksquare_i})^{r+1}\frac{x_{\blacksquare_i}-x_{\blacksquare_j} }{{x_{\blacksquare_i}-x_{\blacksquare_j}-2\hbar}}
\Big)\\
&-\Big(
(x_{\blacksquare_i})^r(x_{\blacksquare_j})^{s+1}\frac{x_{\blacksquare_j}-x_{\blacksquare_i} }{{x_{\blacksquare_j}-x_{\blacksquare_i}-2\hbar}}
+ 
(x_{\blacksquare_j})^r(x_{\blacksquare_i})^{s+1}\frac{x_{\blacksquare_i}-x_{\blacksquare_j} }{{x_{\blacksquare_i}-x_{\blacksquare_j}-2\hbar}}
\Big)\\
&+\Big((x_{\blacksquare_i})^{s+1}(x_{\blacksquare_j})^r\frac{x_{\blacksquare_j}-x_{\blacksquare_i} }{{x_{\blacksquare_j}-x_{\blacksquare_i}-2\hbar}}
+ 
(x_{\blacksquare_j})^{s+1}(x_{\blacksquare_i})^r\frac{x_{\blacksquare_i}-x_{\blacksquare_j} }{{x_{\blacksquare_i}-x_{\blacksquare_j}-2\hbar}}
\Big)\\
=&
\Big((x_{\blacksquare_i})^{r}(x_{\blacksquare_j})^s (x_{\blacksquare_j}-x_{\blacksquare_i})
\Big(
x_{\blacksquare_i}
\frac{1}{{x_{\blacksquare_j}-x_{\blacksquare_i}-2\hbar}}
-x_{\blacksquare_i}\frac{-1}{{x_{\blacksquare_i}-x_{\blacksquare_j}-2\hbar}}
-x_{\blacksquare_j}\frac{
1 }{{x_{\blacksquare_j}-x_{\blacksquare_i}-2\hbar}}
+x_{\blacksquare_j}\frac{-1}{{x_{\blacksquare_i}-x_{\blacksquare_j}-2\hbar}}\Big)
\\
+&(x_{\blacksquare_j})^{r}(x_{\blacksquare_i})^s (x_{\blacksquare_i}-x_{\blacksquare_j})
\Big(x_{\blacksquare_j}
\frac{1}{{x_{\blacksquare_i}-x_{\blacksquare_j}-2\hbar}}
-x_{\blacksquare_j}\frac{-1}{{x_{\blacksquare_j}-x_{\blacksquare_i}-2\hbar}}
-x_{\blacksquare_i}\frac{1 }{{x_{\blacksquare_i}-x_{\blacksquare_j}-2\hbar}}
+x_{\blacksquare_i}\frac{-1}{{x_{\blacksquare_j}-x_{\blacksquare_i}-2\hbar}}\Big)
\end{align*}

For $-2\hbar(e_{r}\circ e_{s}+e_{s}\circ e_{r})$, by factoring out the same common factor, we are left with 
\begin{align*}
&(x_{\blacksquare_i})^r(x_{\blacksquare_j})^s
\frac{x_{\blacksquare_j}-x_{\blacksquare_i} }{{x_{\blacksquare_j}-x_{\blacksquare_i}-2\hbar}}
+ 
(x_{\blacksquare_j})^r(x_{\blacksquare_i})^s\frac{x_{\blacksquare_i}-x_{\blacksquare_j} }{{x_{\blacksquare_i}-x_{\blacksquare_j}-2\hbar}}\\
+&
(x_{\blacksquare_i})^s(x_{\blacksquare_j})^r\frac{x_{\blacksquare_j}-x_{\blacksquare_i} }{{x_{\blacksquare_j}-x_{\blacksquare_i}-2\hbar}}
+ 
(x_{\blacksquare_j})^s(x_{\blacksquare_i})^r\frac{x_{\blacksquare_i}-x_{\blacksquare_j} }{{x_{\blacksquare_i}-x_{\blacksquare_j}-2\hbar}}\\
=&(x_{\blacksquare_i})^r(x_{\blacksquare_j})^s(x_{\blacksquare_j}-x_{\blacksquare_i}) \Big(\frac{1 }{{x_{\blacksquare_j}-x_{\blacksquare_i}-2\hbar}}+\frac{-1}{{x_{\blacksquare_i}-x_{\blacksquare_j}-2\hbar}}\Big)
\\&
+(x_{\blacksquare_j})^r(x_{\blacksquare_i})^s(x_{\blacksquare_i}-x_{\blacksquare_j})\Big(
\frac{1}{{x_{\blacksquare_i}-x_{\blacksquare_j}-2\hbar}}+\frac{-1}{{x_{\blacksquare_j}-x_{\blacksquare_i}-2\hbar}}
\Big)
\end{align*}
It is straightforward to verify that
\begin{align*}
&x_{\blacksquare_i}
\frac{1}{{x_{\blacksquare_j}-x_{\blacksquare_i}-2\hbar}}
-x_{\blacksquare_i}\frac{-1}{{x_{\blacksquare_i}-x_{\blacksquare_j}-2\hbar}}
-x_{\blacksquare_j}\frac{
1 }{{x_{\blacksquare_j}-x_{\blacksquare_i}-2\hbar}}
+x_{\blacksquare_j}\frac{-1}{{x_{\blacksquare_i}-x_{\blacksquare_j}-2\hbar}}\\
=&-2\hbar\Big(\frac{1 }{{x_{\blacksquare_j}-x_{\blacksquare_i}-2\hbar}}+\frac{-1}{{x_{\blacksquare_i}-x_{\blacksquare_j}-2\hbar}}\Big). 
\end{align*}
This shows that 
$e_{s}\circ e_{r+1}-e_{r+1}\circ e_{s}
-e_{s+1}\circ e_{r}
+e_{r}\circ e_{s+1}=-2\hbar(e_{r}e_{s}+e_{s}e_{r})$. 

When the two addible boxes are on the same $i$-th string, we denote them by $\blacksquare_i, \tilde{\blacksquare}_i$ with weights $x_{\blacksquare_i}-x_{\tilde\blacksquare_i}=-2\hbar$. To verify the relation, it suffices to check it when $w=1$ and there is only one string. In this case, we have: 
\begin{align*}
\langle\lambda|e_r| \lambda+\blacksquare_i\rangle
=&(x_{\blacksquare_i})^r \big(2(v-\ell)\hbar\big). 
\end{align*}
Therefore, the operator $e_{s}\circ e_{r}$ is given by
\begin{align*}
\langle\lambda|e_r| \lambda+\blacksquare_i 
\rangle \langle \lambda+\blacksquare_i\lambda|e_s| \lambda+\blacksquare_i +\tilde\blacksquare_i
\rangle =4\hbar^2 (x_{\blacksquare_i})^r 
(x_{\tilde{\blacksquare}_i})^s (v-\ell)(v-\ell+1)
\end{align*}

Factoring out the common factor $4\hbar^2(v-\ell)(v-\ell+1)$, the operator
$e_{s}\circ e_{r+1}-e_{r+1}\circ e_{s}
-e_{s+1}\circ e_{r}
+e_{r}\circ e_{s+1}$ corresponds to
\begin{align*}
&(x_{\blacksquare_i})^{r+1}
(x_{\tilde{\blacksquare}_i})^s-(x_{\blacksquare_i})^s 
(x_{\tilde{\blacksquare}_i})^{r+1}-(x_{\blacksquare_i})^r 
(x_{\tilde{\blacksquare}_i})^{s+1}+(x_{\blacksquare_i})^{s+1} 
(x_{\tilde{\blacksquare}_i})^r
=(x_{\blacksquare_i})^r 
(x_{\tilde{\blacksquare}_i})^s(x_{\blacksquare_i}-x_{\tilde\blacksquare_i} )
-(x_{\blacksquare_i})^s 
(x_{\tilde{\blacksquare}_i})^{r}(
x_{\tilde\blacksquare_i}-
x_{\blacksquare_i}
)\\
=&-2\hbar(x_{\blacksquare_i})^r 
(x_{\tilde{\blacksquare}_i})^s
-2\hbar(x_{\blacksquare_i})^s 
(x_{\tilde{\blacksquare}_i})^{r}
)=-2\hbar\Big(x_{\blacksquare_i})^r 
(x_{\tilde{\blacksquare}_i})^s+ (x_{\blacksquare_i})^s 
(x_{\tilde{\blacksquare}_i})^{r}\Big), 
\end{align*}
which corresponds to the operator $-2\hbar(e_{r}\circ e_{s}+e_{s}\circ e_{r})$. 
\end{proof}

\begin{lemma}
The formula \eqref{faction}  respects the relation
\[
[f_{r+1}, f_{s}]-[f_{r}, f_{s+1}]=-2\hbar(f_{r}f_{s}+f_{s}f_{r})  
\] in \eqref{Y4}. 
\end{lemma}
\begin{proof}
If the two removable boxes 
$\blacksquare_{i}$ and $\blacksquare_{j}$ are on two different strings, the matrix coefficient of $f_{n}\circ f_{m}$ is given by
\begin{align*}
& \langle\lambda+\blacksquare_{i}+\blacksquare_{j}|f_m f_n| \lambda\rangle\\
=&
\langle\lambda+\blacksquare_{i}+\blacksquare_{j}|f_m
| \lambda+\blacksquare_i\rangle \langle \lambda+\blacksquare_i|
f_n| \lambda\rangle
+\langle\lambda+\blacksquare_{i}+\blacksquare_{j}|f_m
| \lambda+\blacksquare_j\rangle \langle \lambda+\blacksquare_j|
f_n| \lambda\rangle\\
=& \frac{(x_{\blacksquare_j})^m (x_{\blacksquare_i})^n}{x_{\blacksquare_j}-x_{\blacksquare_i}}(x_{\blacksquare_j}-x_{\blacksquare_i}+2\hbar)
+ \frac{(x_{\blacksquare_i})^m (x_{\blacksquare_j})^n}{x_{\blacksquare_i}-x_{\blacksquare_j}}(x_{\blacksquare_i}-x_{\blacksquare_j}+2\hbar)
\end{align*}
The operator $f_{s}\circ f_{r+1}  -f_{r+1}\circ f_{s}-f_{s+1}\circ f_{r} +f_{r}\circ f_{s+1} $ action is given by 
\begin{align*}
& \frac{x_{\blacksquare_j}^r x_{\blacksquare_i}^s}{x_{\blacksquare_j}-x_{\blacksquare_i}}\Big( 
x_{\blacksquare_j}(x_{\blacksquare_j}-x_{\blacksquare_i}+2\hbar)
+x_{\blacksquare_j}((x_{\blacksquare_i}-x_{\blacksquare_j}+2\hbar))
-x_{\blacksquare_i}((x_{\blacksquare_j}-x_{\blacksquare_i}+2\hbar))
-x_{\blacksquare_i}((x_{\blacksquare_i}-x_{\blacksquare_j}+2\hbar))
\Big)\\&
+\frac{x_{\blacksquare_i}^r x_{\blacksquare_j}^s}{x_{\blacksquare_i}-x_{\blacksquare_j}}\Big( 
x_{\blacksquare_i}(x_{\blacksquare_i}-x_{\blacksquare_j}+2\hbar)
+x_{\blacksquare_i}((x_{\blacksquare_j}-x_{\blacksquare_i}+2\hbar))
-x_{\blacksquare_j}((x_{\blacksquare_i}-x_{\blacksquare_j}+2\hbar))
-x_{\blacksquare_j}((x_{\blacksquare_j}-x_{\blacksquare_i}+2\hbar))
\Big)\\
=&4\hbar(
x_{\blacksquare_j}^r x_{\blacksquare_i}^s
+x_{\blacksquare_i}^r x_{\blacksquare_j}^s
). 
\end{align*}

The operator $2\hbar(f_{r}\circ f_{s}+f_{s} \circ f_{r})$ acts via the following. 
\begin{align*}
2\hbar(\frac{x_{\blacksquare_j}^r x_{\blacksquare_i}^s}{x_{\blacksquare_j}-x_{\blacksquare_i}}(
2 x_{\blacksquare_j}-2 x_{\blacksquare_i}
)+ 
\frac{x_{\blacksquare_i}^r x_{\blacksquare_j}^s}{x_{\blacksquare_i}-x_{\blacksquare_j}}(
2 x_{\blacksquare_i}-2 x_{\blacksquare_j}
))
=4\hbar(
x_{\blacksquare_j}^r x_{\blacksquare_i}^s
+x_{\blacksquare_i}^r x_{\blacksquare_j}^s). 
\end{align*}
Therefore, we have
\[
f_{s}\circ f_{r+1}  -f_{r+1}\circ f_{s}-f_{s+1}\circ f_{r} +f_{r}\circ f_{s+1} 
=2\hbar(f_{r}\circ f_{s}+f_{s} \circ f_{r}). 
\]
When the two removable boxes are on the same $i$-th string, we denote them by $\blacksquare_i, \tilde{\blacksquare}_i$ with weights $x_{\blacksquare_i}-x_{\tilde\blacksquare_i}=-2\hbar$. 
The verification in this case can be reduced to $w=1$ and there is only one string. 

In this case when there is only one string, the action of $f_m$ is given by 
\begin{align*}
\langle \lambda+\blacksquare|f_m| \lambda\rangle
=&(x_{\blacksquare})^m 
\prod_{ \square\in \lambda}\frac{x_{\blacksquare}-x_{\square}+2\hbar}{ x_{\blacksquare}-x_{\square}}=(x_{\blacksquare})^m \frac{v+1}{2v+1}. 
\end{align*}
Therefore, the operator $f_{n}\circ f_{m}$ is given by
\begin{align*}
& \langle\lambda+\blacksquare_{i}+\tilde\blacksquare_{i}|f_m f_n| \lambda\rangle
=(x_{\tilde\blacksquare_{i}})^m
(x_{\blacksquare_{i}})^n
\frac{v+2}{2v+3} \frac{v+1}{2v+1}
\end{align*}
Factoring out the common factor $\frac{v+2}{2v+3} \frac{v+1}{2v+1}$, the actions of both $-[f_{r+1}, f_{s}]+[f_{r}, f_{s+1}]$ and $2\hbar(f_{r}f_{s}+f_{s}f_{r})$ are given by
\[
2\hbar\big(
(x_{\tilde\blacksquare_{i}})^r(x_{\blacksquare_{i}})^s
+(x_{\tilde\blacksquare_{i}})^s(x_{\blacksquare_{i}})^r
\big). 
\]
\end{proof}

\begin{lemma}
The formulas \eqref{eq:psi} \eqref{eaction} satisfy the relation
\[
[\psi_{r+1}, e_s]-[\psi_{r}, e_{s+1}]=
2\hbar(\psi_{r}e_s+e_s\psi_{r})
\] in \eqref{Y3}. 
\end{lemma}
\begin{proof}
We first write down the action of $e_s\circ \psi_{r}$. 
We have 
\begin{align*}
&\langle
\lambda|
\psi_r | \lambda\rangle\langle\lambda|e_s|\lambda+\blacksquare\rangle
\\
=&
\Res_{z=\infty}
\Big(z^{r}\prod_{\square\in \lambda} \frac{z-x_{\square}+2\hbar }{z-x_{\square}-2\hbar} \prod_{a=1}^w\frac{z-z_a-\ell_a\hbar}{z-z_a+\ell_a\hbar}\Big)
\Res_{z=x_{\blacksquare}} \Big(z^s 
\prod_{\square\in \lambda} \frac{z-x_{\square} }{z-x_{\square}-2\hbar} \prod_{a=1}^w\frac{z-z_a-\ell_a\hbar}{z-z_a+\ell_a\hbar}\Big)\\
=& -\sum_{\tilde{\blacksquare}}
\Big((x_{\tilde{\blacksquare}})^{r}\prod_{\square\in \lambda} \frac{x_{\tilde{\blacksquare}}-x_{\square}+2\hbar }{\widehat{x_{\tilde{\blacksquare}}-x_{\square}-2\hbar}} \prod_{a=1}^w\frac{x_{\tilde{\blacksquare}}-z_a-\ell_a\hbar}{\widehat{x_{\tilde{\blacksquare}}-z_a+\ell_a\hbar}}\Big)
 \Big((x_{\blacksquare})^s 
\prod_{\square\in \lambda} \frac{x_{\blacksquare}-x_{\square} }{\widehat{x_{\blacksquare}-x_{\square}-2\hbar}} \prod_{a=1}^w\frac{x_{\blacksquare}-z_a-\ell_a\hbar}{\widehat{x_{\blacksquare}-z_a+\ell_a\hbar}}\Big)
\end{align*}

The action of $\psi_{r}\circ e_s$ is given by 
\begin{align*}
&\langle\lambda|e_s|\lambda+\blacksquare\rangle \langle
\lambda+\blacksquare|
\psi_r | \lambda+\blacksquare\rangle
\\
=&
\Res_{z=x_{\blacksquare}} \Big(z^s 
\prod_{\square\in \lambda} \frac{z-x_{\square} }{z-x_{\square}-2\hbar} \prod_{a=1}^w\frac{z-z_a-\ell_a\hbar}{z-z_a+\ell_a\hbar}\Big)
\Res_{z=\infty}
\Big(z^{r}\prod_{\square\in \lambda} \frac{z-x_{\square}+2\hbar }{z-x_{\square}-2\hbar} 
\frac{z-x_{\blacksquare}+2\hbar }{z-x_{\blacksquare}-2\hbar} \prod_{a=1}^w\frac{z-z_a-\ell_a\hbar}{z-z_a+\ell_a\hbar}\Big)
\\
=-& \sum_{\tilde{\blacksquare}}\Big((x_{\blacksquare})^s 
\prod_{\square\in \lambda} \frac{x_{\blacksquare}-x_{\square} }{\widehat{x_{\blacksquare}-x_{\square}-2\hbar}} \prod_{a=1}^w\frac{x_{\blacksquare}-z_a-\ell_a\hbar}{\widehat{x_{\blacksquare}-z_a+\ell_a\hbar}}\Big)
\Big((x_{\tilde{\blacksquare}})^{r}\prod_{\square\in \lambda} \frac{x_{\tilde{\blacksquare}}-x_{\square}+2\hbar }{\widehat{x_{\tilde{\blacksquare}}-x_{\square}-2\hbar}}
\frac{x_{\tilde{\blacksquare}}-x_{\blacksquare}+2\hbar }{x_{\tilde{\blacksquare}}-x_{\blacksquare}-2\hbar} 
\prod_{a=1}^w\frac{x_{\tilde{\blacksquare}}-z_a-\ell_a\hbar}{\widehat{x_{\tilde{\blacksquare}}-z_a+\ell_a\hbar}}\Big)
\end{align*}

Therefore, fix a pair $(\blacksquare,  \tilde{\blacksquare})$, we have the common factor 
$$-
\Big(\prod_{\square\in \lambda} \frac{x_{\tilde{\blacksquare}}-x_{\square}+2\hbar }{\widehat{x_{\tilde{\blacksquare}}-x_{\square}-2\hbar}} \prod_{a=1}^w\frac{x_{\tilde{\blacksquare}}-z_a-\ell_a\hbar}{\widehat{x_{\tilde{\blacksquare}}-z_a+\ell_a\hbar}}\Big)
 \Big(
\prod_{\square\in \lambda} \frac{x_{\blacksquare}-x_{\square} }{\widehat{x_{\blacksquare}-x_{\square}-2\hbar}} \prod_{a=1}^w\frac{x_{\blacksquare}-z_a-\ell_a\hbar}{\widehat{x_{\blacksquare}-z_a+\ell_a\hbar}}\Big).$$
Factoring out this common factor, the operator $e_s\circ \psi_{r}+\psi_r \circ e_{s}$ corresponds to 
\[
(x_{\tilde{\blacksquare}})^{r}(x_{\blacksquare})^s (\frac{x_{\tilde{\blacksquare}}-x_{\blacksquare}+2\hbar }{x_{\tilde{\blacksquare}}-x_{\blacksquare}-2\hbar} +1 )=(x_{\tilde{\blacksquare}})^{r}(x_{\blacksquare})^s (\frac{2(x_{\tilde{\blacksquare}}-x_{\blacksquare})}{x_{\tilde{\blacksquare}}-x_{\blacksquare}-2\hbar} )
\]
Similarly, factoring out the same common factor, the operator $e_s\circ \psi_{r+1}-\psi_{r+1}\circ e_s-e_{s+1}\circ \psi_{r}+\psi_{r}\circ e_{s+1}$ corresponds to the following. 
\begin{align*}
&(x_{\tilde{\blacksquare}})
^{r+1}(x_{\blacksquare})^s-(x_{\tilde{\blacksquare}})
^{r+1}(x_{\blacksquare})^s \frac{x_{\tilde{\blacksquare}}
-x_{\blacksquare}+2\hbar }{x_{\tilde{\blacksquare}}-x_{\blacksquare}-2\hbar}
-(x_{\tilde{\blacksquare}})
^{r}(x_{\blacksquare})^{s+1}
+
(x_{\tilde{\blacksquare}})
^{r}(x_{\blacksquare})^{s+1}
\frac{x_{\tilde{\blacksquare}}
-x_{\blacksquare}+2\hbar }{x_{\tilde{\blacksquare}}-x_{\blacksquare}-2\hbar}\\
=& (x_{\tilde{\blacksquare}})
^{r}(x_{\blacksquare})^s \Big(
x_{\tilde{\blacksquare}}(1-\frac{x_{\tilde{\blacksquare}}-x_{\blacksquare}+2\hbar }{x_{\tilde{\blacksquare}}-x_{\blacksquare}-2\hbar})
-x_{\blacksquare} (1-\frac{x_{\tilde{\blacksquare}}-x_{\blacksquare}+2\hbar }{x_{\tilde{\blacksquare}}-x_{\blacksquare}-2\hbar})\Big)
\\
=&(x_{\tilde{\blacksquare}})
^{r}(x_{\blacksquare})^s (x_{\tilde{\blacksquare}}-x_{{\blacksquare}} )\frac{-4\hbar}{x_{\tilde{\blacksquare}}-x_{\blacksquare}-2\hbar}. 
\end{align*}
This shows the relation
\[
e_s\circ \psi_{r+1}-\psi_{r+1}\circ e_s-e_{s+1}\circ \psi_{r}+\psi_{r}\circ e_{s+1}
=-2\hbar(e_s\circ \psi_{r}+\psi_r \circ e_{s}). 
\]
\end{proof}
\begin{lemma}
The formulas \eqref{eq:psi} \eqref{faction} satisfy the relation: 
\[
[\psi_{r+1}, f_s]-[\psi_{r}, f_{s+1}]=
-2\hbar(\psi_{r}f_s+f_s\psi_{r})
\]
\end{lemma}
\begin{proof}
For the operator $f_{s}\circ \psi_{r}$, 
we have 
\begin{align*}
&\langle
\lambda+\blacksquare|
\psi_r | \lambda+\blacksquare\rangle
\langle \lambda+\blacksquare|f_s| \lambda\rangle
\\
=-& \sum_{\tilde{\blacksquare}}
\Big((x_{\tilde{\blacksquare}})^{r}\prod_{\square\in \lambda} \frac{x_{\tilde{\blacksquare}}-x_{\square}+2\hbar }{\widehat{x_{\tilde{\blacksquare}}-x_{\square}-2\hbar}}
\frac{x_{\tilde{\blacksquare}}-x_{\blacksquare}+2\hbar }{x_{\tilde{\blacksquare}}-x_{\blacksquare}-2\hbar} 
\prod_{a=1}^w\frac{x_{\tilde{\blacksquare}}-z_a-\ell_a\hbar}{\widehat{x_{\tilde{\blacksquare}}-z_a+\ell_a\hbar}}\Big)
(x_{\blacksquare})^s 
\prod_{ \square\in \lambda}\frac{x_{\blacksquare}-x_{\square}+2\hbar}{ x_{\blacksquare}-x_{\square}}. 
\end{align*}
and for the operator $\psi_{r}\circ f_{s}$ we have
\begin{align*}
&
\langle \lambda+\blacksquare|f_s| \lambda\rangle\langle
\lambda|
\psi_r | \lambda\rangle
\\
=-& \sum_{\tilde{\blacksquare}}
\Big((x_{\tilde{\blacksquare}})^{r}\prod_{\square\in \lambda} \frac{x_{\tilde{\blacksquare}}-x_{\square}+2\hbar }{\widehat{x_{\tilde{\blacksquare}}-x_{\square}-2\hbar}}
\prod_{a=1}^w\frac{x_{\tilde{\blacksquare}}-z_a-\ell_a\hbar}{\widehat{x_{\tilde{\blacksquare}}-z_a+\ell_a\hbar}}\Big)
(x_{\blacksquare})^s 
\prod_{ \square\in \lambda}\frac{x_{\blacksquare}-x_{\square}+2\hbar}{ x_{\blacksquare}-x_{\square}}. 
\end{align*}
Fix a pair $(\blacksquare,  \tilde{\blacksquare})$, we have the common factor 
\[
\Big(\prod_{\square\in \lambda} \frac{x_{\tilde{\blacksquare}}-x_{\square}+2\hbar }{\widehat{x_{\tilde{\blacksquare}}-x_{\square}-2\hbar}}
\prod_{a=1}^w\frac{x_{\tilde{\blacksquare}}-z_a-\ell_a\hbar}{\widehat{x_{\tilde{\blacksquare}}-z_a+\ell_a\hbar}}\Big)
\prod_{ \square\in \lambda}\frac{x_{\blacksquare}-x_{\square}+2\hbar}{ x_{\blacksquare}-x_{\square}}. 
\]
The operator $\psi_{r}f_s+f_s\psi_{r}$ corresponds to the following 
\begin{align*}
(x_{\tilde{\blacksquare}})^{r}(x_{\blacksquare})^s (\frac{x_{\tilde{\blacksquare}}-x_{\blacksquare}+2\hbar }{x_{\tilde{\blacksquare}}-x_{\blacksquare}-2\hbar} +1 )=(x_{\tilde{\blacksquare}})^{r}(x_{\blacksquare})^s (\frac{2(x_{\tilde{\blacksquare}}-x_{\blacksquare})}{x_{\tilde{\blacksquare}}-x_{\blacksquare}-2\hbar} )
\end{align*}

Similarly, factoring out the same common factor, the operator $f_s\circ \psi_{r+1}-\psi_{r+1}\circ f_s-f_{s+1}\circ \psi_{r}+\psi_{r}\circ f_{s+1}$ corresponds to the following. 
\begin{align*}
&(x_{\tilde{\blacksquare}})
^{r+1}(x_{\blacksquare})^s \frac{x_{\tilde{\blacksquare}}-x_{\blacksquare}+2\hbar }{x_{\tilde{\blacksquare}}-x_{\blacksquare}-2\hbar}
-(x_{\tilde{\blacksquare}})
^{r+1}(x_{\blacksquare})^s-(x_{\tilde{\blacksquare}})
^{r}(x_{\blacksquare})^{s+1}
\frac{x_{\tilde{\blacksquare}}-x_{\blacksquare}+2\hbar }{x_{\tilde{\blacksquare}}-x_{\blacksquare}-2\hbar}+(x_{\tilde{\blacksquare}})
^{r}(x_{\blacksquare})^{s+1}\\
=& (x_{\tilde{\blacksquare}})
^{r}(x_{\blacksquare})^s \Big(
x_{\tilde{\blacksquare}}(\frac{x_{\tilde{\blacksquare}}-x_{\blacksquare}+2\hbar }{x_{\tilde{\blacksquare}}-x_{\blacksquare}-2\hbar}-1)
-x_{\blacksquare} (\frac{x_{\tilde{\blacksquare}}-x_{\blacksquare}+2\hbar }{x_{\tilde{\blacksquare}}-x_{\blacksquare}-2\hbar}-1)\Big)
\\
=&(x_{\tilde{\blacksquare}})
^{r}(x_{\blacksquare})^s (x_{\tilde{\blacksquare}}-x_{{\blacksquare}} )\frac{4\hbar}{x_{\tilde{\blacksquare}}-x_{\blacksquare}-2\hbar}. 
\end{align*} 
This shows the relation
\[
f_s\circ \psi_{r+1}-\psi_{r+1}\circ f_s-f_{s+1}\circ \psi_{r}+\psi_{r}\circ f_{s+1}
=2\hbar(\psi_{r}f_s+f_s\psi_{r}). 
\]
\end{proof}

\begin{lemma}
The formulas \eqref{eq:psi} \eqref{eaction}\eqref{faction} satisfy the relation \eqref{Y5}. 
\end{lemma}
\begin{proof}

Notations as before. Let $\lambda$ be a $\C^*$-fixed point. By \eqref{eaction} and \eqref{faction}, for the composition $e_i\circ f_j$, we have
\begin{align*}
&\langle\lambda+\blacksquare|(f_j)| \lambda\rangle \langle\lambda|(e_i)| \lambda+\blacksquare\rangle \\
=&
(x_{\blacksquare})^j 
\prod_{ \square\in \lambda}\frac{x_{\blacksquare}-x_{\square}+2\hbar}{ x_{\blacksquare}-x_{\square}}
\Res_{z=x_{\blacksquare}} \big(z^i 
\prod_{\square\in \lambda} \frac{z-x_{\square} }{z-x_{\square}-2\hbar} \prod_{j=1}^w\frac{z-z_j-\ell_j\hbar}{z-z_j+\ell_j\hbar}\Big)\\
=&
\Res_{z=x_{\blacksquare}} \big(z^{i+j} 
\prod_{\square\in \lambda} \frac{(z-x_{\square}+2\hbar)(z-x_{\square}) }{(z-x_{\square})(z-x_{\square}-2\hbar)} \prod_{j=1}^w\frac{z-z_j-\ell_j\hbar}{z-z_j+\ell_j\hbar}\Big)\\
=& \Res_{z=x_{\blacksquare}} \big(z^{i+j} \prod_{\square\in \lambda} \frac{z-x_{\square}+2\hbar }{z-x_{\square}-2\hbar} 
 \prod_{j=1}^w 
\frac{z-z_j-\ell_j\hbar}{z-z_j+\ell_j\hbar}\Big). 
\end{align*}
On the other hand, for the composition $f_j\circ e_i$, we have
\begin{align*}
&\langle\lambda-\blacksquare|(e_i)| \lambda\rangle
\langle\lambda|(f_j)| \lambda-\blacksquare\rangle \\
=&\Res_{z=x_{\blacksquare}} \big(z^i 
\prod_{\square\in \lambda-\blacksquare} \frac{z-x_{\square} }{z-x_{\square}-2\hbar} \prod_{j=1}^w\frac{z-z_j-\ell_j\hbar}{z-z_j+\ell_j\hbar}\Big)
(x_{\blacksquare})^j 
\prod_{ \square\in \lambda}\frac{x_{\blacksquare}-x_{\square}+2\hbar}{ x_{\blacksquare}-x_{\square}}
\\
=&\Res_{z=x_{\blacksquare}} \big(z^{i+j}
\prod_{\square\in \lambda-\blacksquare} \frac{(z-x_{\square})(z-x_{\square}+2\hbar) }{(z-x_{\square}-2\hbar)(z-x_{\square})} \prod_{j=1}^w\frac{z-z_j-\ell_j\hbar}{z-z_j+\ell_j\hbar}\Big)\\
=&\Res_{z=x_{\blacksquare}} \big(z^{i+j}
\prod_{\square\in \lambda-\blacksquare} \frac{z-x_{\square}+2\hbar }{z-x_{\square}-2\hbar} \prod_{j=1}^w\frac{z-z_j-\ell_j\hbar}{z-z_j+\ell_j\hbar}\Big)\\
=&-\Res_{z=x_{\blacksquare}} \big(z^{i+j}
\prod_{\square\in \lambda} \frac{z-x_{\square}+2\hbar }{z-x_{\square}-2\hbar} \prod_{j=1}^w\frac{z-z_j-\ell_j\hbar}{z-z_j+\ell_j\hbar}\Big). 
\end{align*}
This implies that, for the commutator $-[e_i, f_j]$, we have
\begin{align}
-[e_i, f_j]_{\lambda}
=&(f_j\circ e_i)|_{\lambda}
-( e_i\circ f_j)|_{\lambda}\notag\\
=&-\sum_{\text{removable boxes}} 
  \Res_{z=x_{\blacksquare}} \big(z^{i+j}
\prod_{\square\in \lambda} \frac{z-x_{\square}+2\hbar }{z-x_{\square}-2\hbar} \prod_{j=1}^w\frac{z-z_j-\ell_j\hbar}{z-z_j+\ell_j\hbar}\Big)\notag\\&
-\sum_{\text{addible boxes}}
\Res_{z=x_{\blacksquare}} \Big(z^{i+j}
\prod_{\square\in \lambda} \frac{z-x_{\square}+2\hbar }{z-x_{\square}-2\hbar} \prod_{j=1}^w\frac{z-z_j-\ell_j\hbar}{z-z_j+\ell_j\hbar}\Big)\notag\\
=& -\sum_{\text{addible boxes $\cup$ removable boxes}} 
\Res_{z=x_{\blacksquare}} \Big(z^{i+j}
\prod_{\square\in \lambda} \frac{z-x_{\square}+2\hbar }{z-x_{\square}-2\hbar} \prod_{j=1}^w\frac{z-z_j-\ell_j\hbar}{z-z_j+\ell_j\hbar}\Big)\notag\\
=&  \Res_{z=\infty}\Big(z^{i+j}
\prod_{\square\in \lambda} \frac{z-x_{\square}+2\hbar }{z-x_{\square}-2\hbar} \prod_{j=1}^w\frac{z-z_j-\ell_j\hbar}{z-z_j+\ell_j\hbar}\Big)\label{eq:residue}\\
=& \Res_{z=\infty}z^{i+j}\tilde{\psi}_(z)|_{\lambda}. \notag
\end{align}
where the equality \eqref{eq:residue} follows from the residue theorem and Lemma \ref{lem:psi}.

As the expansion of $\psi(z)$ is given by $1+2\hbar\sum_{i\geq 0} \psi_i z^{-i-1}$. We have
\[
-\frac{1}{2\hbar}\Res_{z=\infty}z^{i+j}\psi_(z)|_{\lambda}=\ \psi_{i+j}|_{\lambda}. 
\]
Therefore, 
$
[e_i, f_j]=2\hbar\psi_{i+j}. 
$ 
This implies the relation \eqref{Y5}.
\end{proof}

\subsection{The proof of Proposition \ref{prop:w=1}}
\label{sub:proof of prop}
When $w=1$, denote the maximal torus $T_w$ at the framing vertex by $\C^*_{fr}$. By the Proposition \ref{prop:dimred} and Lemma \ref{lem:w=1}, we have the isomorphims as graded vector spaces
\begin{align*}
\bigoplus_{v=0}^{\infty} H^*_{c,\C^*\times \C^*_{fr}}(\tilde{X}(v, 1)^{st}/G_V, \varphi_{\bf{w}})^{\vee}
&\cong \bigoplus_{v=0}^{\infty} H^*_{\C^*\times \C^*_{fr}}(\mathfrak{M}^+(v, 1), \C)&&\text{by Proposition \ref{prop:dimred}}\\
&
=\bigoplus_{v=0}^{\ell} H^*_{c,\C^*\times \C^*_{fr}}(\pt, \C)^{\vee} &&\text{by Lemma \ref{lem:w=1}}\\
&=\C[\hbar, z]\otimes{\C^{\ell+1}}
\end{align*}

We now analyze the action of $Y_\hbar(\mathfrak{sl}_2)$ on $\C[\hbar, z]\otimes{\C^{\ell+1}}$. 

By formula \eqref{eaction}, we have
\begin{align*}
\langle\lambda|e_k| \lambda+\blacksquare\rangle=
&\Res_{u=x_{\blacksquare}}  u^k\frac{u-z_1-\ell\hbar}{u-z_1+\ell\hbar-2v\hbar}  \\
=&(z_1-\ell\hbar+2v\hbar )^k \big((z_1-\ell\hbar+2v\hbar)-z_1-\ell\hbar\big)\\
=&(z_1-\ell\hbar+2v\hbar )^k \big(2(v-\ell)\hbar\big). 
\end{align*}
Choose a basis $\{b_{v}\mid v=0, 1, 2, \cdots, \ell\}$ of $\C^{\ell+1}$ such that 
$b_v$ comes from the component $\mathfrak{M}^+(v, 1)$. In other words, $b_v$ is represented by the following picture with $v$ many $\bullet$'s. 
\[
\begin{tikzpicture}[scale=0.7]
\tikzstyle{every node}=[font=\small]
\node at (1, 4) {$\bullet$};	
\node at (1, 3) {$\vdots$};	
\foreach \x in {1}
{\node at (\x, 0) {$\square$};}	
\foreach \x in {1}
{\node at (\x, 1) {$\bullet$};}	
\foreach \x in {1}
{\node at (\x, 2) {$\bullet$};}	
\foreach \x in {1}
{\draw[->, thick] (\x, 0.1) -- (\x, 0.9) ;}
\foreach \x in {1}
{\draw[->, thick] (1, 1.1) -- (1, 1.9) ;}
\draw[->, thick] (1, 3.1) -- (1, 3.9) ;
\end{tikzpicture}
\]
Then, for $\lambda=b_v$, we have $\lambda+\blacksquare=b_{v+1}$. 
Therefore, we have for $0\leq v \leq \ell$, 
\begin{align*}
&e_k(b_{0})=-2\ell\hbar (z_1-\ell \hbar)^k b_1, \cdots \\
& e_k(b_{v})=(z_1-\ell\hbar+2v\hbar )^k \big(2(v-\ell)\hbar b_{v+1}, \cdots, \\
& e_k(b_{\ell})=0. 
\end{align*}
By formula \eqref{faction}, the action of $f_m$ is given by 
\begin{align*}
\langle \lambda+\blacksquare|f_m| \lambda\rangle
=&(x_{\blacksquare})^m 
\prod_{ \square\in \lambda}\frac{x_{\blacksquare}-x_{\square}+2\hbar}{ x_{\blacksquare}-x_{\square}}. 
\end{align*}
Similar as before, for $\lambda=b_v$, we have $\lambda+\blacksquare=b_{v+1}$.
Therefore, we have
\begin{align*}
&f_m(b_0)=0, \\
&f_m(b_1)=(z_1-\ell\hbar)^m b_0,  \cdots, \\
&f_m(b_{v+1})=(z_1-\ell\hbar+2v\hbar)^m (v+1)b_v, v\leq \ell. 
\end{align*}
Clearly, the module $\C[\hbar, z]\otimes{\C^{\ell+1}}$ of $Y_{\hbar}(\mathfrak{sl}_2)$ is an irreducible $\ell+1$-dimensional representation. Therefore, it is isomorphic to $ev^*(L(\ell))$. 

Comparing with the convention in \cite[Chapter 12.1, section E]{CP}, we have
 $\hbar=-\frac{1}{2}$, and the basis $\{v_0, v_1, \cdots, v_{\ell}\}$ in \cite[Chapter 12.1, section E]{CP} is given by
 \[
 v_s:={\ell\choose s} b_{\ell-s}, s=0, 1, 2, \cdots, \ell. 
 \]
The element $X_{1}^{-}$ is $f_{0}$ and $X_{1}^{+}$ is $e_{0}$. Therefore, setting $\hbar=-\frac{1}{2}$, as a representation of $\mathfrak{sl}_2$, we have
\[
e_0(v_s)=(\ell-s+1) v_{s-1},\,\  f_0(v_s)=(s+1)v_{s+1}, s=0, 1, 2, \cdots, \ell. 
\]
Furthermore, the equality 
\begin{align*} 
&e_k(b_{\ell-s})=(z_1+\ell\hbar-2s\hbar )^k \big(2(-s)\hbar b_{\ell-s+1}=(z_1-\frac{\ell}{2}+s)^k s  b_{\ell-s+1}
\end{align*}
implies that 
\begin{align} 
&e_k(v_s)=(z_1-\frac{\ell}{2}+s)^k (\ell-s+1)  v_{s-1}. \label{e_k}
\end{align}
The Drinfeld polynomial in the current paper is given by (see Lemma \ref{lem:psi})
\begin{align*}
P(z)=&\prod_{s=0}^{\ell-1}(z-z_1+(\ell-2s)\hbar). 
\end{align*}
so that, by the computation in Lemma \ref{lem:psi}, we have $${\psi}(z)|_{b_{\ell}}=\frac{P(z-2\hbar)}{P(z)}b_{\ell}.$$ 
The convention of the Drinfeld polynomial in \cite[Chapter 12.1]{CP} is given by $P_1(u)=\prod_{s=0}^{\ell-1}(u-a-\frac{1}{2}l+\frac{1}{2}+s)\hbar).$ Therefore, the formula \eqref{e_k} and the formula \cite[Chapter 12.1 (27)]{CP} are slightly different. 

\begin{corollary}\label{cor:tensor product}
We have the following isomorphism as representations of $Y_{\hbar}(\mathfrak{sl}_2)$
\[
\bigoplus_{v\in \N} H^*_{c, \C^*\times T_w}(\tilde{X}(v, \vec{1}_w)^{st}/G_v, \varphi_{\bf{w}})^{\vee}\otimes_R K
\cong \ev^*_{z_1}(L(\ell_1))\otimes \cdots \otimes \ev^*_{z_w}(L(\ell_w)), 
\]
where the tensor product is the fusion tensor product which comes from the Drinfeld coproduct of $Y_{\hbar}(\mathfrak{sl}_2)$. 
\end{corollary}
\begin{proof}
Let $\lambda=(v_1, \ldots, v_w)\in S(\ell_1, \ldots, \ell_w)$ be the torus fixed point. Denote by $\lambda_j=(v_j)$  the $j$-th column of the picture \eqref{pic:fixed points}. In particular, $\lambda_j\in S(\ell_j)$ is a torus fixed point of $\mathfrak{M}^+(v_j, 1)$. 

Define the vector space isomorphism 
\begin{align*}
\bigoplus_{v\in \N} H^*_{c, \C^*\times T_w}(\tilde{X}(v, \vec{1}_w)^{st}/G_v, \varphi_{\bf{w}})^{\vee}\otimes_R K
& \to  \ev^*_{z_1}(L(\ell_1))\otimes \cdots \otimes \ev^*_{z_w}(L(\ell_w)), \\ \lambda& \mapsto \lambda_1\otimes \cdots \otimes \lambda_w. 
\end{align*}
We now rewrite the formulas in Proposition \ref{prop:ef formula} and Proposition \ref{prop:psi(z)}.
Let $\blacksquare_j$ be the addible box at column $j$. Note that the weights $x_{\square}$ are given by the formula \eqref{eq:weights of fix}. 
For any regular function $g(z)$, 
we have  
\begin{align*}
&\langle\lambda|(e(g))| \lambda+\blacksquare_j\rangle\\
=
	&  \Res_{z=x_{\blacksquare_j}} \Big(g(z) 
\prod_{\square\in \lambda} \frac{z-x_{\square} }{z-x_{\square}-2\hbar} \prod_{j=1}^w\frac{z-z_j-\ell_j\hbar}{z-z_j+\ell_j\hbar}\Big), 
	\\
 =&\Res_{z=x_{\blacksquare_j}} \Big(g(z) 
\prod_{\square\in \lambda_j} \frac{z-x_{\square} }{z-x_{\square}-2\hbar} \frac{z-z_j-\ell_j\hbar}{z-z_j+\ell_j\hbar}
\prod_{\square\in \lambda\setminus \lambda_j} \frac{z-x_{\square} }{z-x_{\square}-2\hbar}
\prod_{i\in [1, w]\setminus j}\frac{z-z_i-\ell_i\hbar}{z-z_i+\ell_i\hbar}\Big)\\
=&\Big(\prod_{\square\in \lambda\setminus \lambda_j} \frac{x_{\blacksquare_j}-x_{\square} }{x_{\blacksquare_j}-x_{\square}-2\hbar}
\prod_{i\in [1, w]\setminus j}\frac{x_{\blacksquare_j}-z_i-\ell_i\hbar}{x_{\blacksquare_j}-z_i+\ell_i\hbar}\Big) \langle\lambda_j|(e(g))| \lambda_j+\blacksquare_j\rangle
\end{align*}
Therefore, the action of $e(g)$ on the tensor $\lambda_1\otimes \cdots \lambda_w$ is given by
\[
e(g)(\lambda_1\otimes \cdots \lambda_w)=\sum_{j=1}^w \Big(\prod_{\square\in \lambda\setminus \lambda_j} \frac{x_{\blacksquare_j}-x_{\square} }{x_{\blacksquare_j}-x_{\square}-2\hbar}
\prod_{i\in [1, w]\setminus j}\frac{x_{\blacksquare_j}-z_i-\ell_i\hbar}{x_{\blacksquare_j}-z_i+\ell_i\hbar} \big(\lambda_1\otimes \cdots \otimes \lambda_{j-1}\otimes  e(g)(\lambda_j) \otimes \lambda_{j+1}\otimes \cdots \otimes \lambda_w\big)\Big). 
\]
Similarly, we have
	\begin{align*}
	\langle \lambda+\blacksquare_j|(f(g))| \lambda\rangle=&\prod_{ \square\in \lambda\setminus \lambda_j}\frac{x_{\blacksquare_j}-x_{\square}+2\hbar}{x_{\blacksquare_j}-x_{\square}}
\Big(g(z)
\langle \lambda_j+\blacksquare_j|(f(g))| \lambda_j\rangle\Big).
\end{align*}
Thus, the action of $f(g)$ on the tensor $\lambda_1\otimes \cdots \lambda_w$ is given by
\[
f(g)(\lambda_1\otimes \cdots \lambda_w)=\sum_{j=1}^w \prod_{ \square\in \lambda\setminus \lambda_j}\frac{x_{\blacksquare_j}-x_{\square}+2\hbar}{x_{\blacksquare_j}-x_{\square}} \big(\lambda_1\otimes \cdots \otimes \lambda_{j-1}\otimes  f(g)(\lambda_j) \otimes \lambda_{j+1}\otimes \cdots \otimes \lambda_w\big). 
\]
By Proposition \ref{prop:psi(z)}, the action of $\psi(z)$ on the tensor $\lambda_1\otimes \cdots \lambda_w$ is given by
\[
\psi(z)(\lambda_1\otimes \cdots \lambda_w)
=\psi(z)(\lambda_1)\otimes \cdots \otimes \psi(z)(\lambda_w). 
\]
Comparing with the formulas in \cite[\S 3.4, example]{H07} (the case of quantum affine algebra) and \cite{YZ16} (the case of Yangian), the action of Yangian on the tensor product $\ev^*_{z_1}(L(\ell_1))\otimes \cdots \otimes \ev^*_{z_w}(L(\ell_w))$ comes from the Drinfeld coproduct of $Y_{\hbar}(\mathfrak{sl}_2)$. 
\end{proof}

\subsection{Geometric interpretation}
In this subsection, we discuss the geometric interpretation of the formulas \eqref{eaction} and \eqref{faction}. 
\subsubsection{Torus fixed points on $\tilde{X}(v, \vec{1}_w)^{st}/G_v$}
In this subsection, we compute the torus fixed points under $\C^*\times T_{w}$ on the ambient space $\tilde{X}(v, \vec{1}_w)^{st}/G_v$ similar as the proof of Lemma \ref{lem:fixed points}. 

\begin{lemma}\label{fixpointamb}
The set of torus fixed points $(\tilde{X}(v, \vec{1}_w)^{st}/G_v)^{\C^*\times T_{w}}$ is given by 
\begin{equation*}
\coprod_{\{(v_1, v_2, \cdots, v_w)\in \N^w \mid \sum_{i=1}^wv_i=v\}}\big(\prod_{i=1}^w\Hom(V(-\ell_i), W_i)\Big),   
\end{equation*}
where $V(-\ell_i)$ is the 
$-\ell_i$-weight subspace of 
$V$ for the 
$\C^*$-action. In particular, if $v_i\leq \ell_i$, for all $i\in [1, w]$, then,  $\prod_{i=1}^w\Hom(V(-\ell_i), W_i)$ is a point.
\end{lemma}
\begin{proof}

For any $(\epsilon, a^*_1, \ldots, a_w^*, a_1, \ldots, a_w) \in (\tilde{X}(v, \vec{1}_w)^{st}/G_v)^{\C^*}$ and $t\in \C^*$, then 
\[
t\cdot (\epsilon, a^*_1, \ldots, a_w^*, a_1, \ldots, a_w)
=(t^{-2} \epsilon, t^{\ell_1}a^*_1, \ldots, t^{\ell_w}a_w^*, t^{\ell_1}a_1, \ldots, t^{\ell_w}a_w). 
\]
Therefore, there exists a group homomorphism $\rho: \C^*\to G_v, t\mapsto \rho(t)$, such that 
\[
(\rho(t)\epsilon \rho(t)^{-1}, \rho(t)a_i^*, a_i\rho(t)^{-1})=
(t^{-2} \epsilon, t^{\ell_i}a^*_i, t^{\ell_i}a_i). 
\]
We decompose $V$ into eigenspaces of $\rho(t)$: 
\[
V=\oplus_{n\in \Z} V(n), \text{where $V(n)=\{v\in V\mid \rho(t)\cdot v=t^n v\}$. }
\]
From the proof of Lemma \ref{lem:fixed points},  
we have  $\im(a_i^*)\subset V({\ell_i})$ and  $\epsilon(V(n))\subset V(n-2)$. 

For the linear maps $a_i: V\to W_i$, we have for any $v\in V(n)$
\[
a_i\rho(t)^{-1}(v)=t^{-n} a_i(v)=t^{\ell_i}a_i. 
\]
Therefore, $a_i: V(-\ell_i)\to W_i$, and $a_i(V(n))=0$, if $n\neq -\ell_i$.

By the stability condition, $V$ is generated by the image of $a_i^*$ and the operator $\epsilon$. Fix a basis of $W=\oplus_i W_i$. 
The point $(\epsilon, a^*_1, \ldots, a_w^*)$ is invariant under $T_w$, then we can choose a basis of $V$, such that $\im(a_i^*), \im(\epsilon a_i^*), \im(\epsilon^2 a_i^*), \cdots $ are all coordinate quotient spaces of $W$. 

For each $i=1, \ldots, w$, we have the following chain
\[
\xymatrix{
W_i \ar@{->>}[r]^{a_{i}^*}& 
V({\ell_i}) \ar@{->>}[r]^{\epsilon} & V({\ell_i-2})
\ar@{->>}[r]^{\epsilon} & V({\ell_i-4}) \ar@{->>}[r]^{\epsilon} & \cdots  
}
\]
The above chain terminates at $V({\ell_i-(2v_i-2)})$, for some $v_i$. Now unlike in Lemma \ref{lem:fixed points}, $v_i$ may be greater than $\ell_i$. 
Therefore, the quotient $(\{\epsilon, a_1^*, \cdots, a_w^*\}/G_v)^{\C^*\times T_w}$ is a point, which is illustrated as in Picture \eqref{pic:fixed points}. 

For any linear maps $a_i: V(-\ell_i)\to W_i$, if $v_i < \ell_i$ for some $i\in [1, w]$, we have $V(-\ell_i)=0$, therefore, $a_i$ is the zero map. If $v_i\geq \ell_i$, $a_i$ is automatically $T_w$-invariant. 
Therefore, the fixed points are labeled by the set
\[
(v_1, v_2, \cdots, v_w) \mid  \sum_{i=1}^wv_i=v\}
\]
For fixed component $(v_1, v_2, \cdots, v_w)$, the fixed point set is isomorphic to  $\prod_{i=1}^w\Hom(V(-\ell_i), W_i)$, which parametrizes linear maps $a_i$. 
\end{proof}
\begin{example}
Let $w=1$, $v=2$ and $\ell=1$. Then, $W=\C$ and $V$ has two eigenspaces $V(1)\cong \C$ and $V(-1)\cong \C$.
Let $(\epsilon, a^*, a)\in \tilde{X}(2, 1)^{st}/\GL_2$ be fixed by $C^*$. Then, we have  $a^*: W\surj V(1)$, $\epsilon: V(1)\surj V(-1)$ and $a: V(-1)\to W$ are three complex numbers with $a^*, \epsilon$ non-zero.
The 3-dimensional torus $T_w\times \GL(V(1))\times \GL(V(-1))$ acts by: 
\[
(t, d_1, d_2)\cdot (a^*, \epsilon, a)=(d_1a^*t^{-1}, d_2\epsilon d_1^{-1},  tad_2^{-1}). 
\]
Up to the action of $\GL(V(1))\times \GL(V(-1))$, $T_w$ acts trivially. Therefore, we have
\[
(\tilde{X}(2, 1)^{st}/\GL_2)^{\C^*}
\cong 
(\tilde{X}(2, 1)^{st}/\GL_2)^{\C^*\times T_w}\cong (\C^*\times \C^*\times \C)/(\GL(V(1))\times \GL(V(-1)))\cong \C. 
\]
\end{example}
\begin{remark}
The fixed point $(\tilde{X}(v, \vec{1}_w)^{st} / G_v)^{\C^* \times T_w}$ is not compact, and the restriction 
$\varphi_{\mathbf{w}}|_{(\tilde{X}(v, \vec{1}_w)^{st} / G_v)^{\C^* \times T_w}}$ does not vanish. 
Therefore, our setup is not the equivariantly compact-type (see \cite[Setting~5.13]{CZ23}). It is the geometric phase as described in \cite[Setting~5.15]{CZ23}. 
\end{remark}

\subsubsection{}  
By abuse of notation we denote by $V$ in this section the vector space of dimension $v+1$. Fix $\xi\subset V$, a one dimensional subspace and let $V_2:=V/\xi$ be the quotient which has dimension $v$. 
Consider the following correspondence 
\[
\xymatrix@R=1em@C=0.5em{
	&\tilde{X}^{st}(v, v+1, \vec{1}_w)\ar[ld]_{p}\ar[rd]^{q}&\\
	\tilde{X}^{st}(v, \vec{1}_w)&&\tilde{X}^{st}(v+1, \vec{1}_w)
}\]
where 
\[
\tilde{X}^{st}(v, v+1, \vec{1}_w)=\{(\epsilon, a^*_1, \ldots, a_w^*, a_1, \ldots, a_w)\in  \tilde{X}(v+1, \vec{1}_w)^{st} \mid \epsilon(\xi)\subset \xi\}, a_i(\xi)=0, \forall i\in [1, w]\}. 
\]
The parabolic subgroup $P:=\{x\in G_{v+1}=\GL(V)\mid x(\xi)\subset \xi\}$ acts on $\tilde{X}^{st}(v, v+1, \vec{1}_w)$. 
The map $p$ is given by $(\epsilon, a^*_1, \ldots, a_{w}^*, a_1, \ldots, a_{w})\mapsto (\epsilon, a^*_1, \ldots, a_{w}^*, a_1, \ldots, a_{w})$ mod $\xi$, where $\epsilon$ mod $\xi$ is the induced map $V/\xi\to V/\xi$, and $a^*_i$ mod $\xi$ is the composition of 
$a^*_i: W_i \to V$ with the projection $V \surj V_2=V/\xi$. The map $q$ is the natural inclusion.

The subspace $\xi\subset V$ gives a tautological line bundle on the correspondence $\tilde{X}^{st}(v, v+1, \vec{1}_w)/P$, which will be denoted by $\calL$. 
Recall that the Levi subgroup of $P$ is $\C^*\times G_v$ where 
$\xi$ above is the standard weight 1 representation of the $\C^*$-factor. Hence, $c_1(\calL)$ coincides with the equivariant variable of the $\C^*$-factor.

For any regular function $g(x)$ in one variable with coefficients in $H_T(\pt)$, we have the class $g(c_1(\calL))$ in $H^*_{c,\C^*\times T_w}(\tilde{X}^{st}(v, v+1, \vec{1}_w)/P, \varphi_{\bold{w}})^{\vee}$. 

Define the actions of the operators 
\begin{align}\label{eqn:EFdef}
&\tilde{e}(g): H^*_{c,\C^*\times T_w}(\tilde{X}^{st}(v, \vec{1}_w)/G_v, \varphi_{\bold{w}})^{\vee} \to H^*_{c,\C^*\times T_w}(\tilde{X}^{st}(v+1, \vec{1}_w)/G_{v+1}, \varphi_{\bold{w}})^{\vee}\\
&\tilde{f}(g):H^*_{c,\C^*\times T_w}(\tilde{X}^{st}(v+1, \vec{1}_w)/G_{v+1}, \varphi_{\bold{w}})^{\vee}  \to H^*_{c,\C^*\times T_w}(\tilde{X}^{st}(v, \vec{1}_w)/G_{v}, \varphi_{\bold{w}})^{\vee}
\end{align}
 by the following convolutions: 
\[
\tilde{e}(g):=q_*(g(c_1(\calL))\cup p^*), \,\  \tilde{f}(g):=p_*(g(c_1(\calL))\cup q^*). 
\]
By Lemma \ref{fixpointamb}, we have
\[
S(\ell_1, \ldots, \ell_w)\subsetneq (\tilde{X}^{st}(v, \vec{1}_w)/G_v)^{\C^*\times T_w}. 
\]
Note that $\varphi_{\bold{w}}$ is supported on the critical locus of the potential $\bold{w}$. Therefore, the pullback $i^*\varphi_{\bold{w}}$ under the embedding $i: (\tilde{X}^{st}(v, \vec{1}_w)/G_v)^{\C^*\times T_w} \inj \tilde{X}^{st}(v, \vec{1}_w)/G_v$ is supported on the fixed points set $S(\ell_1, \ldots, \ell_w)$ \eqref{eq:fixed points}. 
Denote by $p_T, q_T$ the restrictions of $p, q$ on $\C^*\times T_{w}$-fixed points. On the component,  $ (\tilde{X}^{st}(v, \vec{1}_w)/G_v)^{\C^*\times T_w}\setminus S(\ell_1, \ldots, \ell_w)$, we have $q_{T, *}p_{T}^*|_{(\tilde{X}^{st}(v, \vec{1}_w)/G_v)^{\C^*\times T_w}\setminus S(\ell_1, \ldots, \ell_w)}$ is zero.   

Recall the Levi subgroup of $P$ is $\C^*_{L}\times G_v$. At the point $(\lambda, \lambda+\blacksquare) \in (\tilde{X}^{st}(v, v+1, \vec{1}_w)/P)^{\C^*\times T_{w}}\cap \text{Supp}(\varphi_{\bold{w}})$ the tautological bundle $\calL$ has an action of $\C^*_{L}$ as well. So we consider the $\C^*_{L}\times \C^*\times T_w$-equivariance. 
Now 
\cite[Appendix (65)]{RSYZ2} gives
\begin{align*}
&\langle\lambda|(\tilde{e}(g))| \lambda+\blacksquare\rangle
=q_{T, *}p_{T}^*\Res_{z=x_{\blacksquare}}g(z) \frac{e(T_{\lambda} (\tilde{X}^{st}(v, \vec{1}_w)/G_v))} {e(T_{\lambda, \lambda+\blacksquare} (\tilde{X}^{st}(v, v+1,  \vec{1}_w)/P))}
\end{align*}
Here $e$ denotes  the $\C^*_{L}\times \C^*\times T_w$-equivariant Euler class.

Next  we calculate all the tangent spaces at the fixed points. Recall that we have short exact sequence of vector spaces
\[
0\to \xi\to V\to V_2 \to 0, 
\]
where $V$ has dimension $v+1$ and $V_2$ has dimension $v$.

For any $\lambda\in S(\ell_1, \ldots, \ell_w)$, the tangent complex of $\tilde{X}^{st}(v, \vec{1}_w)/G_v$ 
as a class in its Grothendieck group is given by 
\begin{align*}
-\End(\calV_2)+\Big(\calV_2^\vee\otimes \calV_2\otimes \C_{\epsilon}+ \sum_{i=1}^w\calW_i^{\vee}\otimes \calV_2\otimes \C_{a^*_i}
+\sum_{i=1}^w \calV_2^{\vee}\otimes \calW_i\otimes \C_{a_i}
\Big)
\end{align*}
The tangent complex of the correspondence $\tilde{X}^{st}(v, v+1,  \vec{1}_w)/P$ has the following class in the Grothendieck group,
\[
[T(\tilde{X}^{st}(v, v+1,  \vec{1}_w)/P)]=[-P+(P\otimes \C_{\epsilon}+ \sum_{i=1}^w\calW_i^{\vee}\otimes \calV\otimes \C_{a^*_i}+\sum_{i=1}^w\calV_2^{\vee}\otimes \calW_i\otimes \C_{a_i}\Big)],
\]
where $P=\Hom(\xi, \xi)\oplus \Hom(V_2, V_2)\oplus (V_2, \xi)=(\xi^{\vee}\otimes \xi) \oplus (V_2^{\vee}\otimes V_2)\oplus (V_2^{\vee}\otimes \xi)$. 
Now by \cite[Appendix (65)]{RSYZ2}, we have the following
\begin{align*}
&\langle\lambda|\tilde{e}_k| \lambda+\blacksquare\rangle\\
=&q_{T, *}p_{T}^*\Res_{z=x_{\blacksquare}}c_{1}(\calL)^k \frac{e(T_{\lambda} (\tilde{X}^{st}(v, \vec{1}_w)/G_v)} {e(T_{\lambda, \lambda+\blacksquare} ((\tilde{X}^{st}(v, v+1,  \vec{1}_w)/P)}|_{(\lambda, \lambda+\blacksquare)} \\
 =&q_{T, *}p_{T}^*\Res_{z=x_{\blacksquare}}c_{1}(\calL)^k \frac{e(-\calV_2^{\vee}\otimes \calV_2+(\calV_2^{\vee}\otimes \calV_2\otimes \C_{\epsilon}+ \sum_{i=1}^{w}W_i^{\vee}\otimes \calV_2\otimes \C_{a^*_i}+\sum_{i=1}^{w} \calV_2^{\vee}\otimes W_i\otimes \C_{a_i})}
 {e(-P+(P\otimes \C_{\epsilon}+ \sum_{i=1}^{w}W_i^{\vee}\otimes \calV\otimes \C_{a^*_i}+\sum_{i=1}^{w}\calV_2^{\vee}\otimes W_i \otimes \C_{a_i}}|_{(\lambda, \lambda+\blacksquare)}
 \end{align*}
As before let $z$ be the equivariant Chern root of $\xi$ and recall that $V_2=V/\xi$. We take the basis of  $V_2$ to be $\{{\square} \mid \square\in \lambda\}$. Then the equivariant Chern root of the line spanned by $\square$ is the $\C^*_{L}\times \C^*\times T_w$-weight of $\square$, which is  $x_{\square}$ in our notation. Getting rid of the factor $e(\xi^{\vee}\otimes \xi)$, we obtain
 \begin{align}
&q_{T, *}p_{T}^*\Res_{z=x_{\blacksquare}}c_{1}(\calL)^k \frac{e(
 \calV_2^{\vee}\otimes \xi) }{e(\calV_2^{\vee}\otimes \xi\otimes \C_{\epsilon}+ \sum_{i=1}^{w} W_i^{\vee}\otimes \xi \otimes \C_{a_i^*})}|_{(\lambda, \lambda+\blacksquare)} \notag\\
 =
&q_{T, *}p_{T}^*\Res_{z=x_{\blacksquare}} \Big(z^k 
\prod_{\square\in \lambda} \frac{z-x_{\square} }{z-x_{\square}-2\hbar} \prod_{j=1}^w\frac{1}{z-z_j+\ell_j\hbar}\Big).  \label{for:e}
\end{align}

Similarly, \cite[Appendix (62)]{RSYZ2}  gives
\begin{align*}
&\langle \lambda+\blacksquare|(\tilde{f}(g))| \lambda\rangle
=p_{T, *}q_{T}^*g(c_1(\calL))|_{(\lambda, \lambda+\blacksquare)}  \frac{e(T_{\lambda} (\tilde{X}^{st}(v+1, \vec{1}_w)/G_{v+1}))} {e(T_{\lambda, \lambda+\blacksquare} (\tilde{X}^{st}(v, v+1,  \vec{1}_w)/P))}. 
\end{align*}
Here $e$ stands for the $\C^*\times T_w$-equivariant Euler class.

Therefore, we have
\begin{align*}
\langle \lambda+\blacksquare|\tilde{f}_m| \lambda\rangle \notag
=&p_{T, *}q_{T}^*c_{1}(\calL)^m \frac{e(-\calV^{\vee}\otimes \calV+(\calV^{\vee}\otimes \calV\otimes \C_{\epsilon}+ \sum_{i=1}^{w}W_i^{\vee}\otimes \calV\otimes \C_{a^*_i}+\sum_{i=1}^{w} \calV^{\vee}\otimes W_i\otimes \C_{a_i})}
 {e(-P+(P\otimes \C_{\epsilon}+ \sum_{i=1}^{w}W_i^{\vee}\otimes \calV\otimes \C_{a^*_i}+\sum_{i=1}^{w}\calV_2^{\vee}\otimes W_i \otimes \C_{a_i}}|_{(\lambda, \lambda+\blacksquare)}
 \end{align*}
 Getting rid of the factor $e(\xi^{\vee}\otimes\xi)$, we obtain
 \begin{align}
&p_{T, *}q_{T}^*c_1(\calL)^m 
\frac{
e(\xi^\vee\otimes \calV_2\otimes \C_{\epsilon}+ \sum_{i=1}^{w}\xi^{\vee}\otimes W_i \otimes \C_{a_i})} 
{e(
\xi^{\vee}\otimes \calV_2)
} \notag\\
=&p_{T, *}q_{T}^*(x_{\blacksquare})^m 
\prod_{ \square\in \lambda}\frac{-x_{\blacksquare}+x_{\square}-2\hbar}{ -x_{\blacksquare}+x_{\square}}
\prod_{i=1}^w (-x_{\blacksquare}+z_i+\ell_i \hbar) \notag
\\
=&p_{T, *}q_{T}^*(-1)^w(x_{\blacksquare})^m 
\prod_{ \square\in \lambda}\frac{x_{\blacksquare}-x_{\square}+2\hbar}{ x_{\blacksquare}-x_{\square}}
\prod_{i=1}^w (x_{\blacksquare}-z_i-\ell_i \hbar)
. \label{for:tf}
\end{align}

Recall the line bundle $\calL$ is defined using the subspace $\xi\subset V$, with the quotient $V/\xi=V_2$.  Therefore, the $\C^*\times T_w$-weight of $\xi$ at the fixed point ${(\lambda, \lambda+\blacksquare)}$ is given by $x_{\blacksquare}$.

We normalize the formulas \eqref{for:e} and \eqref{for:tf} by moving the factor 
\(\prod_{i=1}^w (x_{\blacksquare}-z_i-\ell_i \hbar)\) appearing in \eqref{for:tf} to \eqref{for:e}. 
This yields the formulas \eqref{eaction} and \eqref{faction}.

\begin{remark}
The torus fixed points $(\mathfrak{M}^+(v, \vec{1}_w))^{\C^*\times T_{w}}$ and 
$(\tilde{X}(v, \vec{1}_w)^{st}/G_v)^{\C^*\times T_{w}}$ are different, as shown in 
Lemma~\ref{fixpointamb}. Hence, the notions of addible and removable boxes differ. 
The rational function \eqref{for:e} also has only simple poles. After normalization, 
\eqref{eaction} has simple poles at the addible boxes in the sense of 
the torus fixed points $(\mathfrak{M}^+(v, \vec{1}_w))^{\C^*\times T_{w}}$ 
(see Lemma~\ref{lem:e}).

\end{remark}

\section{Framed shuffle formula}\label{sec:shuffle}
\subsection{Framed shuffle algebra}\label{sub:shuffle algebra}
Fix $w\in \N$. We write an element in $\bbN^w$ as the vector form $\vec{n}=(n_1, \cdots, n_w)$. 
Let  $\SH_{fr}$ be an $\bbN\times \bbN^w$-graded $\mathbb{C}[\hbar]$-algebra. As a $\mathbb{C}[\hbar]$-module, we have 
\[\SH_{fr}=\bigoplus_{{v}\in\bbN, \vec{n}\in\bbN}\SH_{{v}, \vec{n}}, \,\ \text{where}\,\ \SH_{{v}, \vec{n}}:=\mathbb{C}[\hbar]\otimes \mathbb{C}
[ y_s ]_{s=1,\dots, v}^{\mathcal{S}_{{v}}}\otimes \mathbb{C}
[ z_j^i ]_{j=1,\dots, w, i=1, \ldots, n_j} ,\]
where $\mathcal{S}_{{v}}$ is the symmetric group of $v$ letters, and $\mathcal{S}_{{v}}$ naturally acts on the variables $\{ y_s \}_{s=1,\dots, v}$ by permutation. 
For any $({v_1}, \vec{n_1})$ and $({v_2}, \vec{n_2})\in \bbN\times \bbN^w$, write $v=v_1+v_2$ and $\vec{n}=\vec{n}_1+\vec{n}_2$.
We consider $\SH_{{v_1}, \vec{n_1}}\otimes_{\mathbb{C}[\hbar]} \SH_{{v_2}, \vec{n_2}}$ as a subalgebra of $
\mathbb{C}[\hbar][y_s, z_{j}^i]_{\{s=1,\dots, v,  j=1,\dots, w, i=1, \ldots, n_j\}}$ 
by sending $(y'_s, z_j\rq{}^{i}) \in \SH_{{v_1}, \vec{n_1}} $ to $(y_s, z_j^i)$, and 
$(y''_s, z_j\rq{}\rq{}^{i}) \in \SH_{{v_2}, \vec{v_2}} $ to $(y_{s+v_1}, z_{j}\rq{}\rq{}^{i+n_j})$.

We define the shuffle product $\SH_{{v_1}, {\vec{n}_1}}\otimes_{\mathbb{C}[\hbar]} \SH_{{v_2}, {\vec{n}_2}}\to \SH_{v, \vec{n}}$,
\begin{align}
&f(y_{{v_1}} , z_{{\vec{n}_1}})\otimes g(y_{{v_2}} , z_{{\vec{n}_2}})\mapsto
\sum_{\sigma\in \Sh(v_1, v_2)} \sigma\Big(f(y'_{{v_1}} , z'_{{\vec{n}_1}})\cdot g(y''_{{v_2}} , z''_{{\vec{n}_2}})\cdot \fac_{{v_1}+{v_2},{\vec{n}_1}+{\vec{n}_2}}\Big), \label{eq:shuffle}
\end{align}
where the factor is given by

\begin{equation*}
\fac_{(v_1, \vec{n}_1), (v_2, \vec{n}_2)}:=\prod_{s=1}^{v_2}
\prod_{t=1}^{v_1}\frac{y\rq{}\rq{}_s-y\rq{}_t-2\hbar}{y\rq{}_t-y\rq{}\rq{}_s}
\prod_{j=1}^w \Big(\prod_{a\in [1, v_1], b\in [1, n_{2j}]} (z\rq{}\rq{}^{b}_j-y'_a+\ell_j\hbar)
\prod_{a\in [1, n_{1j}], b\in [1, v_2]} (y''_b-z_j\rq{}^{a}+\ell_j\hbar)\Big), 
\end{equation*}
here $-2\hbar$ comes from the weight of $\epsilon$, and $\ell_j\hbar$ comes from the weight of $a_j$ and $a^*_j$, $j=1, \ldots, w$. 

The graded vector space $\SH_{fr}$ together with the shuffle product gives an algebra structure.

\subsection{The geometric construction of the framed shuffle algebra}
The geometric interpretations of the framed shuffle algebra $\SH_{fr}$ can be obtained from \cite[\S 3.2 and Proposition 3.2]{YZ1}. For the convenience of the readers, we recall the results here.

For any pair of dimension vectors $(v_1, \vec{n}_1),(v_2, \vec{n}_2)\in\bbN\times \bbN^w$, let $v=v_1+v_2$, and $\vec{n}=\vec{n}_1+\vec{n}_2$. 
We consider a map \[m^S_{(v_1, \vec{n}_1),(v_2, \vec{n}_2)}:H^*_{c, G_{v_1}\times T_{{n}_1}\times \C^*}(\tilde{X}(v_1, \vec{n}_1))^{\vee}\otimes_{\C[\hbar]}H^*_{c, G_{v_2}\times T_{{n}_2}\times \C^*}(\tilde{X}(v_2, \vec{n}_2))^{\vee}\to H^*_{c,G_{v}\times T_{n}\times \C^*}(\tilde{X}(v, \vec{n}))^{\vee},\] 
defined as follows. 
For a vector space $V$ of dimension $v$, let $V_1$ be a subspace of $V$ with dimension $v_1$. 
Denote by $N$ a $w$-tuple of vector spaces of dimension vector $\vec{n}$. Let $N_1 \subset N$ be $w$-tuple of subvector spaces of dimension vector $\vec{n}_1$. Define
\[\tilde{X}_{(v_1, \vec{n}_1), (v_2, \vec{n}_2)}:=\{(\epsilon, a_1, a^*_1, \ldots, a_w, a^*_w)\in \tilde{X}(v, \vec{n})\mid \epsilon(V_1)\subset V_1, a_i(V_1)\subset N_1, a^*_i(N_1)\subset V_1, i\in [1, w]\}\subset \tilde{X}(v, \vec{n}).\]
An element $(\epsilon, a_i, a^*_i)\in \tilde{X}_{(v_1, \vec{n}_1), (v_2, \vec{n}_2)}$ induces the following linear maps:
\[
\xymatrix{
V_1\ar[r]^{\epsilon|_{V_1}} \ar@{^{(}->}[d]& V_1 \ar@{^{(}->}[d]\\
V\ar[r]^{\epsilon} \ar@{->>}[d]& V\ar@{->>}[d]\\
V/V_1\ar[r]^{\bar\epsilon} & V/V_1\\
}\,\ 
\xymatrix{
V_1\ar[r]^{(a_1, \ldots, a_n)|_{V_1}} \ar@{^{(}->}[d]& N_1 \ar@{^{(}->}[d]\\
V\ar[r]^{(a_1, \ldots, a_n)} \ar@{->>}[d]& N\ar@{->>}[d]\\
V/V_1\ar[r]^{(\bar a_1, \ldots, \bar a_n)} & N/N_1\\
}\,\ 
\xymatrix{
N_1\ar[r]^{(a_1^*, \ldots, a_n^*)|_{N_1}} \ar@{^{(}->}[d]& V_1 \ar@{^{(}->}[d]\\
N\ar[r]^{(a_1^*, \ldots, a_n^*)} \ar@{->>}[d]& V\ar@{->>}[d]\\
N/N_1\ar[r]^{(\bar a_1^*, \ldots, \bar a_n^*)} & V/V_1\\
}
\]

We write $G:=G_{v}$, and $P \subset G_v$, the parabolic subgroup preserving the subspace $V_1$.
Let $L:=G_{v_1}\times G_{v_2}$ be the Levi subgroup of $P$. We have the following correspondence of $G\times T_w\times \C^*$-varieties:
\begin{equation*}
\xymatrix{
G\times_P (\tilde{X}(v_1, \vec{n}_1)\times \tilde{X}(v_2, \vec{n}_2))&G\times_P \tilde{X}_{(v_1, \vec{n}_1), (v_2, \vec{n}_2)} \ar[l]_(0.4){\phi} \ar[r]^{\psi} & G\times_P \tilde{X}(v, \vec{n})\ar[r]^{\pi}& 
\tilde{X}(v, \vec{n}).
}
\end{equation*}
Here $\phi$ is induced by the map 
\begin{align*}
&\tilde{X}_{(v_1, \vec{n}_1), (v_2, \vec{n}_2)}
\to \tilde{X}(v_1, \vec{n}_1)\times \tilde{X}(v_2, \vec{n}_2), 
 (\epsilon, a_i, a^*_i)
\mapsto \Big(
(\epsilon|_{V_1}, a_i|_{V_1}, a^*_i|_{N_1}), 
(\bar{\epsilon}, \bar{a}_i, \bar{a}^*_i)
\Big)
\end{align*}
the map $\psi$ is induced by the embedding $\tilde{X}_{(v_1, \vec{n}_1), (v_2, \vec{n}_2)}\subset \tilde{X}(v, \vec{n})$ and $\pi$ is the action map 
\[
 G\times_P \tilde{X}(v, \vec{n})\to 
\tilde{X}(v, \vec{n}), \big(g, (\epsilon, a_i, a^*_i)\big)\mapsto g\cdot (\epsilon, a_i, a^*_i). 
\]

We now define the multiplication map $m^S_{(v_1, \vec{n}_1), (v_2, \vec{n}_2)}$. We first have the K\"unneth isomorphism.
\[
\otimes: H^*_{c,G_{v_1}\times T_{{n}_1}\times \C^*}(\tilde{X}(v_1, \vec{n}_1))^{\vee}\otimes_{\C[\hbar]}H^*_{c,G_{v_2}\times T_{{n}_2}\times \C^*}(\tilde{X}(v_2, \vec{n}_2))^{\vee}\cong H^*_{c,L\times T_n\times \C^*}(\tilde{X}(v_1, \vec{n}_1)\times \tilde{X}(v_2, \vec{n}_2))^{\vee}. 
\] 
Consider the following sequence of morphisms:
\begin{enumerate}
\item The isomorphism:
$
H^*_{c,L\times T_n\times \C^*}(\tilde{X}(v_1, \vec{n}_1)\times \tilde{X}(v_2, \vec{n}_2))^{\vee}
\cong H^*_{c,G\times T_n\times \C^*}(G\times_{P} (\tilde{X}(v_1, \vec{n}_1)\times \tilde{X}(v_2, \vec{n}_2)))^{\vee}.
$
\item The pullback map:
$$
\phi^*: H^*_{c,G\times T_n\times \C^*}(G\times_{P} (\tilde{X}(v_1, \vec{n}_1)\times \tilde{X}(v_2, \vec{n}_2)))^{\vee} \to H^*_{c,G\times T_n\times \C^*}(G\times_P \tilde{X}_{(v_1, \vec{n}_1), (v_2, \vec{n}_2)} )^{\vee}.
$$
\item The pushforward $\psi_*$
\[
H^*_{c,G\times T_n\times \C^*}(G\times_P \tilde{X}_{(v_1, \vec{n}_1), (v_2, \vec{n}_2)} )^{\vee}\to 
H^*_{c,G\times T_n\times \C^*}(G\times_P \tilde{X}(v, \vec{1}_w))^{\vee}
\]
\item The pushforward $\pi_*$
\[ 
H^*_{c,G\times T_n\times \C^*}(G\times_P \tilde{X}(v, \vec{1}_w))^{\vee}\to
H^*_{c,G\times T_n\times \C^*}( \tilde{X}(v, \vec{1}_w))^{\vee}
\]
\end{enumerate}
We define map $m^S_{(v_1, \vec{n}_1), (v_2, \vec{n}_2)}$ to be the composition of the 
K\"unneth isomorphism with the above sequence of morphisms.

Now for any $(v, \vec{n})\in\bbN\times \bbN^w$, we identify \[\SH_{v, \vec{n}}=\mathbb{C}[\hbar]\otimes \mathbb{C}
[ y_s ]_{s=1,\dots, v}^{\mathcal{S}_{{v}}}\otimes \mathbb{C}
[ z_j^i ]_{j=1,\dots, w, i=1, \ldots, n_j} \] with the $H_{\C^*}(\pt)$-module $H^*_{c,G_v\times T_n\times \C^*}(\tilde{X}(v, \vec{n}))^{\vee}$, as $\tilde{X}(v, \vec{n})$ is an affine space. 

We identify the algebraically defined formal shuffle algebra multiplication in \eqref{eq:shuffle} with the geometrically defined $m^S_{v_1,v_2}$ as follows. 
\begin{prop} \cite[Proposition 3.2]{YZ1} \label{prop:geom}
 Under the identification 
\[\SH_{v, \vec{n}}=\mathbb{C}[\hbar]\otimes \mathbb{C}
[ y_s ]_{s=1,\dots, v}^{\mathcal{S}_{{v}}}\otimes \mathbb{C}
[ z_j^i ]_{j=1,\dots, w, i=1, \ldots, n_j} \cong H^*_{c,G_v\times T_n\times \C^*}(\tilde{X}(v, \vec{n}))^{\vee},\] 
the map $m^S_{(v_1, \vec{n}_1), (v_2, \vec{n}_2)}$ is equal to the multiplication map  \eqref{eq:shuffle} of the shuffle algebra.
\end{prop}
\begin{proof}

The map $\psi: G\times_P \tilde{X}_{(v_1, \vec{n}_1), (v_2, \vec{n}_2)}\to G\times_P \tilde{X}(v, \vec{n})$ is an embedding of vector bundles on $G/P$.
The pushforward $\psi_{*}$ is the multiplication by the equivariant Euler class $e^{\psi}_{G\times T_n\times \C^*}$ of the normal bundle of the embedding $\psi$. 
The normal bundle to $\psi$ can be identified with 
\begin{align*}
&\sHom_{\widetilde{Q}}(\calV_1, \calV_2) 
\oplus \sHom_{\widetilde{Q}}(\calV_1, \calN_2) 
\oplus \sHom_{\widetilde{Q}}(\calN_1, \calV_2) 
\end{align*}
where $\calV, \calN$ are the tautological bundles on $\tilde{X}(v, \vec{n})$. 

Taking into account of the $\C^*$-weights of $\epsilon, a_i, a^*_i$, we have the equivariant Euler class of the normal bundle to $\psi$ is
\[
e^{\psi}_{G\times T_n\times \C^*}=
\prod_{s=1}^{v_2}
\prod_{t=1}^{v_1}
(y_s''-y_t'-2\hbar)
\prod_{j=1}^w \Big(\prod_{a\in [1, v_1], b\in [1, n_{2j}]} (z^{''b}_j-y'_a+\ell_j\hbar)
\prod_{a\in [1, n_{1j}], b\in [1, v_2]} (y''_b-z_j^{'a}+\ell_j\hbar)\Big).\]

The map $\pi: G\times_P \tilde{X}(v, \vec{n})\to 
\tilde{X}(v, \vec{n})$ is a Grassmannian bundle. So the pushforward $\pi_*$ can be computed, see for example \cite[Proposition 1.2]{YZ1}. 
Putting all the above together, the map $m_{v_1, v_2}^S$ is given by exactly the same formula as  \eqref{eq:shuffle}.
\end{proof}

\section{The weight functions}
In this section, we define the weight functions using the framed shuffle formula \eqref{eq:shuffle} and study their properties. 

\subsection{}
Let $v\in \N$. For any $\vec{n}=(n_1, n_2, \cdots, n_w)\in \N^{w}$ such that $n_i\in \{0, 1\}$, write 
\[
z_{\vec{n}}=\{z_i\mid \text{with index $i$ such that $n_i=1$}\}.
\]
Set \[\SH_{{v}, \vec{n}}:=\mathbb{C}[\hbar]\otimes \mathbb{C}
[ y_1, y_2, \cdots, y_v ]^{\mathcal{S}_{{v}}}\otimes \mathbb{C}
[ z_{\vec{n}}]. \]
The variable $z_i$ is the equivariant parameter coming from the  the framing vertex $i$, such that $n_i=1$. 
In particular, for the vector $\vec{1}_w=(1, 1, \cdots, 1)\in \N^w$. We have
$\SH_{{v}, \vec{1}_w}=\mathbb{C}[\hbar]\otimes \mathbb{C}
[ y_1, y_2, \cdots, y_v ]^{\fS_{{v}}}\otimes \mathbb{C}
[ z_1, z_2, \ldots, z_w]$.


Let $\lambda=(v_1, \cdots, v_w)\in S(\ell_1, \ldots, \ell_w)$ be a torus fixed  point.  
Let $e_i:=(0, \cdots, 0, 1, 0, \cdots, 0)\in \N^w$, where $i$ is at the $i$-th place. 
Consider the multiplication
\begin{align}
\SH_{v_1, e_1}\otimes \SH_{v_1, e_2}\otimes  \cdots \otimes \SH_{v_w, e_{w}}
\to \SH_{v, \vec{1}_w}. \label{star1}
\end{align}
The domain above depends on the fixed point $\lambda$, and we denote the multiplication \eqref{star1} by $\star_{\lambda}$ to emphasize this dependence on $\lambda$. We consider the image of $1\otimes 1\otimes \cdots \otimes 1$ under $\star_{\lambda}$. 

Define the weight function associated to $\id\in \mathcal{S}_w$ at $\lambda$ to be
\begin{equation}
W^{\id}_{\lambda}:=
1\star_{\lambda} 1\star_{\lambda} \cdots \star_{\lambda} 1\in \mathbb{C}[\hbar]\otimes \mathbb{C}
[ y_s ]_{s=1,\dots, v}^{\mathcal{S}_{{v}}}\otimes \mathbb{C}
[ z_t ]_{t=1,\dots, w}
\end{equation}
The permutation group $\mathcal{S}_{w}$ permutes the variables $\{z_1, \cdots, z_w\}$ in $\SH_{v, w}$. 

We introduce the following notations. 
\begin{itemize}
\item
We say $I=(I_1, I_2, \cdots, I_w)$ is a partition of $[1, v]$, denoted by $(I_1, I_2, \cdots, I_w)\vdash [1, v]$, if $I_1, I_2, \cdots, I_w$ are subsets of $[1, v]$ such that $I_1\sqcup I_2\sqcup \cdots \sqcup I_w=[1, v]$ and $|I_1|=v_1, \cdots, |I_w|=v_w$. 
\item For two sets $I, J$, we write $y_{I}-y_J$ means the product $\prod_{i\in I, j\in J} (y_{i}-y_j)$. In particular, we have
\[
y_{I}-y_J-2\hbar=\prod_{i\in I, j\in J}(y_{i}-y_j-2\hbar), \text{and \,\ $z_j-y_{I}+\ell_j\hbar =\prod_{i\in I} (z_j-y_{i}+\ell_j\hbar)$}. 
\]
\end{itemize}
Inductively using the formula \eqref{eq:shuffle}, we then have
\begin{align}
&W^{\id}_{\lambda}
=\sum_{I=(I_1, I_2, \cdots, I_w)\vdash [1, v], |I_j|=v_j} W^{\id}_{\lambda, I} \label{Wformula}\\
=&\sum_{(I_1, I_2, \cdots, I_w)\vdash [1, v], |I_j|=v_j} \Big(\prod_{1\leq t<s\leq w}\frac{y_{I_s}-y_{I_t}-2\hbar}{y_{I_t}-y_{I_s}}\Big)\cdot\Big(
\prod_{1\leq s<a\leq w}(z_a-y_{I_s}+\ell_a\hbar)
\Big)
\Big(\prod_{1\leq a<s\leq w}(y_{I_s}-z_a+\ell_a\hbar)\Big). \notag
\end{align}
Here \begin{align}
W^{\id}_{\lambda, I}:= \Big(\prod_{1\leq t<s\leq w}&\frac{y_{I_s}-y_{I_t}-2\hbar}{y_{I_t}-y_{I_s}}\Big)
\Big(\prod_{1\leq s<a\leq w}(z_a-y_{I_s}+\ell_a\hbar)\Big)
\Big(\prod_{1\leq a<s\leq w}(y_{I_s}-z_a+\ell_a\hbar)\Big) \label{summand}\\
&\in \mathbb{C}[\hbar]\otimes \mathbb{C}
[ y_s ]_{s=1,\dots, v}^{\mathcal{S}_{{v}}}[\prod_{i\neq j}\frac{1}{y_i-y_j}]\otimes \mathbb{C}
[ z_1, 
\ldots, z_w]\notag
\end{align}
Note that $W^{\id}_{\lambda}$ is regular in the variables $y_1, \ldots, y_v$, but $W^{\id}_{\lambda, I}$ could have poles at $y_i-y_j=0$. 
\subsection{Restriction to torus fixed points}

For each fixed point $\mu=(\mu_1, \mu_2, \cdots, \mu_w)\in S(\ell_1, \ldots, \ell_w)$, we restrict the weight functions at fixed point $\mu$, denoted by $W^{\id}_{\lambda}|_{\mu}$. 
That is, we substitute $y_1, \cdots y_v$ in $W^{\id}_{\lambda}$ according to the equivariant Chern roots of $V$ at the fixed point $\mu$: 
\begin{equation}\label{eq:fix point wts}
z_i-\ell_i\hbar, z_i-(\ell_i-2)\hbar, z_i-(\ell_i-4)\hbar, \,\ 
\cdots, \,\ z_i-(\ell_i-2\mu_i+2)\hbar, \,\ i=1, \cdots, w. 
\end{equation}
Thus, $W^{\id}_{\lambda}|_{\mu}\in  \mathbb{C}[\hbar, z_1, \cdots, z_w]$. 
Note that $y_1, \cdots y_v$ is invariant under the permutation group $\mathcal{S}_v$, therefore any ordering of $y_1, \cdots, y_v$ under the substitution gives the same answer.

\subsection{Properties of the weight functions}

\subsubsection{}
Clearly, when $w=0$, there is only one fixed point $0$. We have
$
W^{\id}_{0}=1. 
$
When $v=1$, the variables are $y, z_1, \cdots, z_w$ in $\SH_{1, \vec{1}_w}$.  
There are $w$ number of partitions of $\{1\}$. We have 
\[
(I_1, \cdots, I_s, \cdots, I_{w})
=(\emptyset, \cdots, \{1\}, \cdots \emptyset), 
\]
for only one subset $I_s=\{1\}$, and $I_t=\emptyset$, for $s\neq t$. 

For the tuple $e_s:=(0, 0, \cdots, 1, \cdots, 0)$ with $1$ at the position $s$. We have
\begin{align}
&W^{\id}_{e_s}=
\prod_{\{a|s<a\leq w\}}(z_a-y+\ell_a\hbar)
\prod_{\{a|1\leq a<s\}}(y-z_a+\ell_a\hbar)\in \SH_{1, \vec{1}_w}. \label{W:v=1}
\end{align}

\subsubsection{}
We use the lexicographical order on the set of torus fixed points. 
That is: for $\lambda=(v_1, \cdots, v_w)$ and $\mu=(\mu_1, \cdots, \mu_w)$ in $S(\ell_1, \ldots, \ell_w)$, we define 
\[
\lambda=(v_1, \cdots, v_w)< \mu=(\mu_1, \cdots, \mu_w) \iff v_1=\mu_1, \cdots, v_a=\mu_a, v_{a+1}<\mu_{a+1}, \text{for some $a\in [1, w]$}. 
\]

For two fixed points $\lambda$ and $\mu$, set 
\begin{align*}
&\deg(W^{\id}_{\lambda, I}\mid_\mu)
:=\sum_{1\leq i<j\leq w} (\deg_{z_i-z_j}(W^{\id}_{\lambda, I}\mid_ \mu)), \\
&\deg(W^{\id}_{\lambda}\mid_\mu)
:=\max\{\deg(W^{\id}_{\lambda, I}\mid_\mu)\mid I\vdash [1, v] \}
\end{align*}

Equivalently, $\deg(W^{\id}_{\lambda}\mid_\mu)$ is the degree of the regular function $W^{\id}_{\lambda}\mid_ \mu$ in $z_1, z_2, \cdots, z_w$, 
such that 
\[
\deg(z_i)=1, i\in [1, w], \,\ \deg(\hbar)=0. 
\]
We have the following properties for the weight functions $W^{\id}_{\lambda}$. They are similar to the properties from \cite[Lemma 3.1]{RTV15}. 

\begin{lemma}\label{lem:property}
Let $\lambda, \mu \in S(\ell_1, \ldots, \ell_w)$ be two fixed points of $\mathfrak{M}^+(v, \vec{1}_w)$. 
\begin{enumerate}
\item If $\lambda<\mu$, we have
\[
W^{\id}_{\lambda}|_{\mu}=0. 
\]
\item If $\lambda=\mu$, we have
\[
W^{\id}_{\lambda}|_{\lambda}= \prod_{1\leq s<a\leq w}\Big(
 \prod_{i=v_s-v_a+1}^{v_s}(z_s-z_a+2i\hbar+(\ell_a-\ell_s)\hbar)
\prod_{i=0}^{v_s-1}(z_a-z_s-2i\hbar+(\ell_s+\ell_a)\hbar)\Big). 
\]
In particular, for $1\leq s<a\leq w$, the degree of $z_s-z_a$ in $W^{\id}_{\lambda}|_{\lambda}$ is $v_a+v_s$ and 
therefore, 
\[
\deg(W^{\id}_{\lambda}|_{\lambda})=(w-1) v. 
\]
\item 
For each partition $I=(I_1, I_2, \cdots, I_w)\vdash [1, v]$, with  $|I_j|=v_j$, for any two fixed points $\lambda\geq \mu$, let 
$I_{i}^i:=I_{i}\cap[\sum_{s=1}^{j-1}\mu_s+1,  \sum_{s=1}^{j}\mu_s]$, then we have 
\[
\deg_{z_i-z_j}( W^{\id}_{\lambda, I}|_{\mu})=\mu_j+\mu_i-|I_{i}^j|
\] and 
\[
\deg( W^{\id}_{\lambda, I}|_{\mu})=(w-1)v-\sum_{i<j}|I_{i}^j|
\]
\item
For any $w\in \N$, we have, for $\lambda\geq \mu$, then,
\[
\deg(W^{\id}_{\lambda}|_{\lambda})\geq \deg (W^{\id}_{\lambda}|_{\mu}). 
\]
\end{enumerate}
\end{lemma}
\begin{proof}
This follows from a direct computation. For the convenience of the reader, we include a proof here. 

At the fixed point $\mu$, we use the following substitution
\begin{align*}
& y_1=z_1-\ell_1\hbar,  \,\ y_2=z_1-\ell_1\hbar+2\hbar,   \,\ 
\cdots, \,\ y_{\mu_1}=z_1-\ell_1\hbar+2(\mu_1-1)\hbar, \\ 
& y_{\mu_1+1}=z_2-\ell_2\hbar,  y_{\mu_1+2}=z_2-\ell_2\hbar+2\hbar, 
\cdots, \,\ y_{\mu_1+\mu_2}=z_2-\ell_2\hbar+2(\mu_2-1)\hbar\\ 
&\cdots\\
& y_{\sum_{i=1}^{a-1}\mu_i+1}=z_a-\ell_a\hbar, \,\  y_{\sum_{i=1}^{a-1}\mu_i+2}=z_a-\ell_a\hbar+2\hbar,  
\cdots, \,\ y_{\sum_{i=1}^{a}\mu_i}=z_a-\ell_a\hbar+2(\mu_a-1)\hbar. 
\end{align*}
For each partition  $I=(I_1, I_2, \cdots, I_w)\vdash [1, v]$ such that $|I_1|=v_1, \cdots, |I_w|=v_w$, the summand $W^{\id}_{\lambda, I}|_{\mu}$ is given by the formula \eqref{summand}. 
Note that the factor $y_{I_s}-z_a+\ell_a\hbar$ is zero, if and only if there exists an $a\in [1, w]$ and an $s$ with $a<s\leq w$ such that 
\[
\sum_{i=1}^{a-1}\mu_i+1\in I_s, 
\]

Assume $W^{\id}_{\lambda, I}|_{\mu}$ \eqref{summand} is non-zero. Then, $y_{I_s}-z_a+\ell_a\hbar\neq 0$ implies
$\sum_{i=1}^{a-1}\mu_i+1\notin I_s$, 
for any $a\in [1, w]$ and any $s$ such that $a<s\leq w$. This forces that
\begin{align*}
1\in I_1, \,\ \mu_1+1\in I_1\sqcup I_2,\,\  \cdots, 
\sum_{i=1}^{a-1}\mu_i+1\in \sqcup_{s=1}^{a} I_{s}, \cdots
\end{align*}

The factor $y_{I_s}-y_{I_t}-2\hbar\neq 0$, with $s>t$,  implies that
\begin{align*}
&\text{If $1\in I_1$, then we must have $2, 3, \cdots, \mu_1\in I_1$}, \\
&\text{If $\mu_1+1\in I_1\sqcup I_2$, then we must have $\mu_1+2, \mu_1+3, \cdots, \mu_1+\mu_2\in I_1\sqcup I_2$}, \\
&\phantom{12} \text{furthermore, if $\mu_1+k\in I_2$, then we must have $\mu_1+k-1, \mu_1+k-2, \cdots, \mu_1+1\in I_2$}, \\
&\cdots, \\
&\text{If $\sum_{i=1}^{a-1}\mu_i+1\in \sqcup_{s=1}^{a} I_{s}$, then we must have $\sum_{i=1}^{a-1}\mu_i+2, \sum_{i=1}^{a-1}\mu_i+3, \cdots, \sum_{i=1}^{a}\mu_i\in \sqcup_{s=1}^{a} I_{s}$}, \\
&\phantom{12} \text{furthermore, if $\sum_{i=1}^{a-1}\mu_i+k\in I_m$, then we must have $\sum_{i=1}^{a-1}\mu_i+k-1,\cdots, \sum_{i=1}^{a-1}\mu_i+1\in I_m\sqcup I_{m+1}\cdots \sqcup I_a$}. 
\end{align*}
Proof of (1): By assumption, we have the total order 
$(v_1, \cdots, v_w)< (\mu_1, \cdots, \mu_w)$. That is, there exists some number $a\in [1, w]$, such that $v=\mu_1, \cdots, v_a=\mu_a$, and $v_{a+1}<\mu_{a+1}$. Therefore 
\begin{align*}
|I_1|=v_1=\mu_1, \cdots, |I_a|=v_a=\mu_a, 
|I_{a+1}|=v_{a+1}<\mu_{a+1}. 
\end{align*}
The discussion above shows that, if $W^{\id}_{\lambda, I}|_{\mu}$ \eqref{summand} is non-zero, we have the following inequalities:
\[
|I_1|\geq \mu_1, |I_1|+|I_2|\geq \mu_1+\mu_2, \cdots, \sum_{i=1}^{a} |I_{i}|\geq \sum_{i=1}^{a}\mu_i, \cdots 
\]
This is a contradiction. Therefore, $W^{\id}_{\lambda}|_{\mu}=0$.

Proof of (2): 
If $W^{\id}_{\lambda, I}|_{\mu}$ \eqref{summand} is non-zero and $\lambda=\mu$, then, we must have
\begin{align*}
I_1=[1, v_1], I_2=[v_1+1, v_1+v_2], \cdots, 
I_{v}=[\textstyle\sum_{i=1}^{w-1} v_i, v]. 
\end{align*}
Therefore, we have the substitution $y_{I_s}=z_{s}-\ell_s+2j\hbar$, $j=0, 1, \cdots, v_{s-1}$, for any $s=1, \cdots, w$. 
Substituting in the formula of $W_{\lambda}^{\id}$, we have
\begin{align*}
W^{\id}_{\lambda}|_{\lambda}
=& \Big(\prod_{1\leq t<s\leq w}\frac{\prod_{i=0}^{v_s-1}\prod_{j=0}^{v_t-1} ((z_s-\ell_s\hbar+2i\hbar)-(z_t-\ell_t\hbar+2j\hbar)-2\hbar)}{\prod_{i=0}^{v_s-1}\prod_{j=0}^{v_t-1} ((z_t-\ell_t\hbar+2j\hbar)-(z_s-\ell_s\hbar+2i\hbar))}\Big)\\& 
\Big(\prod_{1\leq s<a\leq w}\prod_{i=0}^{v_s-1}(z_a-(z_s-\ell_s\hbar+2i\hbar)+\ell_a\hbar)\Big)
\prod_{1\leq a<s\leq w}\Big(\prod_{i=0}^{v_s-1}((z_s-\ell_s\hbar+2i\hbar)-z_a+\ell_a\hbar)\Big)\\
 =& \prod_{1\leq s<a\leq w}\Big(
 \prod_{i=v_s-v_a+1}^{v_s}(z_s-z_a+2i\hbar+(\ell_a-\ell_s)\hbar)
\prod_{i=0}^{v_s-1}(z_a-z_s-2i\hbar+(\ell_s+\ell_a)\hbar)\Big)
\end{align*}

Clearly, for $1\leq s<a\leq w$, the degree of $z_s-z_a$ is $v_a+v_s$. 

Proof of (3): 
In general, for the fixed point $\mu$, set
\begin{align*}
I_{a}^{1}:=I_a\cap [1, \mu_1], \,\ I_{a}^{2}:=I_a\cap [\mu_1+1, \mu_1+\mu_2],\,\ 
\cdots, \,\ I_{a}^{w}:=I_1\cap [\sum_{i=1}^{w-1}\mu_i+1, v],\,\ \text{for $a=1, 2, \cdots w$}. 
\end{align*}
Then, 
$I_{a}=I_{a}^{1} \sqcup I_{a}^{2}\sqcup \cdots \sqcup I_{a}^{w}$ and $[\sum_{i=1}^{k-1}\mu_i+1, \sum_{i=1}^{k}\mu_i]=I_{1}^{k}\sqcup  \cdots \sqcup I_{w}^{k}$. 
The restriction $W^{\id}_{\lambda, I}|_{\mu}$  is non-zero, then
\begin{enumerate}
\item
$I_{a}^{i}=\emptyset$, if $i<a$; 
\item and for each $k\in [1, w]$, in the decomposition $ [\sum_{i=1}^{k-1}\mu_i+1, \sum_{i=1}^{k}\mu_i]=I_{1}^{k}\sqcup  \cdots \sqcup I_{k}^{k}$, the ordering of the elements is $i_s<i_t$, for any $i_s\in I_{s}^{k}, i_t\in I_{t}^{k}$ with $s>t$. We write $I_{k}^k< I_{k-1}^k <\cdots < I_{1}^k$ for short. 
\end{enumerate}

By setting $\hbar=0$, if $W^{\id}_{\lambda, I}|_{\mu}$  is non-zero, then the degree of $W^{\id}_{\lambda, I}|_{\mu}$ can be computed as follows.  
\begin{align*}
\deg(W^{\id}_{\lambda, I}|_{\mu})=
\deg(\prod_{a\in [1, w]} \prod_{\{1\leq s\leq b\leq w\mid s\neq a, b\neq a\}}(z_a-y_{I_s^b})|_{\mu})
=\deg(\prod_{1\leq a\neq b\leq w} \prod_{\{1\leq s\leq b \mid s\neq a\}}(z_a-z_b)^{|I_{s}^b|}). 
\end{align*}
Therefore, for $i<j$,  if $W^{\id}_{\lambda, I}|_{\mu}$ is non-zero, then the power of $z_i-z_j$ in $W^{\id}_{\lambda, I}|_{\mu}$ is given by 
\[
\sum_{\{1\leq s \leq j\mid s\neq i\}}|I_{s}^j|+\sum_{1\leq s \leq i}|I_{s}^i|=\mu_j+\mu_i-|I_{i}^j|. 
\]


Proof of (4): It follows from (3) that 
$$\deg(W^{\id}_{\lambda}|_{\mu, I})
=(w-1)v-\sum_{i<j}^w|I_{i}^j|
\leq (w-1)v
=\deg(W^{\id}_{\lambda}|_{\lambda}),
$$ for any partition $I$. This shows the assertion. 
\end{proof}

\subsection{The action of $\mathcal{S}_{w}$}
Recall we have the fixed point set from \eqref{eq:fixed points}
\[
S(\ell_1, \ldots, \ell_w)=\{(v_1, v_2, \cdots, v_w) \mid 0\leq v_i\leq \ell_i, \forall i\in [1, w],\,\ 
\text{and} \sum_{i=1}^wv_i=v\}.  
\]
For any element $\sigma$ in the symmetric group $\mathcal{S}_{w}$, $\sigma$ permutes $(\ell_1, \ldots, \ell_w)$ by $\sigma\cdot (\ell_1, \ldots, \ell_w)=(\ell_{\sigma(1)}, \ldots, \ell_{\sigma(w)})$. Therefore, we have the following
\[
\sigma: S(\ell_1, \ldots, \ell_w)\to S(\ell_{\sigma(1)}, \ldots, \ell_{\sigma(w)}),  \,\ (v_1, \ldots, v_w)\mapsto (v_{\sigma(1)}, \ldots, v_{\sigma(w)})
\]

Define the weight function associated to $\sigma \in \mathcal{S}_w$ at $\lambda\in S(\ell_{\sigma(1)}, \ldots, \ell_{\sigma(w)})$ to be
\begin{align}\label{eq:Wsigma}
&W^{\sigma}_{\lambda}(y_{1}, \ldots, y_{v}; z_1, \ldots, z_w; \hbar):=W^{\id}_{\sigma^{-1}\lambda}(y_{1}, \ldots, y_{v} ; z_{\sigma^{-1}(1)}, \cdots, z_{\sigma^{-1}(w)}; \hbar),
\end{align}
where $\sigma^{-1}(\lambda)=\sigma^{-1}(v_1, \cdots, v_w)=(v_{\sigma^{-1}(1)},\cdots,  v_{\sigma^{-1}(w)})\in S(\ell_1, \ldots, \ell_w)$. 

For a general element $\sigma\in \mathcal{S}_w$, we define the total order on the set $S(\ell_{\sigma(1)}, \ldots, \ell_{\sigma(w)})$ depending on $\sigma$ by
\[
\lambda\leq _{\sigma} \mu \,\ \text{in $S(\ell_{\sigma(1)}, \ldots, \ell_{\sigma(w)})$}\iff \sigma^{-1}(\lambda)\leq \sigma^{-1}(\mu) \,\ \text{in $S(\ell_1, \ldots, \ell_w)$}. 
\]
Then, Lemma \ref{lem:property} implies the following. 
\begin{lemma}\label{lem:prop2}
\begin{enumerate}
\item
Let $\sigma\in S_{w}$. Let $\lambda, \mu$ be two fixed points in $S(\ell_{\sigma(1)}, \ldots, \ell_{\sigma(w)})$. 
 If $\lambda<_{\sigma}\mu$, we have
\[
W^{\sigma}_{\lambda}|_{\mu}=0. 
\]
\item For any $w\in \N$, we have, for $\lambda\geq_{\sigma} \mu$, then,
\[
\deg(W^{\sigma}_{\lambda}|_{\lambda})\geq \deg (W^{\sigma}_{\lambda}|_{\mu}).
\]
\end{enumerate}
\end{lemma}
\begin{proof}
By definition we have
\begin{align*}
W^{\sigma}_{\lambda}(y_1, \ldots, y_v; z_1, \cdots, z_w; \hbar)|_{\mu}=W^{\id}_{\sigma^{-1}\lambda}(y_1, \ldots, y_v; z_{\sigma^{-1}(1)}, \cdots, z_{\sigma^{-1}(w)}; \hbar)|_{\sigma^{-1}\mu}.\end{align*}
By Lemma \ref{lem:property}, the above is zero if $ \sigma^{-1}(\lambda)<\sigma^{-1}(\mu)$, which is equivalent to $\lambda<_{\sigma}\mu$. This shows (1). (2) follows from Lemma \ref{lem:property} (4). 
\end{proof}

\section{Computation of $R$-matrices}
\label{ComputeR}
\subsection{Case $w=1$}
Fix the dimension vector $(v, \vec{1}_w)$ of $\tilde{Q}$. 
For $w=1$, there is only one fixed point $\lambda=v\in \N$, and the multiplication $\star_{\lambda}$ is the identity map. 
Therefore, 
\[
W^{\id}_{v}=1. 
\]
\subsection{Case $w=2$}
Fix $(\ell_1, \ell_2)\in \N^2$. 
In this case, there are two variables $z_1, z_2$. Let $(21)\in \mathcal{S}_2$ be the non-trivial element. 
Write the fixed points as a pair $(v_1, v_2)$, such that $v_1+v_2=v$, and $0\leq v_1\leq \ell_1$,  $0\leq v_2\leq \ell_2$.

We have
\begin{align*}
&W^{\id}_{(v_1, v_2)}=\sum_{(I_1, I_2)\vdash [1, v], |I_1|=v_1, |I_2|=v_2} \Big(\frac{y_{I_2}-y_{I_1}-2\hbar}{y_{I_1}-y_{I_2}}\Big)
(z_2-y_{I_1}+\ell_2\hbar)
(y_{I_2}-z_1+\ell_1\hbar)\in \C[\hbar][y_1, \cdots, y_v]^{\mathcal{S}_v}\otimes \C[z_1, z_2],\\
&
W^{(21)}_{(v_2, v_1)}=\sum_{(I_1, I_2)\vdash [1, v], |I_1|=v_1, |I_2|=v_2} \Big(\frac{y_{I_2}-y_{I_1}-2\hbar}{y_{I_1}-y_{I_2}}\Big)
(z_1-y_{I_1}+\ell_1\hbar)
(y_{I_2}-z_2+\ell_2\hbar)\in \C[\hbar][y_1, \cdots, y_v]^{\mathcal{S}_v}\otimes \C[z_1, z_2].
\end{align*}

\subsubsection{}
For $v=1, w=2$, there are $2$ fixed points: 
\begin{itemize}
\item
$(v_1, v_2)=(1, 0)$ which gives the partition $(I_1, I_2)=(\{1\}, \emptyset)$. 
\item
$(v_1, v_2)=(0, 1)$ which gives the partition $(I_1, I_2)=(\emptyset, \{1\})$. 
\end{itemize}
Using the formula \eqref{W:v=1}, the weight functions are then
\begin{align*}
W^{\id}_{(1, 0)}=& 
z_2-y+\ell_2\hbar, 
\,\  
W^{\id}_{(0, 1)}=
y-z_1+\ell_1\hbar. 
\end{align*}
and
\begin{align*}
W^{(21)}_{(1, 0)}=&
y-z_2+\ell_2\hbar, \,\ 
W^{(21)}_{(0, 1)}=
z_1-y+\ell_1\hbar. 
\end{align*}
At the fixed point $(1, 0)$, the weight of $V$ is $y=z_1-\ell_1\hbar$, and at the fixed point $(0, 1)$, the weight of $V$ is $y=z_2-\ell_2\hbar$. 
Therefore, 
\begin{align*}
W^{\id}_{(1, 0)}|_{(1, 0)}=& 
z_2-(z_1-\ell_1\hbar)+\ell_2\hbar =-z+(\ell_1+\ell_2)\hbar, 
\\
W^{\id}_{(1, 0)}|_{(0, 1)}=& 
z_2-(z_2-\ell_2\hbar)+\ell_2\hbar =2\ell_2\hbar
\\
W^{\id}_{(0, 1)}|_{(1, 0)}=& 
(z_1-\ell_1\hbar)-z_1+\ell_1\hbar=0,\\
W^{\id}_{(0, 1)}|_{(0, 1)}=& 
(z_2-\ell_2\hbar)-z_1+\ell_1\hbar=-z+(\ell_1-\ell_2)\hbar.
\end{align*}

Similarly, we have the restrictions
\begin{align*}
W^{(21)}_{(1, 0)}|_{(1, 0)}=& (z_1-\ell_1\hbar)-z_2+\ell_2\hbar=z-(\ell_1-\ell_2)\hbar, 
\\
W^{(21)}_{(1, 0)}|_{(0, 1)}=&(z_2-\ell_2\hbar)-z_2+\ell_2\hbar=0, 
\\
W^{(21)}_{(0, 1)}|_{(1, 0)}=&z_1-(z_1-\ell_1\hbar)+\ell_1\hbar=2 \ell_1\hbar,\\
W^{(21)}_{(0, 1)}|_{(0, 1)}=& z_1-( z_2-\ell_2\hbar)+\ell_1\hbar=z+(\ell_1+\ell_2)\hbar. 
\end{align*}
Therefore, we have
\begin{align*}
R(z)=&[W^{\id}_{\lambda}|_{\mu}]^{-1} [W^{(21)}_{\lambda}|_{\mu}] \\
=&\begin{bmatrix}
-z+(\ell_1+\ell_2)\hbar & 0\\
2\ell_2\hbar& -z+(\ell_1-\ell_2)\hbar
\end{bmatrix}^{-1}
\begin{bmatrix}
z-(\ell_1-\ell_2)\hbar & 2 \ell_1\hbar \\
0& z+(\ell_1+\ell_2)\hbar
\end{bmatrix}
\\
=&\frac{1}{(z-(\ell_1+\ell_2)\hbar)(z-(\ell_1-\ell_2)\hbar)}\begin{bmatrix}
-z+(\ell_1-\ell_2)\hbar& 0\\
-2\ell_2\hbar& -z+(\ell_1+\ell_2)\hbar
\end{bmatrix}\begin{bmatrix}
z-(\ell_1-\ell_2)\hbar & 2 \ell_1\hbar \\
0& z+(\ell_1+\ell_2)\hbar
\end{bmatrix}\\
=&\frac{1}{(z-(\ell_1+\ell_2)\hbar)(z-(\ell_1-\ell_2)\hbar)}
\begin{bmatrix}
-(z-(\ell_1-\ell_2)\hbar)^2& -2 \ell_1\hbar(z-(\ell_1-\ell_2)\hbar)\\
-2 \ell_2\hbar(z-(\ell_1-\ell_2)\hbar)&
-z^2+(\ell_1-\ell_2)^2\hbar^2
\end{bmatrix}\\
=&\frac{1}{z-(\ell_1+\ell_2)\hbar}
\begin{bmatrix}
-z+(\ell_1-\ell_2)\hbar & -2\ell_1\hbar\\
-2\ell_2\hbar & -z-(\ell_1-\ell_2)\hbar)
\end{bmatrix}
\end{align*}

In the pure spin case, when $\ell_1=\ell_2=\ell$, we have
\begin{align*}
&W^{(21)}_{(1, 0)}=\frac{-z}{z-2\ell\hbar}W^{\id}_{(1, 0)}+\frac{-2\ell\hbar}{z-2\ell\hbar} W^{\id}_{(0, 1)}
, \\
&W^{(21)}_{(0, 1)}=\frac{-2\ell\hbar}{z-2\ell\hbar}W^{\id}_{(1, 0)}+\frac{-z}{z-2\ell\hbar}W^{\id}_{(0, 1)}
\end{align*}
Therefore, when $\ell_1=\ell_2=\ell$, the transition matrix between the two bases is
\[
 \frac{1}{z-2\ell\hbar}\begin{bmatrix}
 -z& -2\ell\hbar\\
 -2\ell\hbar& -z
\end{bmatrix}=\frac{-z\Id-2\ell\hbar P}{z-2\ell\hbar}, \text{where $\Id=\begin{bmatrix}
 1& 0\\
 0& 1
\end{bmatrix}$ and $P=\begin{bmatrix}
 0& 1\\
 1& 0
\end{bmatrix}$.} 
\]
This is the well-known formula for the $R$-matrix of the Yangian $Y(\mathfrak{sl}_2)$. 
\subsubsection{}
For $v=2, w=2$, fix $(\ell_1, \ell_2)\in \N^2$, such that $\ell_1\geq 2$ and $\ell_2\geq 2$,  there are $3$ fixed points: 
\begin{itemize}
\item
$(v_1, v_2)=(2, 0)$ which gives the partition $(I_1, I_2)=(\{1, 2\}, \emptyset)$. 
\item
$(v_1, v_2)=(1, 1)$ which gives the partitions $(I_1, I_2)=(\{1\}, \{2\})$ and 
$(I_1, I_2)=(\{2\}, \{1\})$. 
\item
$(v_1, v_2)=(0, 2)$ which gives the partition $(I_1, I_2)=(\emptyset, \{1, 2\})$. 
\end{itemize}

The weight functions are then
\begin{align*}
W^{\id}_{(2, 0)}=& 
(z_2-y_{1}+\ell_2\hbar)(z_2-y_{2}+\ell_2\hbar)
\\
W^{\id}_{(1, 1)}=&
\Big(\frac{y_{2}-y_{1}-2\hbar}{y_{1}-y_{2}}\Big)
(z_2-y_{1}+\ell_2\hbar)
(y_{2}-z_1+\ell_1\hbar)
\\
&+\Big(\frac{y_{1}-y_{2}-2\hbar}{y_{2}-y_{1}}\Big)
(z_2-y_{2}+\ell_2\hbar)
(y_{1}-z_1+\ell_1\hbar)\\
W^{\id}_{(0, 2)}=&
(y_{1}-z_1+\ell_1\hbar)
(y_{2}-z_1+\ell_1\hbar)
\\
W^{(21)}_{(2, 0)}=&
(y_{1}-z_2+\ell_2\hbar)
(y_{2}-z_2+\ell_2\hbar)\\
W^{(21)}_{(1, 1)}=&
\Big(\frac{y_{2}-y_{1}-2\hbar}{y_{1}-y_{2}}\Big)
(z_1-y_{1}+\ell_1\hbar)
(y_{2}-z_2+\ell_2\hbar)
\\
&+\Big(\frac{y_{1}-y_{2}-2\hbar}{y_{2}-y_{1}}\Big)
(z_1-y_{2}+\ell_1\hbar)
(y_{1}-z_2+\ell_2\hbar)\\
W^{(21)}_{(0, 2)}=& 
(z_1-y_{1}+\ell_1\hbar)(z_1-y_{2}+\ell_1\hbar)
. 
\end{align*}

\begin{itemize}
\item
At the fixed point $(2, 0)$, the weights of $V$ are 
\[
y_1=z_1-\ell_1\hbar, y_2=z_1-(\ell_1-2)\hbar. 
\]
\item
At the fixed point $(1, 1)$, the weights of $V$ are 
\[
y_1=z_1-\ell_1\hbar, y_2=z_2-\ell_2\hbar. 
\]
\item
At the fixed point $(0, 2)$, the weights of $V$ are 
\[
y_1=z_2-\ell_2\hbar, y_2=z_2-(\ell_2-2)\hbar. 
\]
\end{itemize}

Let $z:=z_1-z_2$. 


The restrictions of the $W_\lambda^{\id}$ to the fixed points can be encoded into a matrix
\[
[W^{\id}_{\lambda}|_{\mu}]=\begin{bmatrix}
\begin{smallmatrix}
(z-(\ell_1+\ell_2)\hbar)(z-(\ell_1+\ell_2-2)\hbar & 0  & 0
\\
-(z-(\ell_1+\ell_2)\hbar)2\ell_2\hbar &-(z-(\ell_1-\ell_2-2)\hbar) (z-(\ell_1+\ell_2)\hbar) & 0 \\
2\ell_2 (2\ell_2-2)\hbar^2& 2(2\ell_2-2)\hbar (z-(\ell_1-\ell_2)\hbar)& (z-(\ell_1-\ell_2)\hbar) (z-(\ell_1-\ell_2+2)\hbar)
\end{smallmatrix}
\end{bmatrix}
\]
where the column is $\lambda$ and the row $\mu$, with the ordering $\{(2,0),(1,1),(0,2)\}$.

Similarly,


\[
[W^{(21)}_{\lambda}|_{\mu}]=\begin{bmatrix}
\begin{smallmatrix}
(z-(\ell_1-\ell_2)\hbar)(z-(\ell_1-\ell_2-2)\hbar)
&-2 (2\ell_1-2)\hbar (z-(\ell_1-\ell_2)\hbar)
&2\ell_1(2\ell_1-2)\hbar^2\\
0
&-(z-(\ell_1-\ell_2+2)\hbar)(z+(\ell_1+\ell_2)\hbar)
& 2\ell_1\hbar (z+(\ell_1+\ell_2)\hbar)\\
0
&0
&(z+(\ell_1+\ell_2)\hbar) (z+(\ell_1+\ell_2-2)\hbar)
\end{smallmatrix}
\end{bmatrix}
\]


Therefore, we have
\begin{align*}
R(z)=[W^{\id}_{\lambda}|_{\mu}]^{-1} [W^{(21))}_{\lambda}|_{\mu}]=&
\begin{bmatrix}
\frac{(z-(\ell_1-\ell_2)\hbar)(z-(\ell_1-\ell_2-2)\hbar)}{(z-(\ell_1+\ell_2)\hbar)(z-(\ell_1+\ell_2-2)\hbar)} &
-\frac{2(2\ell_1-2)\hbar(z-(\ell_1-\ell_2)\hbar)}{(z-(\ell_1+\ell_2)\hbar)(z-(\ell_1+\ell_2-2)\hbar)} 
&
\frac{2\ell_1 (2\ell_1-2)\hbar^2 }{(z-(\ell_1+\ell_2)\hbar)(z-(\ell_1+\ell_2-2)\hbar)} 
\\
&&\\
-\frac{2\ell_2\hbar(z-(\ell_1-\ell_2)\hbar)}{(z-(\ell_1+\ell_2)\hbar) (z-(\ell_1+\ell_2-2)\hbar)}
&
\frac{
z^2-2z\hbar -(2 \ell_1+\ell_1^2+2\ell_2-6\ell_1\ell_2+\ell_2^2)\hbar^2}{(z+(\ell_1+\ell_2)\hbar) (z+(\ell_1+\ell_2-2)\hbar)}
& 
-\frac{2\ell_1\hbar(z+(\ell_1-\ell_2)\hbar) }{(z-(\ell_1+\ell_2)\hbar) (z-(\ell_1+\ell_2-2)\hbar)}
\\
&&\\
\frac{2\ell_2 (2\ell_2-2)\hbar^2}{(z-(\ell_1+\ell_2)\hbar) (z-(\ell_1+\ell_2-2)\hbar) }& -\frac{2(2\ell_2-2)\hbar (z+(\ell_1-\ell_2)\hbar)}{(z-(\ell_1+\ell_2)\hbar) (z-(\ell_1+\ell_2-2)\hbar)}& \frac{(z+(\ell_1-\ell_2)\hbar) (z+(\ell_1-\ell_2+2)\hbar)}{(z-(\ell_1+\ell_2)\hbar) (z-(\ell_1+\ell_2-2)\hbar)}
\end{bmatrix}
\end{align*}

The poles of the entries of $R(z)$ are at 
$z=(\ell_1+\ell_2)\hbar$ and $z=(\ell_1+\ell_2-2)\hbar$. 
In the case when $\ell_1=\ell_2=\ell=2$,  and $\epsilon=-2\ell\hbar, \phi=2\hbar$, 
this gives the spin 1 $R$-matrix in \cite[Page 17]{BZJ20}.

\subsection{The case when $w=2$ and any $v\in \N$}

In this section, let $w=2$ and consider any $v\in \N$ and any $(\ell_1,\ell_2)\in \N^2$, such that $\ell_1\geq v$ and $\ell_2\geq v$.  

For each fixed point $(v_1, v_2)\in S(\ell_1, \ell_2)$. Consider all possible partitions $(I_1, I_2)$ of the set $[1, v]$, such that $|I_1|=v_1$ and $|I_2|=v_2$. Denote by $I^c=[1, v]\setminus I$ the complement of the set $I$ in $[1, v]$. We then have $I_2=I_1^c$. 

In this case, the formula \eqref{Wformula} gives the following 
\begin{align*}
 W^{\id}_{(v_1, v_2)}
=&\sum_{
I\subset [1, v], |I|=v_1} \Big(
\prod_{i\in I, j\in I^c}
\frac{y_{j}-y_{i}-2\hbar}{y_{i}-y_{j}}
\prod_{i\in I}
(z_2-y_{i}+\ell_2\hbar)
\prod_{j\in I^c}
(y_{j}-z_1+\ell_1\hbar\Big). 
\end{align*}

For any fixed point $(\mu_1, \mu_2)\in S(\ell_1, \ell_2)$, by Lemma \ref{lem:property}, the restriction $W^{\id}_{(v_1, v_2)}|_{(\mu_1, \mu_2)}=0$ is zero if $(v_1, v_2)<(\mu_1, \mu_2)$. 
So we consider $(\mu_1, \mu_2)\in S(\ell_1, \ell_2)$, such that $v_1\geq \mu_1$ and $v_2\leq \mu_2$.

Recall $ W^{\id}_{(v_1, v_2)}$, as an element in $\C[\hbar][y_1, \ldots, y_v]^{\mathcal{S}_v}\otimes \C[z_1, z_2]$, is symmetric in the $y$ variables. When computing the restriction $ W^{\id}_{(v_1, v_2)}|_{(\mu_1, \mu_2)}$, the ordering the substitution of $y_1, \cdots y_v$ according to \eqref{eq:fix point wts} does not matter. For convenience, we set
\begin{align*}
&y_1=z_1-\ell_1\hbar,\ y_2=z_1-(\ell_1-2)\hbar,\   
\ldots,\ y_{\mu_1}=z_1-(\ell_1-2\mu_1+2)\hbar,  \\
&
y_{\mu_1+1}=z_2-\ell_2\hbar,\ y_{\mu_1+2}=z_2-(\ell_2-2)\hbar,\ \ldots,\ 
y_{v}=z_2-(\ell_2-2\mu_2+2)\hbar. 
\end{align*}
Under such choice, the summand $W^{\id}_{(v_1, v_2), (I_1, I_2)}|_{(\mu_1, \mu_2)}$ is non-zero, we must have $(I_1, I_2)=(I_1^{st}, I_2^{st})$, where
\begin{align*}
&I_2^{st}=\left\{\mu_1+1, \mu_1+2, \ldots, \mu_1+v_2\right\}=\left[\mu_1+1, \mu_1+v_2\right], \\
&I_1^{st}=\left\{1, 2, \ldots, \mu_1, \mu_2+v_2, \ldots, v\right\}=\left[1, \mu_1\right]\sqcup \left[\mu_2+v_2, v\right]. 
\end{align*}

We now relabel the indices. For $j\in I_2^{st}$, we write the substitution as
\begin{align*}
y_j=z_2-(\ell_2-2b)\hbar, \text{with $j=\mu_1+1+b$}, \quad 0\leq b\leq v_2-1.  
\end{align*}
For $i\in I_1^{st}$, we write the substitution as
\begin{align*}
y_i=
\begin{cases}
z_1-(\ell_1-2a)\hbar, \text{with $i=1+a$}, \quad 0\leq a\leq \mu_1-1, &\text{if $i\in [1, \mu_1]$, }\\
z_2-(\ell_2-2c)\hbar, \text{with $i=\mu_1+1+c$}, \quad v_2\leq c\leq \mu_2-1,  &\text{if $i\in [\mu_2+v_2, v]$ }.  
\end{cases}
\end{align*}
We then have
\begin{align*}
W^{\id}_{(v_1, v_2)}|_{(\mu_1, \mu_2)}=&
\prod_{i\in I_1^{st}, j\in I_2^{st}}
\frac{y_{j}-y_{i}-2\hbar}{y_{i}-y_{j}}
\prod_{i\in I_1^{st}}
(z_2-y_{i}+\ell_2\hbar)
\prod_{j\in I_2^{st}}
(y_{j}-z_1+\ell_1\hbar)\Big)|_{(\mu_1, \mu_2)}\\
=&\prod_{b=0}^{v_2-1}\prod_{a=0}^{\mu_1-1}(-1)\frac{z+(\ell_2-\ell_1+2(a-b+1))\hbar}{z+(\ell_2-\ell_1+2(a-b)\hbar} 
\prod_{b=0}^{v_2-1}\prod_{c=v_2}^{\mu_2-1} (-1)\frac{c-b+1}{c-b}
\\
&\times\prod_{a=0}^{\mu_1-1} 
(-z+(\ell_1+\ell_2-2a)\hbar)\prod_{c=v_2}^{\mu_2-1}(2(\ell_2-c)\hbar )\prod_{b=0}^{v_2-1}(-z-(\ell_2-\ell_1-2b)\hbar)\\
=& (-1)^{\mu_1 v_2} \prod_{b=0}^{v_2-1}
\frac{z+(\ell_2-\ell_1+2(\mu_1-b))\hbar}{z+(\ell_2-\ell_1-2b)\hbar}(-1)^{v_2(\mu_2-v_2)}
\frac{\mu_2 !}{v_2! (\mu_2-v_2)!}\\&
\times (-1)^{\mu_1} \prod_{a=0}^{\mu_1-1}(z-(\ell_1+\ell_2-2a)\hbar) 
\frac{(\ell_2-v_2)!}{(\ell_2-\mu_2)!}
(2\hbar)^{\mu_2-v_2} (-1)^{v_2} \prod_{b=0}^{v_2-1}(z+(\ell_2-\ell_1-2b)\hbar)\\
=&(-1)^{v_1v_2+\mu_1+v_2}
{\mu_2 \choose v_2} \frac{(\ell_2-v_2)!}{(\ell_2-\mu_2)!}
(2\hbar)^{\mu_2-v_2}
\prod_{b=0}^{v_2-1}
(z+(\ell_2-\ell_1+2(\mu_1-b))\hbar) \prod_{a=0}^{\mu_1-1}(z-(\ell_1+\ell_2-2a)\hbar) 
\end{align*}
where we set $z=z_1-z_2$.

For any $\sigma\in \mathcal{S}_2$, denote by $A^{\sigma}$ the matrix whose $(i,j)$-entry is $(A^{\sigma})_{ij}=W^{\sigma}_{(v-j, j)}|_{(v-i, i)}$, for $0\leq i, j\leq v$. The matrix $A^{\id}$ can be easily inverted (see \cite[Appendix B]{BZJ20} for the inverse matrix in the pure spin case): 
\[
    (A^{\id})^{-1}_{ji} = [i\le j] (-1)^{(j+1)v+j}
{j \choose i} \frac{(\ell_2-i)!}{(\ell_2-j)!}
(2\hbar)^{j-i}
\frac{z+(\ell_2-\ell_1+2(v-2i))\hbar}
{
\prod_{b=0}^{j}
(z+(\ell_2-\ell_1+2(v-i-b))\hbar)
\prod_{a=0}^{v-i-1}(z-(\ell_1+\ell_2-2a)\hbar)
}
,\]
where the notation $[i\leq j]$ means $[\text{true}]=1$, $[\text{false}]=0$.

For $(21)\in \mathcal{S}_2$, the restriction $W^{(21)}_{(v_1, v_2)}|_{(\mu_1, \mu_2)}$ can be obtained from
$W^{\id}_{(v_1, v_2)}|_{(\mu_1, \mu_2)}$ by substituting $z_1\leftrightarrow z_2$, $v_1\leftrightarrow v_2$, $\ell_1\leftrightarrow \ell_2$, $\mu_1\leftrightarrow \mu_2$,
cf \eqref{eq:Wsigma}; therefore,
\[
(A^{(21)})_{ij}=[i\le j](-1)^{j(v+1)}
{v-i \choose v-j} \frac{(\ell_1-v+j)!}{(\ell_1-v+i)!}
(2\hbar)^{j-i}
\prod_{b=0}^{v-j-1}
(z+(\ell_2-\ell_1-2(i-b))\hbar) \prod_{a=0}^{i-1}(z+(\ell_1+\ell_2-2a)\hbar). 
\]

Finally, the $R$-matrix, which expresses the change of basis from the weight function at $\sigma=\id$ to the weight function at $\sigma=(21)$, is given by
\begin{align*}
R_{jj'}(z) &= \sum_{i=0}^v (A^{\id})^{-1}_{ji} (A^{(21)})_{ij'} \\
&= 
(-1)^{v+(j-j')(v+1)}\sum_{i=0}^{\min(j,j')} {j\choose i}{v-i\choose v-j'}\frac{(\ell_2-i)!(\ell_1-v+j')!}{(\ell_2-j)!(\ell_1-v+i)!}
(2\hbar)^{j'+j-2i}\\
&\times
\frac{
(z+(\ell_2-\ell_1+2(v-2i))\hbar)
\prod_{b=0}^{v-j'-1}
(z+(\ell_2-\ell_1-2(i-b))\hbar) \prod_{a=0}^{i-1}(z+(\ell_1+\ell_2-2a)\hbar) 
}
{
\prod_{b=0}^{j}
(z+(\ell_2-\ell_1+2(v-i-b))\hbar)
\prod_{a=0}^{v-i-1}(z-(\ell_1+\ell_2-2a)\hbar)
}
\end{align*}

\begin{remark}
    This $R$-matrix satisfies the Yang--Baxter equation, as follows from general arguments. Assume $w=3$ and compare
    $W^{321}_{\lambda}$ and
    $W^{\id}_{\lambda}$. From the definition \eqref{eq:defR} we have the identity
    \[
    R_{\id,321} = R_{\id,213}R_{213,231}R_{231,321}=R_{\id,132}R_{132,312}R_{312,321}
    \]
    which expresses
    the braid relation $321=(12)(23)(12)=(23)(12)(23)$. In particular, the $R$-matrices in the r.h.s.\ can easily be identified with the $w=2$ $R$-matrix computed above, with appropriate substitution of parameters.

    Alternatively, the Yang--Baxter equation follows from the lattice model formulation of the next section.
\end{remark}


\tikzset{gbline/.style={line width=0.8mm,draw=black!35!green}}
\tikzset{gline/.style={line width=0.3mm,draw=black!35!green}}
\section{The lattice model}\label{sec:lm}
As in the rest of the paper, fix the data of two nonnegative integers $w$ and $v$, of a $w$-tuple of positive integers $(\ell_1,\ldots,\ell_w)$, and let $(v_1,\ldots,v_w)$ vary over the set $S(\ell_1, \ldots, \ell_w)$ (cf~\eqref{eq:fixed points}). 

Consider the following {\em lattice model}. 
Given a  $v\times w$ rectangular grid (i.e., a rectangular planar graph with external edges included), number its columns (resp.\ rows) $1,\ldots,w$ from left to right (resp.\ $1,\ldots,v$ from top to bottom). The space of states $\mathcal C_{(v_1,\ldots,v_w)}$ is the set of labellings of the edges of the grid, in such a way that
\begin{enumerate}[(i)]
    \item Labels of horizontal edges are in $\{0,1\}$;
    \item Labels of vertical edges in column $j$ are in $\{0,1,\ldots,\ell_j\}$;
    \item The labels on external edges are fixed: they are $0$ on the South and East sides, $1$ on the West side, and $v_j$ on the North side in column $j$;
    \item Around each vertex, the sum of labels of left and bottom adjacent edges is equal to the sum of labels of right and top adjacent edges.
\end{enumerate}
See Figure~\ref{fig:lm} for an example.
For pictorial purposes, we think of the labels as ``occupations numbers'' of each edge, the condition (iv) ensuring that they can be represented as continuous paths.
In particular label $0$ is drawn as an empty edge, and label $1$ as a single path crossing the edge, and we shall often omit such labels when they can be inferred from the context.

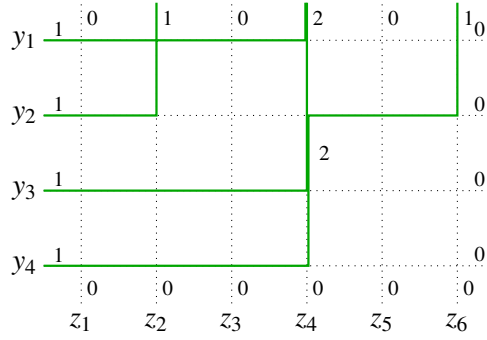
\begin{figure}
\begin{tikzpicture}
    \draw[dotted] (.5,.5) grid (6.5,4.5);
    \draw[gline] (.5,4) -- (2,4) -- (2,4.5);
    \draw[gline] (.5,3) -- (2,3) -- (2,4) -- (3.98,4) -- (3.98,4.5);
    \draw[gline] (.5,2) -- (1,2) -- (4,2) -- (4,4.5);
    \draw[gline] (.5,1) -- (4.02,1) -- node[right,pos=0.75] {$\ss 2$} (4.02,3) -- (6,3) -- (6,4.5);
    \foreach \x/\m in {1/0,2/1,3/0,4/2,5/0,6/1} {
    \node at (\x,.25) {$z_\x$};
    \node at (\x+.15,.7) {$\ss 0$};
    \node at (\x+.15,4.3) {$\ss \m$};
    }
    \foreach \y in {1,2,3,4} {
    \node at (.25,5-\y) {$y_\y$};
    \node at (.7,\y+.15) {$\ss 1$};
    \node at (6.3,\y+.15) {$\ss 0$};
    }
\end{tikzpicture}
    \caption{An example of state of a lattice model with $v=4$, $w=6$, and $(v_1,\ldots,v_w)=(0,1,0,2,0,1)$.}\label{fig:lm}
\end{figure}

We assign the variable $y_i$ to the horizontal line $i$ and $z_j$ to the vertical line $j$, as on Figure~\ref{fig:lm}. 
To each state $c$ in $\mathcal C_{(v_1,\ldots,v_w)}$ we associate its ``Boltzmann weight'' $F(c)$ of the form
\[
F(c)=\prod_{i=1}^v \prod_{j=1}^w f^{(\ell_j)}(c(i,j),y_i-z_j)\ \in \C[\hbar,y_1,\ldots,y_v,z_1,\ldots,z_w]
\]
where $c(i,j)$ symbolically denotes the configuration of labels of the edges around vertex $(i,j)$ in $c$, 
and the function $f$ is given by
\begin{equation}\label{eq:Lmat}
\begin{tikzpicture}[baseline=(lab)]
    \matrix[column sep={2.5cm,between origins},row sep=0.25cm]{
    \node[align=left] {configuration $c$\\ around a vertex $(i, j)$};
    &
    \draw[dotted] (-.5,0) -- (.5,0);
\draw[gbline] (0,-0.5) node[right] {$\ss m$} -- (0,0.5) node[right] {$\ss m$};
& 
\draw[gbline] (0,-0.5) node[right] {$\ss m$} -- (0,0.5) node[right] {$\ss m$};
\draw[gline] (-0.5,0) -- (0.5,0);
& 
    \draw[dotted] (-.5,0) -- (0,0);
\draw[gbline] (0,-0.5) node[right] {$\ss m$} -- (0,0.5) node[right] {$\ss m-1$};
\draw[gline] (0.05,-0.5) -- (0.05,0) -- (0.5,0);
& 
    \draw[dotted] (0,0) -- (.5,0);
\draw[gbline] (0,-0.5) node[right] {$\ss m$} -- (0,0.5) node[right] {$\ss m+1$};
\draw[gline] (-0.05,0.5) -- (-0.05,0) -- (-0.5,0);
\\
\node (lab) {$f^{(l)}(c(i, j),u)$};
&
\node{$u-(l-2m)\hbar$}; & \node{$u+(l-2m)\hbar$}; & \node{$2m\hbar$}; & \node{$2(l-m)\hbar$};
\\
};
\end{tikzpicture}
\end{equation}

Finally, we write
\begin{equation}\label{eq:defWt}
\tilde W_{(v_1,\ldots,v_w)} = \sum_{c\in C_{(v_1,\ldots,v_w)}} F(c). 
\end{equation}

\begin{example}
    For $v=w=2$, and $(v_1,v_2)=(1,1)$ here are the two states of the lattice model
    and their Boltzmann weights:
    
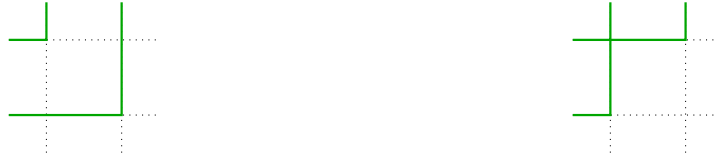
\[
    \begin{tikzpicture}
    \matrix[column sep={7.5cm,between origins},row sep=0.2cm]{
    \draw[dotted] (0.5,0.5) grid (2.5,2.5);
    \draw[gline] (0.5,2) -- (1,2) -- (1,2.5);
    \draw[gline] (0.5,1) -- (2,1) -- (2,2.5);
    &
    \draw[dotted] (0.5,0.5) grid (2.5,2.5);
    \draw[gline] (0.5,2) -- (1,2) -- (1,2.5);
    \draw[gline] (0.5,1) -- (1,1) -- (1,2) -- (2,2) -- (2,2.5);
    \\
    \node at (1.5,0) {$4\ell_1\ell_2\hbar^2 (y_1-z_2-(\ell_2-2)\hbar)(y_2-z_1+\ell_1\hbar)$};
    &
    \node at (1.5,0) {$4\ell_1\ell_2\hbar^2 (y_1-z_1+(\ell_1-2)\hbar)(y_2-z_2-\ell_2\hbar)$};
    \\
        };
    \end{tikzpicture}
    \]
    Summing them up leads to
    \begin{align*}
    \tilde W_{(1,1)}=4\ell_1\ell_2\hbar^2 &
    \Big(
    2y_{1}y_{2}-\left(z_{1}+z_{2}-\ell_{1}\hbar+\ell_{2}\hbar\right)y_{1}-\left(z_{1}+z_{2}-\ell_{1}\hbar+\ell_{2}\hbar\right)y_{2}\\
    &+2\left(z_{1}z_{2}+\left(\ell_{2}-1\right)\hbar z_{1}-\left(\ell_{1}-1\right)\hbar z_{2}-\left(\ell_{1}\ell_{2}-\ell_{1}-\ell_{2}\right)\hbar^2\right)
    \Big)\\
    =4\ell_1\ell_2\hbar^2&
\Big(\frac{y_{2}-y_{1}-2\hbar}{y_{1}-y_{2}}\Big)
(z_2-y_{1}+\ell_2\hbar)
(y_{2}-z_1+\ell_1\hbar)
\\
&+4\ell_1\ell_2\hbar^2\Big(\frac{y_{1}-y_{2}-2\hbar}{y_{2}-y_{1}}\Big)
(z_2-y_{2}+\ell_2\hbar)
(y_{1}-z_1+\ell_1\hbar)\\
=4\ell_1\ell_2\hbar^2 &W^{\id}_{(1, 1)}. 
    \end{align*}
\end{example}

\begin{theorem}\label{thm:tildeW}
One has
\[
\tilde W_{(v_1,\ldots,v_w)} = \prod_{j=1}^w (-1)^{v_j(v+w-j-\sum_{s=1}^j v_s)}(2\hbar)^{v_j} \frac{\ell_j!}{(\ell_j-v_j)!}\ W^{\id}_{(v_1,\ldots,v_w)}. 
\]
\end{theorem}
\begin{proof}
    Consider the vector space $V=(K^2)^{\otimes v}$, where $K=\C(\hbar,y_1,\ldots,y_v,z_1,\ldots,z_w)$, endowed with its standard basis $\{\ket{a_1,\ldots,a_v}\mid a_i\in \{0,1\}, i=1, \dots, v\}$ and dual basis $\{\bra{a_1,\ldots,a_v}\mid a_i\in \{0,1\}, i=1, \dots, v\}$ of $V^{\vee}$.
    We interpret each column of a state of our lattice model as a linear operator acting on $V$, in such a way that concatenating columns corresponds to the natural linear operator product; explicitly, 
    let $T^{(l)}_{m,n}(x)\in \End(V)=V^{\vee}\otimes V$ be the linear operator on $V$ pictorially represented as
    \[
    T^{(l)}_{m,n}(x)=\begin{tikzpicture}[baseline=(current  bounding  box.center)]
        \draw[very thick] (1,.5) node[below] {$x$} -- (1,3.5);
        \draw (.5,1) node[left] {$y_v$} -- (1.5,1);
        \draw (.5,2) node[left] {$\vdots$} -- (1.5,2);
        \draw (.5,3) node[left] {$y_1$} -- (1.5,3);
        \node at (1.15,.7) {$\ss n$};
        \node at (1.15,3.3) {$\ss m$};
        \node at (1.2,1.5) {$\ss\le l$};
        \node at (1.2,2) {$\ss\vdots$};
        \node at (1.2,2.5) {$\ss\le l$};
    \end{tikzpicture}
    \]
    that is, such that the coefficient $\bra{b_1,\ldots,b_v} T^{(l)}_{m,n}(x) \ket{a_1,\ldots,a_v}$ is zero if there are no configuration of labels of a $v\times 1$ grid, where the external labels are fixed to be $m,a_1,\ldots,a_v,n,b_v,\ldots,b_1$ clockwise, and the internal edges are in $\{0,\ldots,l\}$, satisfying the conditions above, and the product $\prod_{i=1}^v f^{(l)}(c(i),y_i-x)$ if there exists a (necessarily unique) configuration $c$ that does. Here we omit the label of the column as the vertex only depends on $i\in [1, v]$. 

We then have
    \begin{equation}\label{eq:W1}
    \tilde W_{(v_1,\ldots,v_w)} = \bra{1,\ldots,1} T^{(\ell_1)}_{v_1,0}(z_1) \ldots T^{(\ell_w)}_{v_w,0}(z_w)\ket{0,\ldots,0}. 
    \end{equation}
    
    \begin{lemma}\label{lem:lm}
        There exists a basis $\{\widetilde{\ket{a_1,\ldots,a_v}}\mid a_{i}\in \{0, 1\}, i=1, \ldots, v\}$ 
        of $V$ (with dual basis $\{\widetilde{\bra{a_1,\ldots,a_v}}\mid a_{i}\in \{0, 1\}, i=1, \ldots, v\}$
        ) such that
        \begin{enumerate}[(i)]
            \item $\widetilde{\ket{0,\ldots,0}}=\ket{0,\ldots,0}$.
            \item $\widetilde{\bra{1,\ldots,1}}=\bra{1,\ldots,1}$.
        \item For all $0\le m\le l$ and any $x\in K$, the matrix of $T_{m,0}^{(l)}(x)$ is given by
        \begin{align*}
      &\widetilde{\bra{b_1,\ldots,b_v}} T^{(l)}_{m,0}(x) \widetilde{\ket{a_1,\ldots,a_v}}\\
        =&\delta_{\sum_{i=1}^v (b_i-a_i),m}
        (2\hbar)^m \frac{l!}{(l-m)!}
      \prod_{\substack{i,j=1\\b_i=1,\,a_i=0\\b_j=a_j=0}}^v \frac{y_i-y_j-2\hbar}{y_i-y_j}
         \ 
        \prod_{i=1}^v
        \begin{cases}
            y_i-x-l\hbar & \text{for $b_i=a_i=0$}
            \\
            y_i-x+l\hbar & \text{for $b_i=a_i=1$}
            \\
            1
            & \text{for $b_i=1,\ a_i=0$}
            \\
            0 & \text{for $b_i=0,\ a_i=1$}
        \end{cases}
        \end{align*}
    \end{enumerate}
    \end{lemma}
    The proof of this lemma requires more lattice model technology than the rest of this paper, and we defer it to Appendix~\ref{app:lmlemma}.


 We now use Lemma \ref{lem:lm} to prove Theorem \ref{thm:tildeW}. 
    For each $j=1, \ldots, w-1$, write \[\Id=\sum_{\{(a^{(j)}_1,\ldots,a^{(j)}_v)\in \{0, 1\}^{v} )\} } \widetilde{\ket{a^{(j)}_1,\ldots,a^{(j)}_v}}\widetilde{\bra{a^{(j)}_1,\ldots,a^{(j)}_v}}\in \End(V). 
    \]
    
    Then, by \eqref{eq:W1}, we have
    \begin{multline*}
    \tilde W_{(v_1,\ldots,v_w)} = \sum_{\left\{a^{(1)}_1,\ldots,a^{(w-1)}_v\mid \begin{smallmatrix}a^{(j)}_i\in \{0, 1\},\\ i\in [1, v], j\in [1, w-1]\}
    \end{smallmatrix}
    \right\}} 
    \bra{1,\ldots,1}
    T^{(\ell_1)}_{v_1,0}(z_1)\widetilde{\ket{a^{(1)}_1,\ldots,a^{(1)}_v}}\widetilde{\bra{a^{(1)}_1,\ldots,a^{(1)}_v}} T^{(\ell_2)}_{v_2,0}(z_2) \ldots
    \\\ldots T^{(\ell_{w-1})}_{v_{w-1},0}(z_{w-1})\widetilde{\ket{a^{(w-1)}_1,\ldots,a^{(w-1)}_v}}\widetilde{\bra{a^{(w-1)}_1,\ldots,a^{(w-1)}_v}} T^{(\ell_w)}_{v_w,0}(z_w)
    \ket{0,\ldots,0}. 
    \end{multline*}    
Use the convention that $a_i^{(0)}=1$ and $a_i^{(w)}=0$ for all $i=1, \ldots, v$. 
    According to Lemma~\ref{lem:lm}, for fixed $i\in [1, v]$, the value of $a_i^{(j)}$, $j\in [0, w]$, can only decrease, thus jumping from $1$ to $0$ at exactly one value of $j$. Furthermore, for fixed $j$, because of the factor $\delta_{\sum_{i=1}^v (b_i-a_i),v_j}$, exactly $v_j$ jumps occur.
    One can therefore replace the summation over the intermediate basis elements with a summation over subsets $I_j=\big\{i\,:\,a_i^{(j-1)}=1,\ a_i^{(j)}=0\big\}$, $j=1,\ldots, w$:
    \begin{align*}
    &\tilde W_{(v_1,\ldots,v_w)} \\
    =& \sum_{\substack{I_1,\ldots,I_w\subseteq\{1,\ldots,v\}\\ |I_j|=v_j\\ I_1,\ldots,I_w\text{ disjoint}}} \prod_{j=1}^w (2\hbar)^{v_j} \frac{\ell_j!}{(\ell_j-v_j)!} \prod_{1\leq k<j\leq w} (y_{I_k}-z_j-\ell_j\hbar) \frac{y_{I_k}-y_{I_j}+2\hbar}{y_{I_k}-y_{I_j}}\prod_{w\geq k>j\geq 1}(y_{I_k}-z_j+\ell_j\hbar)\\
  =&\left(\prod_{j=1}^w (2\hbar)^{v_j} \frac{\ell_j!}{(\ell_j-v_j)!}\right)
  \sum_{\substack{I_1,\ldots,I_w\subseteq\{1,\ldots,v\}\\ |I_j|=v_j\\ I_1,\ldots,I_w\text{ disjoint}}}  \prod_{1\leq k<j\leq w} (-1)^{v_k+v_kv_j}(z_j-y_{I_k}+\ell_j\hbar)\frac{y_{I_j}-y_{I_k}-2\hbar}{y_{I_k}-y_{I_j}}
  \prod_{w\geq k>j\geq 1}(y_{I_k}-z_j+\ell_j\hbar)\\
   =&\left(\prod_{j=1}^w (-1)^{v_j(v+w-j-\sum_{s=1}^j v_s)}(2\hbar)^{v_j} \frac{\ell_j!}{(\ell_j-v_j)!}\right) W^{\id}_{(v_1, \dots, v_w)}. 
    \end{align*}
\end{proof}



\section{Relation to the stable envelope}
Let $p_\lambda\in S(\ell_1, \ldots, \ell_w)$ be a torus fixed point. 
We have the embeddings
\[
p_{\lambda} \xrightarrow{i_{\lambda}} \tilde{X}(v, \vec{1}_w)^{st} \xrightarrow{j}  \tilde{X}(v, \vec{1}_w),  
\]
where $i_{\lambda}$ is a closed embedding and $j$ is an open embedding. Pushforward along $i_{\lambda}$ and pullback along $j$ gives the following commutative diagram. 
\[
\xymatrix{
H^*_{T_w\times \C^*}
(p_{\lambda}) \ar[r]^(0.3){i_{\lambda*}} \ar@{=}[d]& 
H^*_{c, T_w\times G_v \times \C^*}
(\tilde{X}(v, \vec{1}_w)^{st}, \varphi_\bold{w})^{\vee} \ar[d]^{s_*}& H^*_{c, T_w\times G_v \times \C^*}
(\tilde{X}(v, \vec{1}_w), \varphi_\bold{w})^{\vee} \ar[l]_{j^*}\ar[d]^{s_*}\\
H^*_{T_w\times \C^*}
(p_{\lambda}) \ar[r]^(0.3){i_{\lambda*}} & 
H^*_{c, T_w\times G_v \times \C^*}
(\tilde{X}(v, \vec{1}_w)^{st}, \C)^{\vee} & H^*_{c, T_w\times G_v \times \C^*}
(\tilde{X}(v, \vec{1}_w), \C)^{\vee} \ar[l]_{j^*}
}
\]
By the dimension reduction in Proposition \ref{prop:dimred}, we have the isomorphism 
\[
H^*_{c,\C^*\times T_w}(\tilde{X}(v, \vec{1}_w)^{st}/G_V, \varphi_{\bf{w}})^{\vee}
\cong H^{*}_{c,\C^*\times T_w}(\mu_{v, w}^{-1}(0)^{st}\times_{G_v} \C^{wv}, \C)^{\vee}
\cong H^{*-wv}_{c,\C^*\times T_w}(\mathfrak{M}^+(v, \vec{1}_w), \C)^{\vee}.  
\]
The vertical maps $s_*$ are pushforward along the closed embeddings $s: \mu_{v, w}^{-1}(0)^{st}\times_{G_v} \C^{wv}\inj  \tilde{X}(v, \vec{1}_w)^{st}$ and $\mu_{v, w}^{-1}(0)\times_{G_v} \C^{wv}\inj  \tilde{X}(v, \vec{1}_w)$. 

We have the following tangent complexes (see \cite[Section 3.1]{CZ23} in detail for the derived structure of the critical loci and the derived tangent complexes)
\begin{align*}
&[T(\mathfrak{M}^+(v, \vec{1}_w))]
=[-\End(\calV)+(\calV^\vee\otimes \calV\otimes \C_{\epsilon}+ \sum_i\calW_i^{\vee}\otimes \calV\otimes \C_{a^*_i})-\sum_i\calW_i^{\vee}\otimes V\otimes \C_{\epsilon^{\ell_i} a^*_i}], \\
&[T\big((\mu_{v, w}^{-1}(0)^{st}\times_{G_v} \C^{wv})/G_v\big)]
=[-\End(\calV)+(\calV^\vee\otimes \calV\otimes \C_{\epsilon}+ \sum_i\calW_i^{\vee}\otimes \calV\otimes \C_{a^*_i}+\sum_i\calV^{\vee}\otimes \calW_i\otimes \C_{a_i})\\&\phantom{1234567890}-\sum_i\calW_i^{\vee}\otimes V\otimes \C_{\epsilon^{\ell_i} a^*_i}], \\
&[T\tilde{X}(v, \vec{1}_w)]
=[-\End(\calV)+(\calV^\vee\otimes \calV\otimes \C_{\epsilon}+ \sum_i\calW_i^{\vee}\otimes \calV\otimes \C_{a^*_i}+\sum_i\calV^{\vee}\otimes \calW_i\otimes \C_{a_i})]. 
\end{align*}
Therefore, $s^*s_*$ is given by the multiplication of the Euler class of the bundle $\sum_i\calW_i^{\vee}\otimes V\otimes \C_{\epsilon^{\ell_i} a^*_i}$: 
\begin{equation}\label{eq:ss}
s^*s_*=\prod_{i=1}^{w} \prod_{j=1}^{v}(-z_i+y_j-\ell_i\hbar)=(-1)^{vw}\prod_{i=1}^{w} \prod_{j=1}^{v}(z_i-y_j+\ell_i\hbar). 
\end{equation}
It is also the same as the Euler class of the bundle $\Hom(\calW, \calV)$ coming from the arrows $a_i, i\in [1, w]$, up to the sign $(-1)^{vw}$. 

Similar as the discussion in \cite[\S 2.5]{RTV15}, for the torus $T_w$, let $\mathfrak{t}_{w}:=\text{Lie} (T_w)$ be the Lie algebra of $T_w$. We then have the real vector space embedding
\[
(\mathfrak{t}_{w})_{\mathbb{R}}=\text{Cochar}(T_w)\otimes_{\Z}\mathbb{R}\subset \mathfrak{t}_{w}. 
\]
Let $z_1, \ldots, z_w$ be the standard basis of the dual space $\mathfrak{t}_{w}^{\vee}$. 
We have the root hyperplanes $z_i-z_j=0$ of $(\mathfrak{t}_{w})_{\mathbb{R}}$, for any $1\leq i\neq j \leq w$. This gives the open chambers $\mathfrak{C}_{\sigma}=\{x\in (\mathfrak{t}_{w})_{\mathbb{R}}\mid z_{\sigma(1)}(x)> \cdots > z_{\sigma(w)}(x)\}$, such that
\[
(\mathfrak{t}_{w})_{\mathbb{R}}\setminus \bigcup_{1\leq i\neq j \leq w}\{z_i-z_j=0\}=\bigsqcup_{\sigma\in \mathcal{S}_w} \mathfrak{C}_{\sigma}. 
\]

As conjectured in \cite{BZJ20}, the stable envelope $\{{S}_{\lambda}^{\sigma}\mid \lambda\in S(\ell_1, \ldots, \ell_w)\}$, where $\sigma\in \calS_{w}$ is the labeling of the chamber $\mathfrak{C}_{\sigma}$, should be unique classes in $H^*_{c, T_w\times G_v \times \C^*}(\tilde{X}(v, \vec{1}_w)^{st}, \varphi_\bold{w})^{\vee}\cong H^{*-wv}_{c,\C^*\times T_w}(\mathfrak{M}^+(v, \vec{1}_w), \C)^{\vee}$ satisfying the characterisation in \cite[Conjecture 2]{BZJ20}. 

The weight functions $\{W_{\lambda}^{\sigma}\mid \lambda\in S(\ell_1, \ldots, \ell_w)\}$ are elements in $H^*_{c, T_w\times G_v \times \C^*}
(\tilde{X}(v, \vec{1}_w), \C)^{\vee}$. Pushforward along $s$ of the stable envelope, it is expected that  
\begin{equation}\label{eq:SW}
s_*({S}_{\lambda}^{\sigma})=j^*(W_{\lambda}^{\sigma}). 
\end{equation}
Formally, by \eqref{eq:ss}, we have
\[
(-1)^{vw}\prod_{i=1}^{w} \prod_{j=1}^{v}(z_i-y_j+\ell_i\hbar)({S}_{\lambda}^{\sigma})=
s^*s_*({S}_{\lambda}^{\sigma})=s^*j^*(W_{\lambda}^{\sigma}). 
\] Therefore, after inverting the Euler class $\prod_{i}\prod_{j}(z_i-y_j+\ell_i\hbar)$ and by \eqref{eq:SW}, it suggests that the class
${S}_{\lambda}^{\sigma}$ satisfies the equality 
\begin{equation}\label{eq:Stable}
{S}_{\lambda}^{\sigma}=\frac{s^*j^*(W_{\lambda}^{\sigma})}{(-1)^{vw}\prod_{i=1}^{w} \prod_{j=1}^{v}(z_i-y_j+\ell_i\hbar)}  
\end{equation}
in the localised ring $H^*_{c, T_w\times G_v \times \C^*}(\tilde{X}(v, \vec{1}_w)^{st}, \C)^{\vee}\otimes_{H_{T_w\times G_v\times \C^*}(\pt)}H_{T_w\times G_v\times \C^*}(\pt)[\prod_{i=1}^{w} \prod_{j=1}^{v}\frac{1}{z_i-y_j+\ell_i\hbar}]$. We remark that the formula \eqref{eq:Stable} is not enough for the support condition in \cite[Conjecture 2]{BZJ20}. We will investigate this in the future \cite{CYZJ}.

\appendix
\section{Proof of Lemma~\ref{lem:lm}}\label{app:lmlemma}

In this appendix we use more of the diagrammatic calculus that was introduced in
\S\ref{sec:lm}. In what follows, when we draw a diagram, we view it as a linear operator
where the source is the tensor product of vector spaces associated to the incoming
lines (with out conventions, lines coming from the right; or from the top for vertical lines)
and the target is associated to the outgoing lines (going to the left; or going down for vertical lines). In particular, the labels of internal edges are implicitly summed over,
corresponding to natural operator product. Labelling the external edges corresponds to picking a particular entry of those linear operators.
Pictorially, the following two diagrams mean some linear operator from $V_1\otimes V_2$ to $V_3\otimes V_4$. 
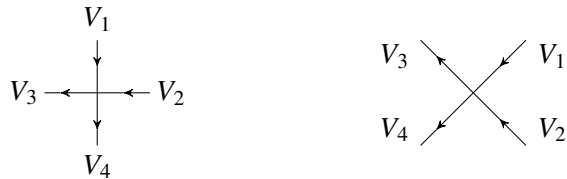
\[
\begin{tikzpicture}
    \matrix[column sep={2.5cm,between origins},row sep=0.25cm,cells={scale=.7}]{
                \draw[->] (0,1) -- (0,0.5);
        \draw[->] (0,0.5) -- (0, -.7);
        \draw (0,-.7) -- (0,-1);
        \draw[->] (1,0) -- (0.5,0);
   \draw[->] (0.5,0) -- (-.7,0);
   \draw (-.7,0) -- (-1,0);
        \node at (0, 1.4) {$V_1$};
         \node at (1.4, 0) {$V_2$};
            \node at (-1.4, 0) {$V_3$};
         \node at (0, -1.4) {$V_4$};
         &\,\ & 
                \draw[->] (1,1) -- (0.5,0.5);
        \draw[->] (0.5,0.5) -- (-.7,-.7);
        \draw (-.7,-.7) -- (-1,-1);
        \draw[<-] (-.7,.7) -- (0.5,-0.5);
        \draw (-.7,.7) -- (-1,1);
        \draw[->] (1,-1) -- (0.5,-0.5);
        \node at (1.5,.7) {$V_1$};
         \node at (1.5,-.7) {$V_2$};
            \node at (-1.5,.7) {$V_3$};
         \node at (-1.5,-.7) {$V_4$};
   \\ };
\end{tikzpicture}\]

\subsection{The six-vertex $R$-matrix}
Let $K$ be the base field. Given $u\in K$, define the six-vertex $R$-matrix $\check R(u)$ to be the following linear operator on $(K^2)^{\otimes 2}$. Denote the matrix entries by $\bra{b_1b_2} \check R(u)\ket{a_1a_2}$, where $b_1, b_2, a_1, a_2\in \{0, 1\}$. The following 6 entries of $\check R(u)$ are best described diagrammatically as
\begin{equation}\label{eq:6v}
\begin{tikzpicture}
    \matrix[column sep={4.5cm,between origins},row sep=0.25cm,cells={scale=.7}]{
        \draw[gline] (-1,-1) -- (1,1) (1,-1) -- (-1,1);
        &
        \draw[dotted] (-1,-1) -- (1,1) (1,-1) -- (-1,1);
        &
        \draw[dotted] (-1,-1) -- (1,1); \draw[gline] (1,-1) -- (-1,1);
        \\
           \node{$\bra{11} \check R(u)\ket{11}=1$};
        &
        \node{$\bra{00} \check R(u)\ket{00}=1$};
        &
        \node{$\bra{10} \check R(u)\ket{01}=\beta(u)$};\\
        \draw[gline] (-1,-1) -- (1,1); \draw[dotted] (1,-1) -- (-1,1);
        &
        \draw[dotted] (-1,-1) -- (0,0) -- (1,-1); \draw[gline] (1,1) -- (0,0) -- (-1,1);
        &
        \draw[gline] (-1,-1) -- (0,0) -- (1,-1); \draw[dotted] (1,1) -- (0,0) -- (-1,1);
        \\   
        \node{$\bra{01} \check R(u)\ket{10}=\beta(u)$};
        &
        \node{$\bra{10} \check R(u)\ket{10}=\gamma(u)$};
        &
        \node{$\bra{01} \check R(u)\ket{01}=\gamma(u)$};
        \\
    };
\end{tikzpicture}
\end{equation}
where
\begin{align*}
    \beta(u)&=\frac{u}{u-2\hbar},
    \\
    \gamma(u)&=\frac{2\hbar}{u-2\hbar}.
\end{align*}
All other entries of $\check R(u)$ are zero. 

When we want to talk about the linear operator itself without specifying a particular matrix entry, we'll write
\[
\check R(u_1-u_2)=
\begin{tikzpicture}[scale=.7,baseline=(current  bounding  box.center)]
        \draw (-1,-1) node[left] {$u_2$} -- (1,1) (1,-1) -- (-1,1) node[left] {$u_1$};
\end{tikzpicture}
\]
where a black line means a line whose label is unspecified. Furthermore, the
parameter $u$ is always chosen to be the difference of two parameters each attached to each of the two lines crossing.

The $R$-matrix satisfies several key identities. Firstly,
the Yang--Baxter equation is an identity in $\text{End}((K^2)^{\otimes 3})$ 
that holds for all $u_1,u_2,u_3\in K$:
\begin{equation}\label{eq:ybe6v}
\begin{tikzpicture}[baseline=-3pt,y=2cm,x=.8cm,rotate=90]
\draw[rounded corners=4mm] (-0.5,0.5) node[left] {$u_3$} -- (0.75,0) -- (1.5,-0.5) (0.5,0.5) node[left] {$u_2$} -- (0,0) -- (0.5,-0.5)  (1.5,0.5) node[left] {$u_1$} -- (0.75,0) -- (-0.5,-0.5);
\end{tikzpicture}
=
\begin{tikzpicture}[baseline=-3pt,y=2cm,x=.8cm,rotate=90]
\draw[rounded corners=4mm] (-0.5,0.5) node[left] {$u_3$} -- (0.25,0) -- (1.5,-0.5) (0.5,0.5)  node[left] {$u_2$} -- (1,0) -- (0.5,-0.5) (1.5,0.5) node[left] {$u_1$} -- (0.25,0) -- (-0.5,-0.5);
\end{tikzpicture}
\end{equation}
Secondly, the unitarity equation holds as an identity in $\text{End}((K^2)^{\otimes 2})$:
\begin{equation}\label{eq:unit6v}
\begin{tikzpicture}[baseline=-3pt,rotate=90,scale=1.1]
\draw[rounded corners=4mm] (-0.5,1) node[left] {$u_2$} -- (0.5,0) -- (-0.5,-1) (0.5,1) node[left] {$u_1$} -- (-0.5,0) -- (0.5,-1);
\end{tikzpicture}
=
\begin{tikzpicture}[baseline=-3pt,rotate=90,scale=1.1]
\draw (-0.5,1) node[left] {$u_2$} -- (-0.5,-1) (0.5,1) node[left] {$u_1$} -- (0.5,-1);
\end{tikzpicture}
\end{equation}
where the right hand side is nothing but the identity operator.

The last relation involves both the six-vertex $R$-matrix introduced above, and the so-called $L$-matrix implicitly defined in \S\ref{sec:lm} as an operator from $K^{l+1}\otimes K^2$ to $K^2\otimes K^{l+1}$:
\[
\check L^{(k)}(u_1-u_2) = \begin{tikzpicture}[baseline=-3pt]
    \draw[very thick] (0,-1) node[below] {$u_2$} -- (0,1);
    \draw (-1,0) node[left] {$u_1$} -- (1,0);
\end{tikzpicture}
\]
where the nonzero entries of the operator are given by \eqref{eq:Lmat}. These weights depend on the extra parameter $l$, which is the upper bound on the labels of the thick line; this parameter will not be mentioned on the picture for the sake of simplicity.

The $RLL$ relation, which is another form of the Yang--Baxter equation, then states that
\begin{equation}\label{eq:RLL}
\begin{tikzpicture}[baseline=-3pt,scale=.6]
\draw[rounded corners=4mm] (-1,-1) node[left] {$u_2$} -- (1,1) -- (3,1);
\draw[rounded corners=4mm] (-1,1) node[left] {$u_1$} -- (1,-1) -- (3,-1);
\draw[very thick] (2,-2) node[below] {$u_3$} -- (2,2);
\end{tikzpicture}
=
\begin{tikzpicture}[baseline=-3pt,scale=.6]
\draw[rounded corners=4mm] (-1,-1) node[left] {$u_2$} -- (1,-1) -- (3,1);
\draw[rounded corners=4mm] (-1,1) node[left] {$u_1$} -- (1,1) -- (3,-1);
\draw[very thick] (0,-2) node[below] {$u_3$} -- (0,2);
\end{tikzpicture}
\end{equation}

All these identities can be checked by direct computation, and we leave it to the
interested reader.

Finally, given a permutation $\sigma \in \mathcal S_v$, 
pick any wiring diagram of it and consider it as a product of $R$-matrices; e.g., if $\sigma=231$,
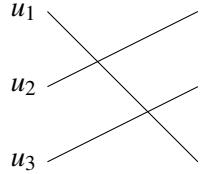
\[
\begin{tikzpicture}[xscale=2]
    \draw (0,1) node[left] {$u_3$} -- (1,2);
    \draw (0,2) node[left] {$u_2$} -- (1,3);
    \draw (0,3) node[left] {$u_1$} -- (1,1);    
\end{tikzpicture}
\]
Thanks to equations \eqref{eq:ybe6v}--\eqref{eq:unit6v}, the resulting element of $\End((K^2)^{\otimes v})$ is independent of the choice of wiring diagram (reduced or not).
We denote it symbolically by
\[
\check R_\sigma(u_1,\ldots,u_v) = 
\begin{tikzpicture}[scale=.7,baseline=(current  bounding  box.center)]
    \draw (0,0) rectangle (2,4); \node at (1,2) {$\sigma$};
    \draw (2,1) -- (3,1) node[right] {$u_{\sigma(v)}$}
    (2,2) -- (3,2) node[right] {$\vdots$} (2,3) -- (3,3) node[right] {$u_{\sigma(1)}$};
    \foreach\y/\s in {3/{u_1},2/{\vdots},1/{u_v}} \draw (0,\y) -- (-1,\y) node[left] {$\s$};
\end{tikzpicture}
\]
In the example above,
$\check R_{(231)}(u_1, u_2, u_3):=\check R_{(12)}(u_1-u_2)\check R_{(23)}(u_1-u_3)\in \End((K^2)^{\otimes 3})$.

We comment here that the summation over all the labels in the interior of a wiring diagram is equivalent to a summation over {\em subwords}\/ of the corresponding word of $\sigma$. We shall not need to spell out this statement more explicitly, except for the following important consequence: consider binary strings $(a_1,\ldots,a_v)$ and $(b_1,\ldots,b_v)$ with the same number of 0s and 1s, and the smallest permutation (in Bruhat order) $\rho$ in $\mathcal S_v$ such that $a_i=b_{\rho(i)}$ (which is easily shown to exist). Then
$\bra{b_1,\ldots,b_v} \check R_\sigma \ket{a_1,\ldots,a_v}$ is zero unless $\sigma$ is greater or equal to $\rho$ in Bruhat order. We shall use this property repeatedly in what follows.

As a simple consequence of the formalism above, we have the following result:
\begin{lemma}
The function $\tilde W_{(v_1,\ldots,v_w)}$ defined in \eqref{eq:defWt} is a {\em symmetric}\/ polynomial in $y_1,\ldots,y_v$.
\end{lemma}
\begin{proof}
    With our updated graphical conventions, one can rewrite \eqref{eq:defWt} explicitly as
    \[
    \tilde W_{(v_1,\ldots,v_w)} = 
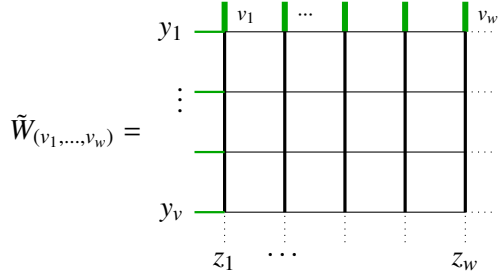
\begin{tikzpicture}[baseline=(current  bounding  box.center),scale=.8]
    \foreach\x/\s/\t in {1/{v_1}/{z_1},2/{\cdots}/{\cdots},3//,4//,5/{v_w}/{z_w}} {
    \draw[very thick] (\x,1) -- (\x,4);
    \draw[dotted] (\x,.5) node[below] {$\t$} -- (\x,1);
    \draw[gbline] (\x,4) -- node[right] {$\ss\s$} (\x,4.5);
    }
    \foreach\y/\s in {1/{y_v},2/,3/{\vdots},4/{y_1}} {
    \draw (1,\y) -- (5,\y);
    \draw[dotted] (5,\y) -- (5.5,\y);
    \draw[gline] (.5,\y) node[left] {$\s$} -- (1,\y);
    }
    \end{tikzpicture}
    \]
    where once again the dependence on $\ell_1,\ldots,\ell_w$ is suppressed on the picture.

    Now pick $i\in\{1,\ldots,v-1\}$ and consider the following diagram
    
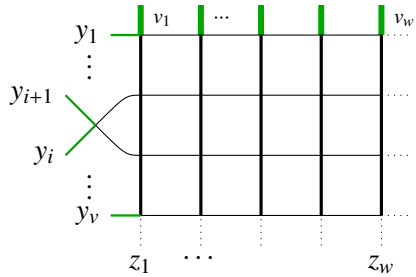
\[
    \begin{tikzpicture}[baseline=(current  bounding  box.center),scale=.8]
    \foreach\x/\s/\t in {1/{v_1}/{z_1},2/{\cdots}/{\cdots},3//,4//,5/{v_w}/{z_w}} {
    \draw[very thick] (\x,1) -- (\x,4);
    \draw[dotted] (\x,.5) node[below] {$\t$} -- (\x,1);
    \draw[gbline] (\x,4) -- node[right] {$\ss\s$} (\x,4.5);
    }
    \foreach\y/\s in {1/{y_v},4/{y_1}} {
    \draw (1,\y) -- (5,\y);
    \draw[dotted] (5,\y) -- (5.5,\y);
    \draw[gline] (.5,\y) node[left] {$\s$} -- (1,\y);
    }
    \draw[dotted] (5,2) -- (5.5,2);
    \draw[dotted] (5,3) -- (5.5,3);
    \draw (5,2) -- (1,2) .. controls (.75,2) .. (.25,2.5);
    \draw (5,3) -- (1,3) .. controls (.75,3) .. (.25,2.5);
    \draw[gline] (.25,2.5) -- (-.25,3) node[left] {$y_{i+1}$} (.25,2.5) -- (-.25,2) node[left] {$y_{i}$};
    \node at (.15,3.6) {$\vdots$};
    \node at (.15,1.6) {$\vdots$};
    \end{tikzpicture}
    \]
    which differs from the preceding one by the addition of an extra $R$-matrix to the left between rows $i$ and $i+1$. According to \eqref{eq:6v}, the only possible labels on the right side of this $R$-matrix (i.e., that produce a nonzero weight) are $1$s, and the corresponding weight is $1$. Therefore this diagram is also equal to  $\tilde W_{(v_1,\ldots,v_w)}$.
    
    On the other hand, apply repeatedly the RLL relation \eqref{eq:RLL} to it: one obtains
    \[
    \tilde W_{(v_1,\ldots,v_w)} =
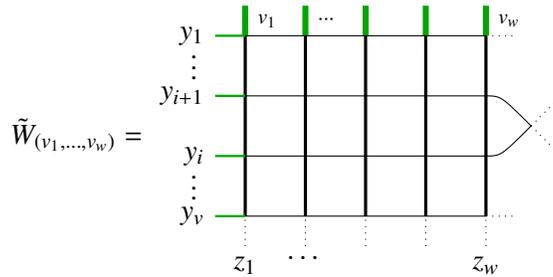
    \begin{tikzpicture}[baseline=(current  bounding  box.center),scale=.8]
    \foreach\x/\s/\t in {1/{v_1}/{z_1},2/{\cdots}/{\cdots},3//,4//,5/{v_w}/{z_w}} {
    \draw[very thick] (\x,1) -- (\x,4);
    \draw[dotted] (\x,.5) node[below] {$\t$} -- (\x,1);
    \draw[gbline] (\x,4) -- node[right] {$\ss\s$} (\x,4.5);
    }
    \foreach\y/\s in {1/{y_v},4/{y_1}} {
    \draw (1,\y) -- (5,\y);
    \draw[dotted] (5,\y) -- (5.5,\y);
    \draw[gline] (.5,\y) node[left] {$\s$} -- (1,\y);
    }
    \draw (1,2) -- (5,2) .. controls (5.25,2) .. (5.75,2.5);
    \draw (1,3) -- (5,3) .. controls (5.25,3) .. (5.75,2.5);
    \draw[dotted] (5.75,2.5) -- (6.25,3) (5.75,2.5) -- (6.25,2);
    \node at (.15,3.6) {$\vdots$};
    \node at (.15,1.6) {$\vdots$};
    \draw[gline] (.5,2) node[left] {$y_{i}$} -- (1,2);
    \draw[gline] (.5,3) node[left] {$y_{i+1}$} -- (1,3);
    \end{tikzpicture}
    \]
    By the same argument, the only possible labels on the left side of the $R$-matrix are $0$s, with corresponding weight $1$. We then recognize the same diagram that we started with,
    but with $y_i$ and $y_{i+1}$ switched.
\end{proof}

\subsection{Definition of the $F$-basis}
We define the tilded basis of Lemma~\ref{lem:lm}, usually known as $F$-basis.\footnote{Note that the $F$-basis has the simple geometric
meaning of fixed point basis, so that the computation of \S\ref{ssec:appcalc} is a thinly disguised equivariant localisation calculation, similar in spirit to those of \S\ref{sec:proof} and \S\ref{sec:shuffle}.}

In this section we abbreviate binary strings $(a_1,\ldots,a_v)$, $a_i\in \{0, 1\}$, with the notation $a$. 
Given such an $a=(a_1,\ldots,a_v)$ with say $v'$ $0$s and $v''$ $1$s, where $v'+v''=v$, 
we define
\begin{align*}
\text{sort}(a)&=(\underbrace{0,\ldots,0}_{v'},\underbrace{1,\ldots,1}_{v''})
\\
\text{rsort}(a)&=(\underbrace{1,\ldots,1}_{v''},\underbrace{0,\ldots,0}_{v'})
\end{align*}
Now consider any permutation $\sigma\in \mathcal S_v$ such that $\text{sort}(a)_{i}=a_{\sigma(i)}$ for all $i$, and then define
\[
\widetilde{\ket{a}} = \check R_\sigma(y_1,\ldots,y_v) \ket{\text{sort}(a)} =
\begin{tikzpicture}[scale=.7,baseline=(current  bounding  box.center)]
    \draw (0,0) rectangle (2,5); \node at (1,2.5) {$\sigma$};
    \draw[gline] (2,1) -- (3,1) node[right] {$y_{\sigma(v)}$} (2,2) -- (3,2);
    \draw[dotted] (2,3) -- (3,3) node[right] {$\vdots$} (2,4) -- (3,4) node[right] {$y_{\sigma(1)}$};
    \foreach\y/\s in {4/{y_1},3/{\vdots},2/,1/{y_v}} \draw (0,\y) -- (-1,\y) node[left] {$\s$};
    \draw[decorate,decoration={brace}] (4.3,2) -- node[right] {$v''$} (4.3,1);
    \draw[decorate,decoration={brace}] (4.3,4) -- node[right] {$v'$} (4.3,3);
\end{tikzpicture}
\]
This definition is independent of the choice of representative $\sigma$ in the left coset space $\mathcal S_v/(\mathcal S_{v'}\times \mathcal S_{v''})$ because varying $\sigma$ within the coset would amount to inserting/removing extra crossings on the right of the picture between two lines with equal occupancy,
and the first two weights of \eqref{eq:6v} are $1$.

In particular, note that if $a=\text{sort}(a)$, $\widetilde{\ket{a}}=\ket{a}$,
proving part \textit{(i)} of Lemma~\ref{lem:lm}.

Similarly, consider any permutation $\tau$ such that $a_{i} = \text{rsort}(a)_{\tau(i)}$ for all $i$, and then define
\[
\widetilde{\bra{a}} =\kappa_a^{-1} \bra{\text{rsort}(a)} \check R_\tau(y_{\bar\tau(1)},\ldots,y_{\bar\tau(v)})  = \kappa_a^{-1}
\begin{tikzpicture}[scale=.7,baseline=(current  bounding  box.center)]
    \draw (0,0) rectangle (2,5); \node at (1,2.5) {$\tau$};
    \draw[dotted] (0,1) -- (-1,1) node[left] {$y_{\bar\tau(v)}$} (0,2) -- (-1,2);
    \draw[gline] (0,3) -- (-1,3) node[left] {$\vdots$} (0,4) -- (-1,4) node[left] {$y_{\bar\tau(1)}$};
    \foreach\y/\s in {4/{y_1},3/{\vdots},2/,1/{y_v}} \draw (2,\y) -- (3,\y) node[right] {$\s$};
    \draw[decorate,decoration={brace}] (-2.3,1) -- node[left] {$v'$} (-2.3,2);
    \draw[decorate,decoration={brace}] (-2.3,3) -- node[left] {$v''$} (-2.3,4);
\end{tikzpicture}
\]
where $\bar\tau=\tau^{-1}$, and $\kappa_a\in K$ is a constant that
we choose of the form
\[
\kappa_a= \prod_{i,j:\,a_i=1,\, a_j=0} \beta(y_i-y_j)
\]
for reasons that will become clear right below (Lemma \ref{lem:Fduality}). 
Once again, this definition is independent of the choice of representative $\tau$ in the right coset space $(\mathcal S_{v''}\times \mathcal S_{v'})\backslash\mathcal S_v$.

\begin{example}\label{ex:tildeex}
    Consider $a=(0,0,1,0)$, $\text{sort}(a)=(0,0,0,1)$, so we can choose $\sigma=(34)$. Then
\[
\widetilde{\ket{0,0,1,0}}
=
\begin{tikzpicture}[scale=.7,baseline=(current  bounding  box.center)]
    \draw[gline] (2,1) -- (3,1) node[right] {$y_3$};
    \draw[dotted] (2,2) -- (3,2) (2,3) -- (3,3) node[right] {$\vdots$} (2,4) -- (3,4) node[right] {$y_1$};
    \foreach\y/\s in {4/{y_1},3/{\vdots},2/,1/{y_4}} \draw (0,\y) -- (-1,\y) node[left] {$\s$};
    \draw (0,1) .. controls (1,1) and (1,2) .. (2,2) (0,2) .. controls (1,2) and (1,1) .. (2,1) (0,3) -- (2,3) (0,4) -- (2,4); 
\end{tikzpicture}
=\beta(y_3-y_4)\ket{0,0,1,0}+\gamma(y_3-y_4)\ket{0,0,0,1}
\]
    If $a'=(0,1,0,0)$, with same sort as $a$, $\sigma=1342=(23)(34)$, so
    \[
\widetilde{\ket{0,1,0,0}}=
\begin{tikzpicture}[scale=.7,baseline=(current  bounding  box.center)]
    \draw[gline] (2,1) -- (3,1) node[right] {$y_2$};
    \draw[dotted] (2,2) -- (3,2) (2,3) -- (3,3) node[right] {$\vdots$} (2,4) -- (3,4) node[right] {$y_1$};
    \foreach\y/\s in {4/{y_1},3/{\vdots},2/,1/{y_4}} \draw (0,\y) -- (-1,\y) node[left] {$\s$};
    \draw (0,1) .. controls (1,1) and (1,2) .. (2,2) (0,2) .. controls (1,2) and (1,3) .. (2,3) (0,3) .. controls (1,3) and (1,1) .. (2,1) (0,4) -- (2,4); 
\end{tikzpicture}
=\begin{matrix}\beta(y_2-y_3)\beta(y_2-y_4)\ket{0,1,0,0}
+\gamma(y_2-y_3)\beta(y_2-y_4)\ket{0,0,1,0}\\
+\gamma(y_2-y_4)\ket{0,0,0,1}
\end{matrix}
    \]
    Similarly, if $b=(1,1,0,1)$, $\text{rsort}(b)=(1,1,1,0)$, $\sigma=(34)$, we have
    \[    \widetilde{\bra{1,1,0,1}}=\kappa_b^{-1}\Big(\beta(y_4-y_3)\bra{1,1,0,1}+\gamma(y_4-y_3)\bra{1,1,1,0}\Big)
    \] 
\end{example}

As a warm-up for the main calculation of this appendix, let us show the following\def\ba#1#2{\widetilde{\bra{#1\,}}\!\widetilde{\left.\,#2\right>}}
\begin{lemma}\label{lem:Fduality}
Given binary strings $a=(a_1,\ldots,a_v)$ and $b=(b_1,\ldots,b_v)$, one has the duality statement
\[
\ba{b}{a}=\delta_{b,a}.
\]
\end{lemma}
In what follows we use the following well-known fact about double cosets in $(\mathcal S_{\bar v''}\times \mathcal S_{\bar v'})\backslash\mathcal S_v/(\mathcal S_{v'}\times \mathcal S_{v''})$: they possess a unique minimal element which is a {\em bigrassmannian}\/ permutation, namely
a permutation $\sigma$ such that the descent set $\{i\in [1, v-1]\mid \sigma(i)>\sigma(i+1)\}$ of $\sigma$ (resp.\ $\sigma^{-1}$) is contained in $\{v'\}$ (resp.\ $\{\bar v''\}$), and such bigrassmannian permutations
are in bijection with cosets. 
\begin{proof}
    Because $R$-matrices preserve the occupation number, we may assume without loss of generality that $a$ and $b$ have the same number of $0$s and $1$s, otherwise the l.h.s.\ of the equality is trivially zero.
    Choosing $\sigma$ so that $\text{sort}(a)_i=a_{\sigma(i)}$ and $\tau$ such that $b_i = \text{rsort}(b)_{\tau(i)}$, one has
    \[
    \kappa_b\ba{b}{a}=
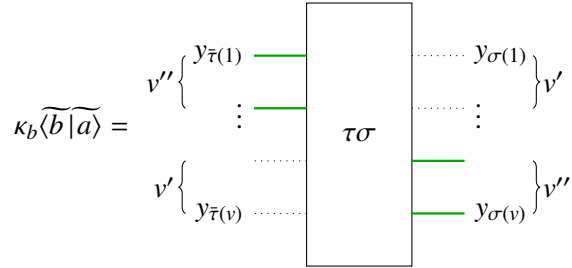
\begin{tikzpicture}[scale=.7,baseline=(current  bounding  box.center)]
    \draw (0,0) rectangle (2,5); \node at (1,2.5) {$\tau\sigma$};
    \draw[gline] (2,1) -- (3,1) node[right] {$y_{\sigma(v)}$} (2,2) -- (3,2);
    \draw[dotted] (2,3) -- (3,3) node[right] {$\vdots$} (2,4) -- (3,4) node[right] {$y_{\sigma(1)}$};
    \draw[dotted] (0,1) -- (-1,1) node[left] {$y_{\bar\tau(v)}$} (0,2) -- (-1,2);
    \draw[gline] (0,3) -- (-1,3) node[left] {$\vdots$} (0,4) -- (-1,4) node[left] {$y_{\bar\tau(1)}$};
    \draw[decorate,decoration={brace}] (-2.3,1) -- node[left] {$v'$} (-2.3,2);
    \draw[decorate,decoration={brace}] (-2.3,3) -- node[left] {$v''$} (-2.3,4);
    \draw[decorate,decoration={brace}] (4.3,2) -- node[right] {$v''$} (4.3,1);
    \draw[decorate,decoration={brace}] (4.3,4) -- node[right] {$v'$} (4.3,3);
\end{tikzpicture}
    \]
    For this to be nonzero, $\tau\sigma$ has to be greater or equal in Bruhat order to the bigrassmannian permutation $\rho$ of the form $\rho(i)=i+v''$ for $1\le i\le v'$, $\rho(i)=i-v'$ for $v'+1\le i\le v$. But this is the greatest bigrassmannian permutation among all $(\mathcal S_{v''}\times \mathcal S_{v'})\backslash\mathcal S_v/(\mathcal S_{v'}\times \mathcal S_{v''})$ double coset representatives, so as a double coset, $\tau\sigma=\rho$. This immediately implies
    $$b_i=\text{rsort}(b)_{\tau(i)} = \text{rsort}(b)_{\rho(\bar\sigma(i))} =\text{sort}(b)_{\bar\sigma(i)}= \text{sort}(a)_{\bar\sigma(i)}=a_i$$ where we have used the fact that $a$ and $b$ have same content. Finally, we compute
    \[
    \kappa_a\ba{a}{a}=
\begin{tikzpicture}[scale=.7,baseline=(current  bounding  box.center)]
    \draw[gline,rounded corners] (-1,3) node[left] {$\vdots$} -- (0,3) -- (2,1) -- (3,1) node[right] {$y_{\sigma(v)}$} (-1,4) node[left] {$y_{\bar\tau(1)}$} -- (0,4) -- (2,2) -- (3,2);
    \draw[dotted,rounded corners] (-1,1) node[left] {$y_{\bar\tau(v)}$} -- (0,1) -- (2,3) -- (3,3) node[right] {$\vdots$} (-1,2) -- (0,2) -- (2,4) -- (3,4) node[right] {$y_{\sigma(1)}$};
    \draw[decorate,decoration={brace}] (-2.3,1) -- node[left] {$v'$} (-2.3,2);
    \draw[decorate,decoration={brace}] (-2.3,3) -- node[left] {$v''$} (-2.3,4);
    \draw[decorate,decoration={brace}] (4.3,2) -- node[right] {$v''$} (4.3,1);
    \draw[decorate,decoration={brace}] (4.3,4) -- node[right] {$v'$} (4.3,3);
\end{tikzpicture}
 = \prod_{i,j:\,a_i=1,\, a_j=0} \beta(y_i-y_j). 
    \]
\end{proof}

In particular, note that if $a=\text{rsort}(a)$, $\widetilde{\bra{a}}=\kappa_a^{-1}\bra{a}$,
proving part \textit{(ii)} of Lemma~\ref{lem:lm}.

\subsection{Matrix entries of $T^{(l)}_{m,0}(x)$}\label{ssec:appcalc}
We now come to part \textit{(iii)} of Lemma~\ref{lem:lm}. The setup is similar to that of the proof of Lemma~\ref{lem:Fduality}, except now the contents of $a$ and $b$ must differ:
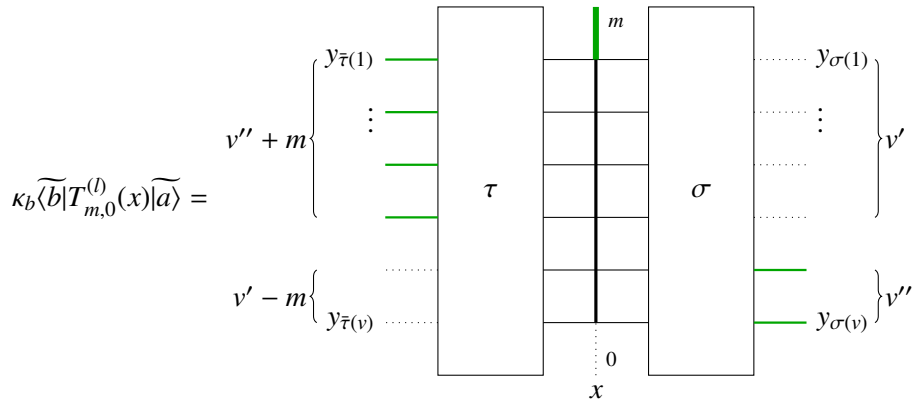
\[
\kappa_b\widetilde{\bra{b}} T^{(l)}_{m,0}(x) \widetilde{\ket{a}} = 
\begin{tikzpicture}[scale=.7,baseline=(current  bounding  box.center)]
    \draw (0,0) rectangle (2,7); \node at (1,3.5) {$\sigma$};
    \draw[gline] (2,1) -- (3,1) node[right] {$y_{\sigma(v)}$} (2,2) -- (3,2);
    \draw[dotted] (2,3) -- (3,3) (2,4) -- (3,4) (2,5) -- (3,5) node[right] {$\vdots$} (2,6) -- (3,6) node[right] {$y_{\sigma(1)}$};
    \foreach\y in {1,...,6} \draw (0,\y) -- (-2,\y);
    \draw[decorate,decoration={brace}] (4.3,2) -- node[right] {$v''$} (4.3,1);
    \draw[decorate,decoration={brace}] (4.3,6) -- node[right] {$v'$} (4.3,3);
    \begin{scope}[xshift=-4cm]
\draw (0,0) rectangle (2,7); \node at (1,3.5) {$\tau$};
    \draw[dotted] (0,1) -- (-1,1) node[left] {$y_{\bar\tau(v)}$} (0,2) -- (-1,2);
    \draw[gline] (0,3) -- (-1,3) (0,4) -- (-1,4) (0,5) -- (-1,5) node[left] {$\vdots$} (0,6) -- (-1,6) node[left] {$y_{\bar\tau(1)}$};
    \draw[decorate,decoration={brace}] (-2.3,1) -- node[left] {$v'-m$} (-2.3,2);
    \draw[decorate,decoration={brace}] (-2.3,3) -- node[left] {$v''+m$} (-2.3,6);
    \end{scope}
    \draw[very thick] (-1,1) -- (-1,6); \draw[dotted] (-1,0) node[below] {$x$}  -- node[right,pos=0.3] {$\ss0$} (-1,1); \draw[gbline] (-1,6) -- node[right,pos=.7] {$\ss m$} (-1,7);
    \end{tikzpicture}
\]
where $a$ has $v'$ $0$s and $v''$ $1$s, whereas $b$ has $v'-m$ $0$s and $v''+m$ $1$s. We recall that the dependence on $l$ lies in the fact that labels on the thick line are in $\{0,\ldots,l\}$, and in the weights \eqref{eq:Lmat}.

By repeated application of the RLL relation \eqref{eq:RLL}, we can move the vertical line to the right:
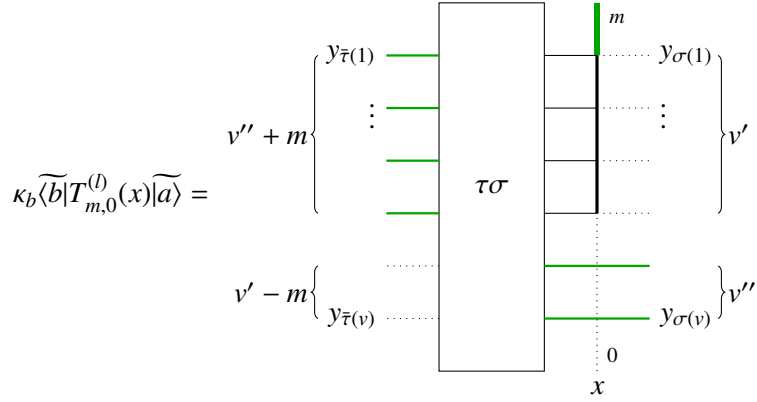
\[
\kappa_b\widetilde{\bra{b}} T^{(l)}_{m,0}(x) \widetilde{\ket{a}} =
\begin{tikzpicture}[scale=.7,baseline=(current  bounding  box.center)]
    \draw (0,0) rectangle (2,7); \node at (1,3.5) {$\tau\sigma$};
    \draw[gline] (2,1) -- (4,1) node[right] {$y_{\sigma(v)}$} (2,2) -- (4,2);
    \draw[dotted] (3,3) -- (4,3) (3,4) -- (4,4) (3,5) -- (4,5) node[right] {$\vdots$} (3,6) -- (4,6) node[right] {$y_{\sigma(1)}$};
    \draw[dotted] (0,1) -- (-1,1) node[left] {$y_{\bar\tau(v)}$} (0,2) -- (-1,2);
    \draw[gline] (0,3) -- (-1,3) (0,4) -- (-1,4) (0,5) -- (-1,5) node[left] {$\vdots$} (0,6) -- (-1,6) node[left] {$y_{\bar\tau(1)}$};
    \draw[decorate,decoration={brace}] (-2.3,1) -- node[left] {$v'-m$} (-2.3,2);
    \draw[decorate,decoration={brace}] (-2.3,3) -- node[left] {$v''+m$} (-2.3,6);
    \draw[decorate,decoration={brace}] (5.3,2) -- node[right] {$v''$} (5.3,1);
    \draw[decorate,decoration={brace}] (5.3,6) -- node[right] {$v'$} (5.3,3);
    \draw[very thick] (3,3) -- (3,6); \draw[dotted] (3,0) node[below] {$x$}  -- node[right,pos=0.3] {$\ss0$} (3,1) -- (3,3); \draw[gbline] (3,6) -- node[right,pos=.7] {$\ss m$} (3,7);
    \foreach\y in{3,...,6} \draw (2,\y) -- (3,\y);
\end{tikzpicture}
\]
We have indicated on the bottom right of the picture that some labels are ``forced'' by the defining condition (iv) for a lattice model state.

Once again, by direct inspection, the only bigrassmannian permutation $\pi$ that is large enough in Bruhat order for the r.h.s.\ to be nonzero
is given by $\pi(i)=i$ for $1\le i\le m$, $\pi(i)=i+v''$ for $m+1\le i\le v'$,
and $\pi(i)=i-v'+m$ for $v'+1\le i\le v$. Therefore, as a double coset, $\tau\sigma=\pi$.
Now compare
\begin{align*}
  a_i &= \text{sort}(a)_{\bar\sigma(i)}
  \\
      &= (\underbrace{0,\ldots,0}_{v'},\underbrace{1,\ldots,1}_{v''})_{\bar\pi(\tau(i))}
  \\
      &= (\underbrace{0,\ldots,0}_m,\underbrace{1,\ldots,1}_{v''},\underbrace{0,\ldots,0}_{v'-m})_{\tau(i)}
        \intertext{with}
  \\
  b_i &= (\underbrace{1,\ldots,1}_{v''+m},\underbrace{0,\ldots,0}_{v'-m})_{\tau(i)}. 
\end{align*}
We see that $a_i=b_i$ except if $i\in I=\bar\tau(\{1,\ldots,m\})$, in which case $a_i=0$, $b_i=1$. In
particular, $b_i\ge a_i$, which is consistent with part \textit{(iii)} of Lemma~\ref{lem:lm}.

We now compute
\begin{align*}
\kappa_b\widetilde{\bra{b}} T^{(l)}_{m,0}(x) \widetilde{\ket{a}} &=
\begin{tikzpicture}[scale=.7,baseline=(current  bounding  box.center)]
    \draw[dotted] (3,5) -- (4,5) node[right] {$\vdots$} (3,6) -- (4,6) node[right] {$y_{\sigma(1)}$};
    \draw[dotted,rounded corners] (4,3) -- (2,3) -- (0,1) -- (-1,1) node[left] {$y_{\bar\tau(v)}$} (4,4) -- (2,4) -- (0,2) -- (-1,2);
    \draw[gline,rounded corners] (4,1) node[right] {$y_{\sigma(v)}$} -- (2,1) -- (0,3) -- (-1,3)  (4,2) -- (2,2) -- (0,4) -- (-1,4) (3,5) -- (-1,5) node[left] {$\vdots$} (3.02,6) -- (-1,6) node[left] {$y_{\bar\tau(1)}$};
    \draw[decorate,decoration={brace}] (-2.3,1) -- node[left] {$v'-m$} (-2.3,2);
    \draw[decorate,decoration={brace}] (-2.3,3) -- node[left] {$v''$} (-2.3,4);
    \draw[decorate,decoration={brace}] (-2.3,5) -- node[left] {$m$} (-2.3,6);
    \draw[decorate,decoration={brace}] (5.3,2) -- node[right] {$v''$} (5.3,1);
    \draw[decorate,decoration={brace}] (5.3,6) -- node[right] {$v'$} (5.3,3);
    \draw[gline] (3.02,5) -- node[right,pos=.4] {$\ss1$} (3.02,6) node[right,yshift=1.5mm] {$\ss\vdots$}; \draw[dotted] (3,0) node[below] {$x$}  -- node[right,pos=0.3] {$\ss0$} (3,1) -- node[right,pos=.9] {$\ss0$} (3,5); \draw[gbline] (3,6) -- node[right,pos=.7] {$\ss m$} (3,7);
\end{tikzpicture}
\\&=
(2\hbar)^m \frac{l!}{(l-m)!}
\prod_{i:\, b_i=a_i=0} (y_i-x-l\hbar)
\prod_{i:\, b_i=a_i=1} (y_i-x+l\hbar)
\prod_{\substack{i:\,a_i=1\\j: b_j=0}} \beta(y_i-y_j), 
\end{align*}
where the last equality follows from the weights in \eqref{eq:Lmat} and \eqref{eq:6v}.

Finally we take care of the product of $\beta$s:
\begin{equation*}
\kappa_b^{-1}
\prod_{\substack{i:\,a_i=1\\j:\, b_j=0}} \beta(y_i-y_j)
=
\frac{
\displaystyle
\prod_{\substack{i:\,a_i=1\\j:\, b_j=0}}
\beta(y_i-y_j)
}
{
\displaystyle
\prod_{\substack{i:\,b_i=1\\j:\, b_j=0}}
\beta(y_i-y_j)
}
\\
=\frac{1}{\displaystyle
\prod_{\substack{i:\,b_i=1,\,a_i=0\\j:\, b_j=a_j=0}}\beta(y_i-y_j)
} 
\end{equation*}
where we have used the fact that $b_i\ge a_i$ for all $i$.

\begin{example}
    The meaning of Lemma~\ref{lem:lm} is that the entries of $T^{(l)}_{m,0}(x)$ are significantly simpler in the $F$-basis than in the standard basis. For example, using the same bras and kets as in Example~\ref{ex:tildeex}, one has configurations of the form
\[
\begin{tikzpicture}[scale=.7,baseline=(current  bounding  box.center)]
\draw[dotted] (1,0) -- (1,1) -- (1,2) (0,2) -- (1,2) -- (1,3) -- (2,3) (1,4) -- (2,4);
\draw[gline] (0,1) -- (1,1) -- (1,2) -- (2,2) (0,3) -- (1,3) -- (1,4) (0,4) -- (1,4);
\draw[gbline] (1,4) -- (1,5) node[right] {$\ss 2$};
\end{tikzpicture}
\]
\end{example}
which means $\bra{b}T^{(l)}_{2,0}(x)\ket{a}\ne0$.
However, according to Lemma~\ref{lem:lm},
$\widetilde{\bra{b}}T^{(l)}_{2,0}(x)\widetilde{\ket{a}}=0$, and indeed,
\begin{align*}
\kappa_b\widetilde{\bra{b}}T^{(l)}_{m,0}(x)\widetilde{\ket{a}}
&=\beta(y_4-y_3)\beta(y_3-y_4)\begin{tikzpicture}[scale=.7,baseline=(current  bounding  box.center)]
\draw[dotted] (1,0) -- (1,1) -- (1,2) (0,2) -- (1,2) -- (1,3) -- (2,3) (1,4) -- (2,4);
\draw[gline] (0,1) -- (1,1) -- (1,2) -- (2,2) (0,3) -- (1,3) -- (1,4) (0,4) -- (1,4);
\draw[gbline] (1,4) -- node[right] {$\ss 2$} (1,5) ;
\end{tikzpicture}
+\beta(y_4-y_3)\gamma(y_3-y_4)\begin{tikzpicture}[scale=.7,baseline=(current  bounding  box.center)]
\draw[dotted] (1,0) -- (1,3) -- (2,3) (0,2) -- (2,2) (1,4) -- (2,4);
\draw[gline] (0,1) -- (2,1) (0,3) -- (1,3) -- (1,4) (0,4) -- (1,4);)
\draw[gbline] (1,4) -- node[right] {$\ss 2$} (1,5);
\end{tikzpicture}
+\gamma(y_4-y_3)\beta(y_3-y_4)\begin{tikzpicture}[scale=.7,baseline=(current  bounding  box.center)]
\draw[dotted] (1,0) -- (1,3) -- (2,3) (0,1) -- (2,1) (1,4) -- (2,4);
\draw[gline] (0,2) -- (2,2) (0,3) -- (1,3) -- (1,4) (0,4) -- (1,4);
\draw[gbline] (1,4) -- node[right] {$\ss 2$} (1,5);
\end{tikzpicture}
\\
&=\frac{(2l\hbar)(2(l-1)\hbar)}{(y_4-y_3-2\hbar)(y_3-y_4-2\hbar)}
\times
\\
&\left(
-(y_3-y_4)^2\hbar(2\hbar)(2\ell\hbar)
+(y_4-y_3)(2\hbar)(y_3-x-l\hbar)(y_4-x+l\hbar)
+(2\hbar)(y_3-y_4)(y_3-x+l\hbar)(y_4-x-l\hbar)
\right)
\\
&=0
\end{align*}
(we have skipped the $\gamma(y_4-y_3)\gamma(y_3-y_4)$ term
because it corresponds to an invalid configuration of the lattice model).

On the other hand, we leave it as an exercise to the reader to check that by summing the 5 diagrams contributing to
$\widetilde{\bra{b}}T^{(l)}_{2,0}(x)\widetilde{\ket{a'}}$,
one finds
\[
\widetilde{\bra{b}}T^{(l)}_{2,0}(x)\widetilde{\ket{a'}}
=
(2\hbar)^2 l(l-1)
\frac{(y_1-y_3-2\hbar)(y_4-y_3-2\hbar)}{(y_1-y_3)(y_4-y_3)}
(y_2-x+l\hbar)(y_3-x-l\hbar)
\]
in agreement with Lemma~\ref{lem:lm}.

\subsection{The $R$-matrix}
\begin{remark}
    In view of Theorem~\ref{thm:tildeW}, the definition of the $R$-matrix \eqref{eq:defRgeneral} for an elementary transposition $\sigma=(j\,j+1)$ can be traced back to the following pictorial identity:
    \[
    \begin{tikzpicture}[baseline=(current  bounding  box.center),scale=.8]
    \foreach\x/\s/\t in {1/{v_1}/{z_1},2/{\cdots}/{\cdots},5/{v_w}/{z_w}} {
    \draw[very thick] (\x,1) -- (\x,4);
    \draw[dotted] (\x,.5) node[below] {$\t$} -- (\x,1);
    \draw[gbline] (\x,4) -- node[right] {$\ss\s$} (\x,5);
    }
    \draw[gbline] (3.5,4.6)--node[right] {$\ss v_{j+1}$} (4,5);\draw[very thick] (3,1) -- (3,4) .. controls (3,4.25) .. (3.5,4.6); 
    \draw[dotted] (3,.5) node[below] {$z_{j}$}  -- (3,1);
     \draw[gbline] (3.5,4.6)--node[left] {$\ss v_j$} (3,5);\draw[very thick] (4,1) -- (4,4) .. controls (4,4.25) .. (3.5,4.6);
    \draw[dotted] (4,.5) node[below] {$z_{j+1}$} -- (4,1);
    
    \foreach\y/\s in {1/{y_v},2/,3/{\vdots},4/{y_1}} {
    \draw (1,\y) -- (5,\y);
    \draw[dotted] (5,\y) -- (5.5,\y);
    \draw[gline] (.5,\y) node[left] {$\s$} -- (1,\y);
    }
    \end{tikzpicture}
    =
    \begin{tikzpicture}[baseline=(current  bounding  box.center),scale=.8]
    \path (4,5);
    \foreach\x/\s/\t in {1/{v_1}/{z_1},2/{\cdots}/{\cdots},3/{v_j}/{z_{j+1}},4/{\!v_{j+1}}/{z_j},5/{v_w}/{z_w}} {
    \draw[very thick] (\x,1) -- (\x,4);
    \draw[dotted] (\x,.5) node[below] {$\t$} -- (\x,1);
    \draw[gbline] (\x,4) -- node[right] {$\ss\s$} (\x,4.5);
    }
    \foreach\y/\s in {1/{y_v},2/,3/{\vdots},4/{y_1}} {
    \draw (1,\y) -- (5,\y);
    \draw[dotted] (5,\y) -- (5.5,\y);
    \draw[gline] (.5,\y) node[left] {$\s$} -- (1,\y);
    }
    \end{tikzpicture}
    \qquad \begin{tikzpicture}[baseline=(current  bounding  box.center),scale=.5] \draw[very thick] (-1,-1) -- (1,1) (-1,1) -- (1,-1); \end{tikzpicture}= P D R(z_j-z_{j+1}) D^{-1}
    \]
    where $D$ is a diagonal matrix that cares of the prefactors occurring in Theorem~\ref{thm:tildeW}, namely $D\ket{v_j,v_{j+1}}=(-1)^{v_jv_{j+1}}{\ell_j!\choose (\ell_j-v_j)!}{\ell_{j+1}!\choose(\ell_{j+1}-v_{j+1})!}\ket{v_j,v_{j+1}}$, and $P\ket{v_j,v_{j+1}}=\ket{v_{j+1},v_j}$ is some conventional reordering of the tensor product.
    \yaping{I do not know how to get the sign $(-1)^{v_j+v_{j+1}+v_jv_{j+1}}$ either...}
    We note that this equation is automatically satisfied on condition that yet another version of the Yang--Baxter equation holds, namely
    \begin{equation}\label{eq:YAYBE}
    \begin{tikzpicture}[baseline=(current  bounding  box.center),scale=.6,rotate=-90]
\draw[rounded corners=4mm,very thick] (-1,-1)  -- (1,1) -- (3,1);
\draw[rounded corners=4mm,very thick] (-1,1)  -- (1,-1) -- (3,-1);
\draw[] (2,-2) -- (2,2);
\end{tikzpicture}
=
\begin{tikzpicture}[baseline=(current  bounding  box.center),scale=.6,rotate=-90]
\draw[very thick,rounded corners=4mm] (-1,-1)  -- (1,-1) -- (3,1);
\draw[very thick,rounded corners=4mm] (-1,1)  -- (1,1) -- (3,-1);
\draw[] (0,-2) -- (0,2);
\end{tikzpicture}
    \end{equation}
    and the fact that $\bra{0,0}R(z) = \bra{0,0}$ (which typically follows from conservation of occupation numbers and normalisation).
    Equation \eqref{eq:YAYBE} is indeed known to be satisfied.
\end{remark}

\newcommand{\arxiv}[1]
{\texttt{\href{http://arxiv.org/abs/#1}{arXiv:#1}}}
\newcommand{\doi}[1]
{\texttt{\href{http://dx.doi.org/#1}{doi:#1}}}
\renewcommand{\MR}[1]
{\href{http://www.ams.org/mathscinet-getitem?mr=#1}{MR#1}}

\end{document}